\documentclass[11pt]{amsart}
\usepackage{amsmath}
\usepackage{bbding}
\usepackage{tikz}
\usepackage{amsmath,amssymb}
\usepackage[utf8]{inputenc}
\usepackage{CJKutf8}
\usepackage{mathrsfs}
\usepackage{url}

\theoremstyle{plain}
\newtheorem{thm}{Theorem}[section]
\newtheorem{lemma}[thm]{Lemma}

\newtheorem{corollary}[thm]{Corollary}

\theoremstyle{definition}
\newtheorem{remark}[thm]{Remark}
\newtheorem{definition}[thm]{Definition}

\newtheorem{example}[thm]{Example}
\numberwithin{equation}{section}

\begin{document}

\title[Uniformly Degenerate Parabolic Equations]
{Global Schauder Regularity and Convergence for Uniformly Degenerate Parabolic Equations}

\author[Han]{Qing Han}
\address{Department of Mathematics\\
University of Notre Dame\\
Notre Dame, IN 46556, USA} \email{qhan@nd.edu}
\author[Xie]{Jiongduo Xie}
\address{Beijing International Center for Mathematical Research\\
Peking University\\
Beijing, 100871, China}
\email{2001110018@stu.pku.edu.cn}

\begin{abstract}
In this paper, we study the global H\"older regularity of solutions to
uniformly degenerate parabolic equations. We also study the convergence of solutions as time goes to infinity
under extra assumptions on the characteristic exponents of the limit uniformly degenerate elliptic equations.
\end{abstract}



\thanks{The first author acknowledges the support of NSF Grant DMS-2305038
and the second author acknowledges the support of
National Key R\&D Program of China Grant 2020YFA0712800.
}

\maketitle


\section{Introduction}\label{Introduction}

Uniformly degenerate parabolic equations appear in many problems, 
including proper harmonic maps between hyperbolic spaces \cite{LiTam1991,LiTam1993,LiTam1993Ind}, 
nonlinear filtering of diffusion processes \cite{L2001,K2006,K2008}, 
the mean curvature flow in the hyperbolic space \cite{U2003,LX2012,ALZ2020}, 
and the Black-Scholes equation \cite{JT2006,EGG2012,EGG2014}. 
The linear version of these equations provide a basic model. 
The Sobolev space theory for the linear uniformly degenerate parabolic equations 
was established in \cite{K2007,DR2024}.
In this paper, we study the global H\"older regularity and convergence of solutions to
linear uniformly degenerate parabolic equations.

Let $\Omega$ be a bounded domain in $\mathbb{R}^{n}$ with a $C^{k,\alpha}$-boundary $\partial\Omega$, 
for some integer $k\geq 1$ and some constant $\alpha\in[0,1]$. 
Fix a constant $T>0$ and let $Q_{T}=\Omega\times(0,T]$. 
For any $t\in [0,T]$, set $\Omega_{t}=\Omega\times\{t\}$ and $S_{t}=\partial\Omega\times[0,t]$. 
We write the parabolic boundary of $Q_{T}$ as
$$
\partial_{p}Q_{T}=\Omega_{0}\cup S_{T}.
$$
Let $a_{i j}$, $b_{i}$, and $c$ be continuous functions in $\bar{Q}_{T}$, with $a_{i j}=a_{j i}$ and, 
for any $(x,t)\in\bar{Q}_{T}$ and any $\xi\in\mathbb{R}^{n}$,
\begin{equation}\label{1a}
\lambda|\xi|^{2}\leq a_{ij}(x,t)\xi_{i}\xi_{j}\leq \Lambda|\xi|^{2},
\end{equation}
for some positive constants $\lambda$ and $\Lambda$.

Let $\rho$ be a  $C^{k,\alpha}(\bar{\Omega})$-function. 
Then, $\rho$ is a {\it defining function} of $\Omega$ if $\rho>0$ in $\Omega$ and $\rho=0$ and $\nabla\rho\neq 0$ on $\partial\Omega$. 
We always require
$$
|\nabla\rho|=1\quad\text{on }\partial\Omega.
$$
Then, $\nabla\rho$ restricted to $\partial\Omega$ is the interior unit normal to $\partial\Omega$.

We consider the parabolic operator $L$ given by
\begin{equation}\label{1b}
Lu=\rho^{2} a_{ij}\partial_{ij}u+\rho b_{i}\partial_{i}u+cu-\partial_{t}u\quad\text{in }Q_{T}.
\end{equation}
We note that the operator $L$ is not uniformly parabolic in $\bar{Q}_T$ and, in fact, is
degenerate along the lateral boundary $S_T$, due to the presence of the factor $\rho^2$ in the
second-order terms. The operator $L$ is called to be {\it uniformly degenerate parabolic}.

In this paper, we will establish global Schauder regularity and some convergence results for the initial-boundary value problem
associated with \eqref{1b}.
Specifically, let $f$ be a continuous function in $\bar{Q}_{T}$ and consider
\begin{equation}\label{1c}
Lu=f\quad\text{in }Q_{T}.
\end{equation}
The initial-boundary value is given by
\begin{align}
u(\cdot,0)&=\phi\quad\text{on }\Omega,\label{1d}\\
u&=h\quad \text{on }S_{T},\label{1e}
\end{align}
where $\phi\in C(\bar{\Omega})$ and
$$
h(x,t)=\phi(x)\exp\left\{\int^{t}_{0}c(x,s)ds\right\}-\int^{t}_{0}\exp\left\{\int^{t}_{s}c(x,\tau)d\tau\right\}f(x,s) d s.
$$
Note that the boundary value is not arbitrarily prescribed due to the degeneracy along the lateral boundary of $L$. In fact,
$h$ and $\phi$ satisfy
\begin{align}\label{1f}\begin{split}
\partial_{t}h-ch+f&=0\quad\text{on } S_{T},\\
h(\cdot,0)&=\phi\quad\text{on }\partial\Omega.
\end{split}
\end{align}
We will introduce various function spaces later.

\subsection{Notations}\label{Notations}
Before we introduce various function spaces for our study, we first examine the equation \eqref{1c}. 
To concentrate on the relation between $u$ and $Lu$, we assume that all coefficients and initial values are sufficiently regular. 
It is obvious that $Lu$ in \eqref{1b} is a linear combination of $u, \rho D_xu, \rho^2D_x^2u$, and $\partial_tu$. 
Then, $Lu\in C^{\alpha}(\bar{Q}_{T})$ if $u, \rho D_xu, \rho^2D_x^2u, \partial_tu\in C^{\alpha}(\bar{Q}_{T})$. 
To this end, we need, in particular, $\rho D_xu, \rho^2D_x^2u\in C^{\alpha}(\bar{Q}_{T})$, 
not the stronger condition $D_xu, D_x^2u\in C^{\alpha}(\bar{Q}_{T})$. 
This hints that we will obtain the regularity of $\rho D_xu$ and $\rho^2D_x^2u$, instead of $D_xu$ and $D_x^2u$. 
A fundamental regularity result to be established in this paper asserts 
$u, \rho D_xu, \rho^2D_x^2u, \partial_tu \in C^{\alpha}(\bar{Q}_{T})$
if $Lu\in  C^{\alpha}(\bar{Q}_{T})$. This leads to the $C^{2+\alpha}$-norm to be defined in \eqref{eq-def-2+alpha-norm}.
In subsequent results, we will demonstrate that such a pattern of regularity can be generalized to arbitrary derivatives
by an induction process. 
The $C^{k, 2+\alpha}$-norm as in \eqref{eq-def-k-2+alpha-norm} is introduced for this process. 
It is a sum of $C^\alpha(\bar Q_T)$-norms of weighted derivatives 
$D^{\beta}_{x}\partial^{i}_{t}u$ for $|\beta|+2i\leq k+2$.

Let $D$ be a bounded domain in $\mathbb{R}^{n}$ with a defining function $\rho$. Let $\alpha\in(0,1)$ be a constant. 
For a function $u:D\rightarrow \mathbb{R}$, we define its H\"{o}lder seminorm by
$$
[u]_{C^{\alpha}(\bar{D})}=\sup_{x,y\in D,x\neq y}\frac{|u(x)-u(y)|}{|x-y|^{\alpha}}.
$$
We also define the H\"{o}lder norm by
$$
\|u\|_{C^{\alpha}(\bar{D})}=\sup_{D}|u|+[u]_{C^{\alpha}(\bar{D})}.
$$
For a nonnegative integer $k$ and a function $u:D\rightarrow \mathbb{R}$ with continuous derivatives up to order $k$, we define
$$
\|u\|_{C^{k,\alpha}(\bar{D})}=\sum_{|\beta|\leq k}\sup_{D}|D^{\beta}u|+\sum_{|\beta|=k}[D^{\beta}u]_{C^{\alpha}(\bar{D})}.
$$
Due to the degeneracy of the problem, for a function $u:D\rightarrow \mathbb{R}$ with continuous derivatives $Du$ and $D^{2}u$, we define
$$
\|u\|_{C^{2+\alpha}(\bar{D})}=\|u\|_{C^{\alpha}(\bar{D})}+\|\rho Du\|_{C^{\alpha}(\bar{D})}+\|\rho^{2} D^2 u\|_{C^{\alpha}(\bar{D})},
$$
and
\begin{align*}
C^{2+\alpha}(\bar{D})&=\{u\in C^{2}(D):\|u\|_{C^{2+\alpha}(\bar{D})}<\infty,\\
&\qquad \rho Du|_{\partial D}=0,\text{ and }\rho^{2} D^{2}u|_{\partial D}=0\}.
\end{align*}
For a positive integer $k$ and a function $u:D\rightarrow \mathbb{R}$ with continuous derivatives up to order $k+2$, we define
$$
\|u\|_{C^{k,2+\alpha}(\bar{D})}=\|u\|_{C^{k-1}(\bar{D})}+\sum_{|\beta|=k}\|D^{\beta}u\|_{C^{2+\alpha}(\bar{D})},
$$
and
\begin{align*}
C^{k,2+\alpha}(\bar{D})&=\{u\in C^{k+2}(D):\|u\|_{C^{k,2+\alpha}(\bar{D})}<\infty,\rho D^{k+1}u|_{\partial D}=0,\\
&\qquad\text{ and }\rho^{2} D^{k+2}u|_{\partial D}=0\}.
\end{align*}
We also write $C^{0,2+\alpha}(\bar{D})=C^{2+\alpha}(\bar{D})$. In fact, $\rho^i D^{k+i}u=0$ on $\partial D$, $i=1,2$, 
for $u$ satisfying $\|u\|_{C^{k,2+\alpha}(\bar{D})}<\infty$.

Write $X=(x,t)$ in $\mathbb{R}^{n}\times[0,\infty)$. 
For any $X_{1}=(x_{1},t_{1}), X_{2}=(x_{2},t_{2})\in\mathbb{R}^{n}\times[0,\infty) $, we define the parabolic distance by
$$
s(X_{1},X_{2})=\max\{|x_{1}-x_{2}|,|t_{1}-t_{2}|^{\frac{1}{2}}\}.
$$
Fix a constant $T>0$ and set 
$$Q_{T}=D\times(0,T] \quad\text{and}\quad S_T=\partial D\times [0,T].$$ 
Let $\alpha\in(0,1)$ be a constant. For a function $u:Q_{T}\rightarrow \mathbb{R}$, we define its H\"{o}lder seminorms by
$$
[u]_{C^{\alpha}(\bar{Q}_{T})}=\sup_{X,Y\in Q_{T},X\neq Y}\frac{|u(X)-u(Y)|}{s(X,Y)^{\alpha}},
$$
and
$$
[u]_{C^{\alpha}_{x}(\bar{Q}_{T})}=\sup_{t\in(0,T]}[u(\cdot,t)]_{C^{\alpha}(\bar{D})},\quad [u]_{C^{\alpha}_{t}(\bar{Q}_{T})}=\sup_{x\in D}[u(x,\cdot)]_{C^{\alpha}([0,T])}.
$$
We also define the H\"{o}lder norm by
$$
\|u\|_{C^{\alpha}(\bar{Q}_{T})}=\sup_{Q_{T}}|u|+[u]_{C^{\alpha}(\bar{Q}_{T})}.
$$
Similarly, for a function $u:S_{T}\rightarrow\mathbb{R}$, we define
$$
\|u\|_{C^{\alpha}(S_{T})}=\sup_{S_{T}}|u|+[u]_{C^{\alpha}(S_{T})},
$$
where
$$
[u]_{C^{\alpha}(S_{T})}=\sup_{X,Y\in S_{T},X\neq Y}\frac{|u(X)-u(Y)|}{s(X,Y)^{\alpha}}.
$$

For a nonnegative integer $k$ and a function $u:Q_{T}\rightarrow \mathbb{R}$ with continuous derivatives 
$D^{\beta}_{x}\partial^{i}_{t}u$ $(|\beta|+2i\leq k)$, we define 
$$
\|u\|_{C^{k}(\bar{Q}_{T})}=\sum_{|\beta|+2i\leq k}\sup_{Q_{T}}|D^{\beta}_{x}\partial^{i}_{t}u|,
$$
and
$$
\|u\|_{C^{k,\alpha}(\bar{Q}_{T})}=\|u\|_{C^{k}(\bar{Q}_{T})}+\sum_{|\beta|+2i=k}[D^{\beta}_{x}\partial^{i}_{t}u]_{C^{\alpha}(\bar{Q}_{T})}.
$$
Due to the degeneracy of the problem, for a function $u:Q_{T}\rightarrow \mathbb{R}$ with continuous derivatives $D_{x}u$, $D^{2}_{x}u$, 
and $\partial_{t}u$, we define
\begin{align}\label{eq-def-2+alpha-norm}\begin{split}
\|u\|_{C^{2+\alpha}(\bar{Q}_{T})}&=\|u\|_{C^{\alpha}(\bar{Q}_{T})}+\|\rho D_{x}u\|_{C^{\alpha}(\bar{Q}_{T})}\\
&\qquad+\|\rho^{2} D^{2}_{x}u\|_{C^{\alpha}(\bar{Q}_{T})}+\|\partial_{t}u\|_{C^{\alpha}(\bar{Q}_{T})},
\end{split}\end{align}
and
\begin{align*}
C^{2+\alpha}(\bar{Q}_{T})&=\{u\in C^{2}(Q_{T}):\|u\|_{C^{2+\alpha}(\bar{Q}_{T})}<\infty,\\
&\qquad\rho D_{x}u|_{S_{T}}=0, \text{ and }\rho^{2} D^{2}_{x}u|_{S_{T}}=0\}.
\end{align*}
For a positive integer $k$ and a function $u:Q_{T}\rightarrow \mathbb{R}$ with continuous derivatives 
$D^{\beta}_{x}\partial^{i}_{t}u$ $(|\beta|+2i\leq k+2)$, we define
\begin{align}\label{eq-def-k-2+alpha-norm}\begin{split}
\|u\|_{C^{k,2+\alpha}(\bar{Q}_{T})}&=\sum_{|\beta|\leq k}\|D^{\beta}_{x}u\|_{C^{2+\alpha}(\bar{Q}_{T})}\\
&\qquad +\sum_{|\beta|+2i\leq k+2,i\geq 2}\|D^{\beta}_{x}\partial^{i}_{t}u\|_{C^{\alpha}(\bar{Q}_{T})},
\end{split}\end{align}
and
\begin{align*}
C^{k,2+\alpha}(\bar{Q}_{T})&=\{u\in C^{k+2}(Q_{T}):\|u\|_{C^{k,2+\alpha}(\bar{Q}_{T})}<\infty,\rho D_{x}D^{k}_{x}u|_{S_{T}}=0,\\
&\qquad\text{ and }\rho^{2} D^{2}_{x}D^{k}_{x}u|_{S_{T}}=0\}.
\end{align*}
We also write $C^{0,2+\alpha}(\bar{Q}_{T})=C^{2+\alpha}(\bar{Q}_{T})$. In fact, $\rho^i D^{k+i}_x u=0$ on $S_T$, $i=1,2$, 
for $u$ satisfying $\|u\|_{C^{k,2+\alpha}(\bar{Q}_T)}<\infty$.

The $C^{k, 2+\alpha}$-norm as in \eqref{eq-def-k-2+alpha-norm} is introduced for the induction process to be employed 
in the study of higher regularity. 
It is a sum of $C^\alpha(\bar Q_T)$-norms of weighted derivatives 
$D^{\beta}_{x}\partial^{i}_{t}u$ for $|\beta|+2i\leq k+2$. 
In fact, only $D^{\beta}_{x}u$ for $|\beta|=k+1, k+2$ has a weight in the forms 
$\rho D^{\beta}_{x}u$ for $|\beta|=k+1$ and $\rho^2 D^{\beta}_{x}u$ for $|\beta|=k+2$. 
All other derivatives have no weight, in particular, those derivatives involving the $t$-variable. 
Equivalently, we can define 
\begin{align*}
\|u\|_{C^{k,2+\alpha}(\bar{Q}_{T})}&=\sum_{|\beta|\leq k}\|D^{\beta}_{x}u\|_{C^{\alpha}(\bar{Q}_{T})}
+\sum_{|\beta|=k}\sum_{j=1}^{2}\|\rho^j D^j_xD^{\beta}_{x}u\|_{C^{\alpha}(\bar{Q}_{T})}\\
&\qquad +\sum_{|\beta|+2i\leq k+2,i\ge 1}\|D^{\beta}_{x}\partial^{i}_{t}u\|_{C^{\alpha}(\bar{Q}_{T})}.
\end{align*}
However, the form in \eqref{eq-def-k-2+alpha-norm} is more suitable for the aforementioned induction process.

We add an asterisk to the subscript of the notations 
in the above definitions to define some partial H\"{o}lder norms. 
For some constant $\alpha\in(0,1)$, we define
\begin{align*}
\|u\|_{C^{\alpha}_{\ast}(\bar{Q}_{T})}&=\sup_{Q_{T}}|u|+[u]_{C^{\alpha}_{x}(\bar{Q}_{T})},\\
\|u\|_{C^{k,\alpha}_{\ast}(\bar{Q}_{T})}&=\sum_{|\beta|\leq k}\|D^{\beta}_{x}u\|_{C^{\alpha}(\bar{Q}_{T})}
+\sum_{|\beta|+2i\leq k,i\geq 1}\|D^{\beta}_{x}\partial^i_t u\|_{C^{\alpha}_\ast(\bar{Q}_{T})}\quad\text{for }k\geq 2,
\end{align*} 
and
\begin{align*}
\|u\|_{C^{k,2+\alpha}_{\ast}(\bar{Q}_{T})}&=\sum_{|\beta|\leq k}\|D^{\beta}_{x}u\|_{C^{\alpha}(\bar{Q}_{T})}+
\sum_{|\beta|=k}\sum_{j=1}^{2}\|\rho^j D^j_xD^{\beta}_{x}u\|_{C^{\alpha}(\bar{Q}_{T})}\\
&\qquad+\sum_{|\beta|+2i\leq k+2,i\geq 1}\|D^{\beta}_{x}\partial^{i}_{t}u\|_{C^{\alpha}_{\ast}(\bar{Q}_{T})}\quad\text{for }k\geq 0.
\end{align*}
We also write $C^{0,2+\alpha}_{\ast}(\bar{Q}_T)=C^{2+\alpha}_{\ast}(\bar{Q}_T)$.

\smallskip
We end this subsection by two remarks. 

First, the classical Schauder theory for uniformly parabolic equations in non-divergence form asserts 
that if the coefficients and the nonhomogeneous terms are H\"{o}lder continuous with respect to all variables, then the same holds for
the second spatial derivatives and the first time derivative of the solutions. 
There are two types of such results, one is in the interior of the domain 
and the other is near the parabolic boundary under appropriate boundary conditions.
Refer to \cite{F1964,L1968,L1996}.

Second, the intermediate Schauder theory (\cite{B1969,K198081,L1992,L2000}), 
for parabolic problems in bounded domains where the coefficients and the nonhomogeneous terms 
are H\"older continuous with respect to the space variables, measurable in time, 
will be an important tool for us. 
First, we need the Schauder theory interior in space and global in time 
when we study the regularity of solutions of \eqref{1c}-\eqref{1e} near the lateral boundary. 
Since we will discuss the regularity in space and time at different scales near the lateral boundary, 
by replacing the classical Schauder theory with the intermediate Schauder theory, 
we can obtain optimal boundary regularity. Second, we will study the convergence of solutions of \eqref{1c}-\eqref{1e} 
without any regularity assumptions on time at infinity, which is based on the interior intermediate Schauder theory.


\subsection{Main Results}\label{Main Results}
Before we state our main results, we first present an example.
In this example, we demonstrate that solutions of uniformly degenerate parabolic equations do not
have better boundary regularity than the initial values when $t>0$ 
even if the domain $\Omega$ has a smooth boundary and the coefficients, nonhomogeneous terms, and boundary values are smooth enough. 
This is sharply different from the solutions of uniformly parabolic equations.

\begin{example}\label{example-regularity}
Let $\Omega$ be a bounded domain in $\mathbb R^n$
with a $C^\infty$-boundary $\partial\Omega$ and $\rho$ be a smooth defining function.
Consider constants $a, b, c$, with $a>0$, and set
\begin{align}\label{eq-ch2-def-operator-example}L=a\rho^2\Delta+b\rho\nabla\rho\cdot\nabla+c-\partial_t.\end{align}
Here, $\nabla$ and $\Delta$ are with respect to $x\in\mathbb R^n$.
In the following, we study the initial-boundary value problem for
\begin{equation}\label{eq-ch2-basic-equation}
Lu=f\quad\text{in }\Omega\times (0,\infty).\end{equation}

For convenience, set, for any constant $\mu>0$,
$$P(\mu)=a\mu(\mu-1)+b\mu +c.$$
Take any function $\psi\in C^\infty(\bar\Omega)$ and any constants $\tau$ and $\mu$.
A straightforward calculation yields
\begin{align*} L(e^{-\tau t}\psi\rho^\mu)
&=e^{-\tau t}\rho^\mu\big\{\big[P(\mu)+\tau\big]\psi+\mu\big[a(\mu-1)+b\big](|\nabla \rho|^2-1)\psi\\
&\qquad+\big[a\mu\psi\Delta\rho+(2a\mu+b)\nabla\psi\cdot\nabla\rho\big]\rho+a\Delta\psi\rho^{2}\big\}.
\end{align*}
All functions in the right-hand side are smooth except $\rho^\mu$ if $\mu$ is not an integer.

Consider nonnegative integers $k$ and $m$ and a constant $\alpha\in (0,1)$, and set $s=k+\alpha$.
Then, $s>0$ is a positive constant which is not an integer.
For functions $\psi_0, \psi_1, \cdots, \psi_m\in C^\infty(\bar\Omega)$ to be determined, consider
\begin{align}\label{eq-ch2-example-expression-u-t-sum}
u=e^{-\tau t}\rho^s\sum_{i=0}^m\psi_i\rho^{i}=e^{-\tau t}\sum_{i=0}^m\psi_i\rho^{s+i}.\end{align}
Then, $u$ satisfies \eqref{eq-ch2-basic-equation} with $f$ given by
\begin{align}\label{eq-ch2-example-expression-f-t-sum}\begin{split}
f&=e^{-\tau t}\rho^s\sum_{i=0}^m\big\{\big[P(s+i)+\tau\big]\psi_i\rho^i\\
&\qquad+(s+i)\big[a(s+i-1)+b\big](|\nabla \rho|^2-1)\psi_i\rho^i\\
&\qquad+\big[a(s+i)\psi_i\Delta\rho+(2a(s+i)+b)\nabla\psi_i\cdot\nabla\rho\big]\rho^{i+1}\\
&\qquad+a\Delta\psi_i\rho^{i+2}\big\}.
\end{split}\end{align}
All functions in the right-hand side are smooth except $\rho^s$.
We note that the second term in the parenthesis has an additional factor $\rho$ due to $|\nabla \rho|^2=1$ on $\partial\Omega$.

In the following, we fix $\psi_0$ arbitrarily. We now choose
\begin{equation}\label{eq-choice-lambda}\tau=-P(s).\end{equation}
Then, the first term in the parenthesis corresponding to $i=0$ in \eqref{eq-ch2-example-expression-f-t-sum} is zero.
We next assume, for $i=1, \cdots, m$,
\begin{equation}\label{eq-requirement-lambda}P(s+i)+\tau\neq 0.\end{equation}
Inductively for $i=1, \cdots, m$, we take
\begin{align*}\psi_{i}&=\frac{1}{P(s+i)+\tau}\big\{
(s+i-1)\big[a(s+i-2)+b\big](|\nabla \rho|^2-1)\psi_{i-1}\rho^{-1}\\
&\qquad+\big[a(s+i-1)\psi_{i-1}\Delta\rho+(2a(s+i-1)+b)\nabla\psi_{i-1}\cdot\nabla\rho\big]\\
&\qquad+a\Delta\psi_{i-1}\rho\big\}.
\end{align*}
In other words, we use the first term in the parenthesis corresponding to $i=1, \cdots, m$
in \eqref{eq-ch2-example-expression-f-t-sum} to cancel the rest of the terms corresponding to $i-1$.
Then,
\begin{align}\label{eq-ch2-example-expression-f-t-sum-v2}\begin{split}
f&=e^{-\tau t}\rho^{s+m+1}\big\{
(s+m)\big[a(s+m-1)+b\big](|\nabla \rho|^2-1)\psi_m\rho^{-1}\\
&\qquad+\big[a(s+m)\psi_m\Delta\rho+(2a(s+m)+b)\nabla\psi_m\cdot\nabla\rho\big]+a\Delta\psi_m\rho\big\}.
\end{split}\end{align}
Hence for any $t\ge 0$,
$f(\cdot, t)\in C^{k+m+1, \alpha}(\bar\Omega)$ and
$u(\cdot, t)\in C^{k,\alpha}(\bar\Omega)$, but $u(\cdot, t)\notin C^{k,\beta}(\bar\Omega)$,
for any $\beta\in (\alpha,1)$. The boundary value is given by
$$u=0\quad\text{on }\partial\Omega\times (0,\infty).$$
By \eqref{eq-choice-lambda}, we have, for any $i=1, \cdots, m$,
$$P(s+i)+\tau=P(s+i)-P(s)=ai(2s+i-1)+bi.$$
For \eqref{eq-requirement-lambda}, we require, for any $i=1, \cdots, m$,
$$a(2s+i-1)+b\neq 0.$$

If $\tau>0$, we have $u(\cdot,t)$ converges to $0$ in $C^{k,\alpha}(\bar\Omega)$ but not in $C^{k,\beta}(\bar\Omega)$, 
for any $\beta\in (\alpha,1)$, as $t\rightarrow\infty$.

We now examine a special case given by $P(s)=0$. Then, \eqref{eq-choice-lambda} yields $\tau=0$
and consequently the factor $e^{-\tau t}$ disappears from the expressions $u$ and $f$. 
In this case, we have stationary solutions.
\end{example}

We point out that $\rho$ is a given defining function of $\Omega$ and fixed throughout the paper. In the following, we will not specify how estimate constants depend on $\rho$.

We now state main results in this paper.
The first main result is the following existence and uniqueness of solutions 
of the initial-boundary value problem \eqref{1c}-\eqref{1e}.

\begin{thm}\label{1A}
Let $\Omega$ be a bounded domain in $\mathbb{R}^{n}$ with a $C^{1}$-boundary $\partial\Omega$ 
and $\rho$ be a $C^1(\bar{\Omega})\cap C^{2}(\Omega)$-defining function with $\rho\nabla^2\rho\in C(\bar{\Omega})$ 
and $\rho\nabla^2\rho=0$ on $\partial\Omega$. 
For some constant $T>0$, let $Q_{T}=\Omega\times(0,T]$. For some constant $\alpha\in(0,1)$, 
assume $a_{ij},b_{i},c\in C^{\alpha}(\bar{Q}_{T})$ with \eqref{1a}. 
Then, for any $f\in C^{\alpha}(\bar{Q}_{T})$ and $\phi\in C^{2+\alpha}(\bar{\Omega})$, 
the initial-boundary value problem \eqref{1c}-\eqref{1e} admits a unique solution $u\in C^{2+\alpha}(\bar{Q}_{T})$.
\end{thm}

The second main result concerns the higher regularity of solutions up to the boundary.

\begin{thm}\label{1B}
For some integer $k\geq 0$ and some constant $\alpha\in(0,1)$, 
let $\Omega$ be a bounded domain in $\mathbb{R}^{n}$ with a $C^{k+1,\alpha}$-boundary $\partial\Omega$ 
and $\rho$ be a $C^{k+1,\alpha}(\bar{\Omega})\cap C^{k+2,\alpha}(\Omega)$-defining function 
with $\rho\nabla^{k+2}\rho\in C^{\alpha}(\bar\Omega)$.  
For some constant $T>0$, let $Q_{T}=\Omega\times(0,T]$. 
Assume $a_{ij},b_{i},c\in C^{k,\alpha}(\bar{Q}_{T})$ with \eqref{1a}. 
Suppose that, for some $f\in C^{k,\alpha}(\bar{Q}_{T})$ and $\phi\in C^{k,2+\alpha}(\bar{\Omega})$, 
the initial-boundary value problem \eqref{1c}-\eqref{1e} 
admits a solution $u\in C^{2}(Q_{T})\cap C(\bar{Q}_{T})$. 
Then, $u\in C^{k,2+\alpha}(\bar{Q}_{T})$ and
$$
\|u\|_{C^{k,2+\alpha}(\bar{Q}_{T})}\leq C\big\{\|\phi\|_{C^{k,2+\alpha}(\bar{\Omega})}+\|f\|_{C^{k,\alpha}(\bar{Q}_{T})}\big\},
$$
where $C$ is a positive constant depending only on $n,$ $T,$ $\lambda,$ $k,$ $\alpha,$ $\Omega,$ $\rho$, 
and the $C^{k,\alpha}$-norms of $a_{ij},$ $b_{i},$ and $c$ in $\bar{Q}_T$.
\end{thm}


\begin{corollary}\label{1C}
Let $\Omega$ be a bounded domain in $\mathbb{R}^{n}$ with a $C^{\infty}$-boundary $\partial\Omega$ 
and a $C^{\infty}(\bar{\Omega})$-defining function $\rho$.  
For some constant $T>0$, let $Q_{T}=\Omega\times(0,T]$. 
Assume $a_{ij},b_{i},c\in C^{\infty}(\bar{Q}_{T})$ with \eqref{1a}. 
Suppose that, for some $f\in C^{\infty}(\bar{Q}_{T})$ and $\phi\in C^{\infty}(\bar{\Omega})$, 
the initial-boundary value problem \eqref{1c}-\eqref{1e} admits a solution $u\in C^{2}(Q_{T})\cap C(\bar{Q}_{T})$. 
Then, $u\in C^{\infty}(\bar{Q}_{T})$.
\end{corollary}

Next, we study the convergence of solutions of the initial-boundary value problem \eqref{1c}-\eqref{1e} 
under appropriate assumptions. 
Set 
$$Q=\Omega\times(0,\infty)\quad\text{and}\quad S=\partial\Omega\times[0,\infty),$$ 
and write 
$$\bar{Q}=\bar{\Omega}\times [0,\infty).$$ 
Let $a_{i j}$, $b_{i}$, and $c$ be bounded and continuous functions in $\bar{Q}$, with $a_{i j}=a_{j i}$ and, 
for any $(x,t)\in\bar{Q}$ and any $\xi\in\mathbb{R}^{n}$,
\begin{equation}\label{1g}
\lambda|\xi|^{2}\leq a_{ij}(x,t)\xi_{i}\xi_{j}\leq \Lambda|\xi|^{2},
\end{equation}
for some positive constants $\lambda$ and $\Lambda$. We consider the parabolic operator $L$ given by
\begin{equation}\label{1h}
Lu=\rho^{2} a_{ij}\partial_{ij}u+\rho b_{i}\partial_{i}u+cu-\partial_{t}u\quad\text{in }Q,
\end{equation}
and its limit elliptic operator $L_{0}$ given by
$$
L_{0}v=\rho^{2}\bar{a}_{ij}\partial_{ij}v+\rho \bar{b}_{i}\partial_{i}v+\bar{c}v\quad\text{in }\Omega,
$$
where $\bar{a}_{i j},\bar{b}_{i},\bar{c}\in C(\bar{\Omega})$ are the limits of $a_{i j}(\cdot,t),b_{i}(\cdot,t),c(\cdot,t)$ 
in certain sense as $t\rightarrow\infty$, respectively, with $\bar{c}<0$ on $\partial\Omega$ 
and, for any $x\in\bar{\Omega}$ and any $\xi\in\mathbb{R}^{n}$,
\begin{equation}\label{1h'}
\lambda|\xi|^{2}\leq \bar{a}_{ij}(x)\xi_{i}\xi_{j}\leq \Lambda|\xi|^{2}.
\end{equation}
Let $f$ be a bounded continuous function in $\bar{Q}$ and $\bar{f}\in C(\bar{\Omega})$ be the limit of $f(\cdot,t)$ 
in certain sense as $t\rightarrow\infty$.

Let $u$ be a solution of the initial-boundary value problem
\begin{align}
Lu&=f\quad\text{in }Q,\label{1i}\\
u(\cdot,0)&=\phi\quad\text{on }\Omega,\label{1j}\\
u&=h\quad \text{on }S,\label{1k}
\end{align}
where $\phi\in C(\bar{\Omega})$ and
$$
h(x,t)=\phi(x)\exp\Big\{\int^{t}_{0}c(x,s)ds\Big\}-\int^{t}_{0}\exp\Big\{\int^{t}_{s}c(x,\tau)d\tau\Big\}f(x,s) d s.
$$
Note that $h$ and $\phi$ satisfy
\begin{align}\label{1l}\begin{split} 
\partial_{t}h-ch+f&=0\quad\text{on } S,\\
h(\cdot,0)&=\phi\quad\text{on }\partial\Omega.
\end{split}
\end{align}
Let $v$ be a solution of the Dirichlet problem
\begin{align}
L_{0}v&=\bar{f}\quad\text{in }\Omega,\label{1m}\\
v&=\frac{\bar{f}}{\bar{c}}\quad\text{on }\partial\Omega.\label{1n}
\end{align}
Note that the boundary value \eqref{1n} is not arbitrarily prescribed and is determined by the equation itself.

We first review the Schauder theory of the Dirichlet problem \eqref{1m}-\eqref{1n}, 
which will inspire our results on the convergence of solutions of the initial-boundary value problem \eqref{1c}-\eqref{1e}. 
Take an arbitrary $\mu\in\mathbb{R}$. We now define
$$
P(\mu)=\mu(\mu-1)\bar{a}_{ij}\nu_{i}\nu_{j}+\mu \bar{b}_{i}\nu_{i}+\bar{c}\quad\text{on }\partial\Omega,
$$
where $\nu=(\nu_{1},\cdots,\nu_{n})$ is the inner unit normal vector along $\partial\Omega$.
The polynomial $P(\mu)$ is referred to as the characteristic polynomial of $L_{0}$. 
The existence, uniqueness, and optimal regularity results of the Dirichlet problem \eqref{1m}-\eqref{1n} are as follows. 
(Refer to \cite{HanXie2024} for details.)

\begin{thm}[Existence and uniqueness]\label{1C'}
Let $\Omega$ be a bounded domain in $\mathbb R^n$ with a $C^{1}$-boundary $\partial\Omega$ 
and $\rho$ be a $C^{1}(\bar\Omega)\cap C^2(\Omega)$-defining function with $\rho\nabla^2\rho\in C(\bar{\Omega})$ 
and $\rho\nabla^2\rho=0$ on $\partial\Omega$. 
For some constant $\alpha\in (0,1)$, assume $\bar{a}_{ij}, \bar{b}_i, \bar{c}\in C^{\alpha}(\bar\Omega)$,
with \eqref{1h'}, $\bar{c}\le 0$ in $\Omega$,
and $\bar{c}<0$ and $P(\alpha)<0$ on $\partial\Omega$. 
Then, for any $\bar{f}\in C^\alpha(\bar\Omega)$, the Dirichlet problem \eqref{1m}-\eqref{1n} 
admits a unique solution $v\in C^{2+\alpha}(\bar\Omega)$.
\end{thm}

We need to point out that we proved the existence of solutions of the Dirichlet problem \eqref{1m}-\eqref{1n} in $C^{1,\alpha}$-domains in \cite{HanXie2024}. We can prove the similar existence result in $C^1$-domains without flattening the boundary. We carry out this process for the initial-boundary value problems for uniformly degenerate parabolic equations in Theorem \ref{1A}.

\begin{thm}[Optimal regularity]\label{1D'}
For some integer $k\geq 0$ and some constant $\alpha\in(0,1)$, let $\Omega$ be a bounded domain in $\mathbb{R}^{n}$ 
with a $C^{k+1,\alpha}$-boundary $\partial\Omega$ 
and $\rho$ be a $C^{k+1,\alpha}(\bar{\Omega})\cap C^{k+2,\alpha}(\Omega)$-defining function 
with $\rho\nabla^{k+2}\rho\in C^{\alpha}(\bar\Omega)$ and $\rho\nabla^{k+2}\rho=0$ on $\partial\Omega$.  
Assume $\bar{a}_{ij}, \bar{b}_i, \bar{c}\in C^{k,\alpha}(\bar\Omega)$,
with \eqref{1h'}, and $\bar{c}\le -c_0$ and $P(k+\alpha)\le -c_{k+\alpha}$ on $\partial\Omega$, 
for some positive constants $c_0$ and $c_{k+\alpha}$.
Suppose that, for some $\bar{f}\in C^{k,\alpha}(\bar\Omega)$, 
the Dirichlet problem \eqref{1m}-\eqref{1n} admits a solution $v\in C(\bar\Omega)\cap
C^{2}(\Omega)$. Then, $v\in C^{k,2+\alpha}(\bar\Omega)$ and
\begin{align*}
\|v\|_{C^{k,2+\alpha}(\bar\Omega)}\le C\big\{\|v\|_{L^\infty(\Omega)}+\|\bar{f}\|_{C^{k,\alpha}(\bar\Omega)}\big\},
\end{align*}
where $C$ is a positive constant depending only on $n$, $\lambda$,
$k$, $\alpha$, $c_0$, $c_{k+\alpha}$, $\Omega$, $\rho$,
and the $C^{k,\alpha}$-norms of $\bar{a}_{ij}$, $\bar{b}_i$, 
and $\bar{c}$ in $\bar\Omega$.
\end{thm}

Although formulated as a global result, Theorem \ref{1D'} actually holds locally near the boundary.

By the assumption $P(0)=\bar c<0$ on $\partial\Omega$,
the characteristic polynomial $P(\mu)$ has two real roots, a positive one and a negative one, 
and any $\mu$ with a negative value of $P(\mu)$
has to be between these two roots. For some integer $k\ge 0$ and constant $\alpha\in (0,1)$, 
the assumption $P(k+\alpha)<0$ on $\partial\Omega$ implies that $k+\alpha$ is less than the positive root. 
Theorem \ref{1D'} asserts a global regularity in $C^{k,\alpha}(\bar\Omega)$. 
In general, the order of global regularity cannot be higher than the positive root even if
the domain has a smooth boundary and the coefficients and the nonhomogeneous terms are smooth up to the boundary. 
See Example 2.1 in \cite{HanXie2024}. This is different from Theorem \ref{1B}, which asserts that the global regularity of the solutions 
can always be improved with the regularity of the boundary, the initial values, the coefficients, and the nonhomogeneous terms.

To characterize the regularity and convergence in our results, we introduce some definitions. 
For convenience, we extend the function spaces defined in Subsection \ref{Notations} to $Q$ and $S$. 
For some integer $k\geq 0$ and some constant $\alpha\in(0,1)$, we define
$$
C^{k,\alpha}(\bar{Q})=\{f:Q\rightarrow\mathbb{R}:f\in C^{k,\alpha}(\bar{Q}_{T})\text{ for any }T>0\}.
$$
Similarly, we can define $C^{k,2+\alpha}(\bar{Q})$, $C^{k,2+\alpha}_{\ast}(\bar{Q})$, $C^{\alpha}(S)$, 
$\cdots$.

Next, we introduce a type of convergence for these function spaces.

\begin{definition}\label{1D}
For some integer $k\geq 0$ and some constant $\alpha\in(0,1)$, let $f,g\in C^{k,\alpha}(\bar{Q})$. 
We call $f$ $C^{k,\alpha}$-converging to $g$ in $\bar{Q}$ at infinity, if
$$
\lim_{T\rightarrow\infty}\|f-g\|_{C^{k,\alpha}(\overline{{Q}_{T+1}\backslash {Q}_{T}})}=0.
$$
\end{definition}

It is easy to verify that the convergence is equivalent if we replace the norm $C^{k,\alpha}(\overline{{Q}_{T+1}\backslash {Q}_{T}})$ 
with $C^{k,\alpha}(\overline{{Q}_{T+\sigma}\backslash {Q}_{T}})$, for any $\sigma>0$. 
Similarly, we can define $C^{k,2+\alpha},C^{k,2+\alpha}_{\ast}$, $\cdots$-convergence 
by replacing $C^{k,\alpha}$ with $C^{k,2+\alpha},C^{k,2+\alpha}_{\ast}$, $\cdots$ in the above definitions, respectively. 
For $g=g(x)\in C^{k,\alpha}(\bar{\Omega})$, we emphasize that $f$ $C^{k,\alpha}$-converging to $g$ in $\bar{Q}$ at infinity 
is stronger than $f(\cdot,t)$ converging to $g$ in $C^{k,\alpha}(\bar{\Omega})$ as $t\rightarrow\infty$. 
For example, we consider 
$$f(x,t)=x^{5/2}\,\frac{\sin{(t+1)^2}}{t+1}\quad\text{in }B_1\times(0,\infty).$$ 
Then, $f(\cdot,t)$ converges to $0$ in $C^{2,1/2}(\bar{B}_1)$ as $t\rightarrow\infty$. 
Note that 
$$\partial_t f(x,t)= x^{5/2}\Big[2\cos{(t+1)^2}-\frac{\sin{(t+1)^2}}{(t+1)^2}\Big].$$ 
Hence, $f$ does not $C^{2,1/2}$-converge to $0$ in $\bar{B}_1\times[0,\infty)$ at infinity.   

For the next three results, we will always assume 
\begin{equation}\label{1o}
\limsup_{t\rightarrow\infty}\sup_{x\in\bar\Omega}c(x,t)\leq 0,\quad \bar{c}\leq 0\quad\text{in } \Omega,
\quad \text{and } \bar{c}<0\quad \text{on }\partial\Omega.
\end{equation}

By Theorem \ref{1D'}, we can expect convergence in some sense, with orders not exceeding the positive root 
of the characteristic polynomial of $L_{0}$. We have the following result concerning the convergence of $u(\cdot,t)$ to $v$ 
as $t\rightarrow\infty$.

\begin{thm}\label{1E}
For some integer $k\geq 0$ and some constant $\alpha\in(0,1)$, let $\Omega$ be a bounded domain in $\mathbb{R}^{n}$ 
with a $C^{k+1,\alpha}$-boundary $\partial\Omega$ and $\rho$ 
be a $C^{k+1,\alpha}(\bar{\Omega})\cap C^{k+2,\alpha}(\Omega)$-defining function with $\rho\nabla^{k+2}\rho\in C^{\alpha}(\bar\Omega)$ 
and $\rho\nabla^{k+2}\rho$ $=0$ on $\partial\Omega$. 
Assume that $a_{ij}, b_{i}, c, f\in C^{k,\alpha}(\bar{Q})$ and $\bar{a}_{ij}, \bar{b}_{i}, \bar{c}, \bar{f}\in C^{k,\alpha}(\bar{\Omega})$, 
with \eqref{1g}, \eqref{1o}, and $P(k+\alpha)<0$ on $\partial\Omega$ 
and that $a_{ij}(\cdot,t)$, $b_{i}(\cdot,t)$, $c(\cdot,t)$, $f(\cdot,t)$ converge to 
$\bar{a}_{ij},\bar{b}_{i},\bar{c},\bar{f}$ in $C^{k,\alpha}(\bar{\Omega})$ as $t\rightarrow\infty$, respectively. 
Suppose that, for some $\phi\in C^{k,2+\alpha}(\bar{\Omega})$, the initial-boundary value problem \eqref{1i}-\eqref{1k} 
admits a solution $u\in C^{2}(Q)\cap C(\bar{Q})$. 
Then, $u$ $C^{k,2+\alpha}_{\ast}$-converges to the unique solution $v\in C^{2}(\Omega)\cap C(\bar{\Omega})$ 
of the Dirichlet problem \eqref{1m}-\eqref{1n} in $\bar{Q}$ at infinity. 
In particular, $u(\cdot,t)$ converges to $v$ in $C^{k,2+\alpha}(\bar{\Omega})$ as $t\rightarrow \infty$.
\end{thm}

If we strengthen the convergence conditions of the coefficients of $L$ and $f$, 
we can obtain stronger convergence results for the solution.

\begin{thm}\label{1F}
For some integer $k\geq 0$ and some constant $\alpha\in(0,1)$, let $\Omega$ be a bounded domain in $\mathbb{R}^{n}$ 
with a $C^{k+1,\alpha}$-boundary $\partial\Omega$ and 
$\rho$ be a $C^{k+1,\alpha}(\bar{\Omega})\cap C^{k+2,\alpha}(\Omega)$-defining function 
with $\rho\nabla^{k+2}\rho\in C^{\alpha}(\bar\Omega)$ and $\rho\nabla^{k+2}\rho$ $=0$ on $\partial\Omega$. 
Assume $a_{ij},b_{i},c,f\in C^{k,\alpha}(\bar{Q})$ $C^{k,\alpha}$-converge to $\bar{a}_{ij}$, $\bar{b}_{i}$, $\bar{c}$, 
$\bar{f}\in C^{k,\alpha}(\bar{\Omega})$ in $\bar{Q}$ at infinity, respectively, with \eqref{1g}, \eqref{1o}, 
and $P(k+\alpha)<0$ on $\partial\Omega$. 
Suppose that, for some $\phi\in C^{k,2+\alpha}(\bar{\Omega})$, 
the initial-boundary value problem \eqref{1i}-\eqref{1k} admits a solution $u\in C^{2}(Q)\cap C(\bar{Q})$. 
Then, $u$ $C^{k,2+\alpha}$-converges to the unique solution $v\in C^{2}(\Omega)\cap C(\bar{\Omega})$ 
of the Dirichlet problem \eqref{1m}-\eqref{1n} in $\bar{Q}$ at infinity. In particular, $u(\cdot,t)$ 
converges to $v$ in $C^{k,2+\alpha}(\bar{\Omega})$ as $t\rightarrow \infty$.
\end{thm}

From Example \ref{example-regularity}, 
it is reasonable to assume that the initial values have higher regularity when considering higher-order convergence.
The proof of Theorem \ref{1F} is similar to that of Theorem \ref{1E}. 
We will only discuss the proof of Theorem \ref{1E} in details and indicate how to modify the proof to get Theorem \ref{1F}.

In particular, we can apply Theorem \ref{1F} when the coefficients of $L$ and $f$ are independent of $t$, i.e.,
\begin{align}\label{1p}\begin{split}
a_{ij}(x,t)\equiv \bar{a}_{ij}(x),\ b_{i}(x,t)\equiv \bar{b}_{i}(x),& \\
  c(x,t)\equiv \bar{c}(x),\ f(x,t)\equiv \bar{f}(x)&\quad\text{in }\bar{Q}.
\end{split}\end{align}

\begin{corollary}\label{1G}
For some integer $k\geq 0$ and some constant $\alpha\in(0,1)$, let $\Omega$ be a bounded domain in $\mathbb{R}^{n}$ 
with a $C^{k+1,\alpha}$-boundary $\partial\Omega$ and 
$\rho$ be a $C^{k+1,\alpha}(\bar{\Omega})\cap C^{k+2,\alpha}(\Omega)$-defining function 
with $\rho\nabla^{k+2}\rho\in C^{\alpha}(\bar\Omega)$ and $\rho\nabla^{k+2}\rho$ $=0$ on $\partial\Omega$. 
Assume \eqref{1p} and $\bar{a}_{ij},\bar{b}_{i},\bar{c},\bar{f}\in C^{k,\alpha}(\bar{\Omega})$, with \eqref{1g}, 
$\bar{c}\leq 0$ in $\Omega$, and $\bar{c}<0$ and $P(k+\alpha)<0$ on $\partial\Omega$. 
Suppose that, for some $\phi\in C^{k,2+\alpha}(\bar{\Omega})$, 
the initial-boundary value problem \eqref{1i}-\eqref{1k} admits a solution $u\in C^{2}(Q)\cap C(\bar{Q})$. 
Then, $u$ $C^{k,2+\alpha}$-converges to the unique solution $v\in C^{2}(\Omega)\cap C(\bar{\Omega})$ 
of the Dirichlet problem \eqref{1m}-\eqref{1n} in $\bar{Q}$ at infinity. 
In particular, $u(\cdot,t)$ converges to $v$ in $C^{k,2+\alpha}(\bar{\Omega})$ as $t\rightarrow \infty$.
\end{corollary}

The existence of $u$ in Theorems \ref{1E} and \ref{1F} and Corollary \ref{1G} is guaranteed by Theorem \ref{1A}.

The paper is organized as follows.
In Section \ref{sec-Existence}, we study the existence of solutions to the uniformly degenerate parabolic equations.
In Section \ref{sec-Higher-regularity}, we study the higher regularity of solutions near the lateral boundary.
In Section \ref{sec-Convergence}, we discuss the convergence of solutions as time goes to infinity.
In Section \ref{sec-Intermediate-Schauder}, we prove an intermediate Schauder regularity result that is needed in the paper.

\section{The Existence of Solutions}\label{sec-Existence}

In this section, we establish the existence of solutions of the initial-boundary value problem \eqref{1c}-\eqref{1e}.

First, we have the following maximum principle.

\begin{lemma}\label{2A}
Let $\Omega$ be a bounded domain in $\mathbb{R}^n$, with a $C^{1}$-boundary $\partial\Omega$ 
and a $C^{1}(\bar{\Omega})$-defining function $\rho$. Assume $a_{ij},b_{i},c\in C(\bar{Q}_{T})$ with \eqref{1a}. 
Suppose $u\in  C^{2}(Q_{T})\cap C(\bar{Q}_{T})$ satisfies $Lu\geq 0$ in $Q_{T}$. 
If $u \leq 0$ on $\partial_{p}Q_{T}$, then $u\leq 0$ in $Q_{T}$.
\end{lemma}

\begin{proof}
The proof is standard. We consider several cases.

First, we assume that $Lu>0$ and $c\leq 0$ in $Q_{T}$. Suppose that the conclusion is not valid. 
Then, there exists a point $X_{0}=(x_0,t_0)\in Q_{T}$ such that $u(X_{0})=\max_{\bar{Q}_{T}}u>0$. 
By $D_{x}u(X_{0})=0$, $\partial_{t}u(X_{0})\geq 0$, and $D^{2}_{x}u(X_{0})\leq 0$, we have
$$
Lu(X_{0})=\rho^{2}(x_{0}) a_{ij}(X_{0})\partial_{ij}u(X_{0})+c(X_{0})u(X_{0})-\partial_{t}u(X_{0})\leq 0.
$$
This yields a contradiction to $Lu>0$ in $Q_{T}$.

Next, we assume that $Lu\geq 0$ and $c\leq 0$ in $Q_{T}$. We consider $v=u-\varepsilon t$, for $\varepsilon>0$. 
Note that $Lv=Lu-c\varepsilon t+\varepsilon>0$ in $Q_{T}$ and $v \leq 0$ on $\partial_{p}Q_{T}$. 
Hence, $v\leq 0 $ in $Q_{T}$. By letting $\varepsilon\rightarrow 0$, we have $ u\leq 0 $ in $Q_{T}$.

Finally, for the general case, we consider $v=e^{-At}u$, for a nonnegative constant $A> \sup_{Q_{T}} c$. 
Note that $v$ satisfies
$$
(L-A)v=e^{-At}Lu\geq 0\quad\text{in }Q_{T},
$$
and $v \leq 0$ on $\partial_{p}Q_{T}$. Then, $ v\leq 0 $ in $Q_{T}$, and hence $ u\leq 0 $ in $Q_{T}$.
\end{proof}

We now derive an $L^{\infty}$-estimate of solutions of the initial-boundary value problem.

\begin{lemma}\label{2B}
Let $\Omega$ be a bounded domain in $\mathbb{R}^n$, 
with a $C^{1}$-boundary $\partial\Omega$ and a $C^{1}(\bar{\Omega})$-defining function $\rho$. 
Assume $a_{ij},b_{i},c\in C(\bar{Q}_{T})$, with \eqref{1a} and $c< c_{0}$ in $\bar{Q}_{T}$ for some nonnegative constant $c_0$. 
Suppose, for some $f\in C(\bar{Q}_{T})$ and $\phi\in C(\bar{\Omega})$, 
the  initial-boundary value problem \eqref{1c}-\eqref{1e} admits a solution $u\in  C^{2}(Q_{T})\cap C(\bar{Q}_{T})$. 
Then,
$$
\sup_{Q_{T}}|u|\leq Ce^{c_{0}T}\Big\{\sup_{\partial_{p}Q_{T}}|u|+\sup_{Q_{T}}|f|\Big\},
$$
where $C$ is a positive constant depending only on $c_{0}$, the diameter of $\Omega$, $\rho$, $\sup_{Q_T}c$, 
and the $L^{\infty}$-norm of $b_{i}$ in $Q_{T}$.
\end{lemma}

\begin{proof}
Set
$$
F = \sup_{Q_{T}}|f|,\quad \Phi=\sup_{\partial_{p}Q_{T}}|u|.
$$
Then, $L(\pm u) \geq -F$ in $Q_{T}$ and $\pm u \leq \Phi $ on $  \partial_{p}Q_{T}$. Without loss of generality, we assume
$$\Omega\subset \{0<x_{1}<r\},$$
for some constant $r>0$. For some constant $\mu>0$ to be chosen later, set
$$
v=e^{c_{0}t}[\Phi+r^{2}(e^{2\mu}-e^{\frac{\mu x_{1}}{r}})F].
$$
We note that $v \geq \Phi$ in $\bar{Q}_{T}$. Next, by a straightforward calculation and $c< c_{0}$ in $\bar{Q}_{T}$, we have
\begin{align*}
Lv&=e^{c_{0}t}\{-(\rho^{2}a_{11}\mu^{2}+r\rho b_{1}\mu)F e^{\frac{\mu x_{1}}{r}}+(c-c_{0})[\Phi+r^{2}(e^{2\mu}-e^{\frac{\mu x_{1}}{r}})F]\}\\
&\leq -e^{c_{0}t+\frac{\mu x_{1}}{r}}F[\rho^{2}a_{11}\mu^{2}+r\rho b_{1}\mu-(c-c_{0})r^{2}(e^{\mu(2-\frac{ x_{1}}{r})}-1)].
\end{align*}

We now claim, by choosing $\mu > 0$ appropriately,
\begin{equation}\label{2a}
\rho^{2}a_{11}\mu^{2}+r\rho b_{1}\mu-(c-c_{0})r^{2}(e^{\mu(2-\frac{x_{1}}{r})}-1)\geq 1\quad\text{in }Q_{T}.
\end{equation}
Then, $Lv \leq  -F $ in $Q_{T}$. Therefore, $L(\pm u) \geq L v $ in $Q_{T}$ and $\pm u \leq  v $ on $\partial_{p}Q_{T}$. 
By the maximum principle, we obtain $\pm u \leq v $ in $Q_{T}$, and hence, for any $(x,t)\in Q_{T}$,
$$
|u(x,t)|\leq e^{c_{0}t}[\Phi+r^{2}(e^{2\mu}-e^{\frac{\mu x_{1}}{r}})F].
$$
This yields the desired result.

We now prove \eqref{2a}. By $c< c_{0}$ in $\bar{Q}_{T}$, we have $c-c_{0}<-c_{1}$ in $\bar{Q}_{T}$ for some positive constant $c_{1}$. 
Hence,
$$
\rho^{2}a_{11}\mu^{2}+r\rho b_{1}\mu-(c-c_{0})r^{2}(e^{\mu(2-\frac{x_{1}}{r})}-1)\geq r\rho b_{1}\mu+c_{1}r^{2}(e^{\mu}-1)\geq 1,
$$
by choosing $\mu$ sufficient large. This finishes the proof of \eqref{2a}.
\end{proof}

For a bounded domain $\Omega\subset\mathbb R^n$ with a $C^{1}$-boundary $\partial\Omega$, 
let $d$ be the distance function in $\Omega$ to its boundary $\partial\Omega$ and $\rho$ be a $C^{1}(\bar\Omega)$-defining function. 
It is well known that 
\begin{equation}\label{defining-distance}
 C^{-1}d\leq \rho\leq C d\quad\text{in }\Omega,   
\end{equation}
where $C$ is a positive constant depending only on $\Omega$ and 
$\rho$.

We start to derive global H\"{o}lder estimates of solutions and their derivatives. 
We first prove a general result that an appropriate decay estimate near the lateral boundary $S_{T}$ 
and a H\"{o}lder estimate interior in space and global in time imply a global H\"{o}lder estimate.  
Recall that $\Omega_0=\Omega\times\{0\}$.

\begin{lemma}\label{2C}
Let $\Omega$ be a bounded domain in $\mathbb{R}^{n}$ with a $C^{1}$-boundary $\partial\Omega$. 
For some constants $A>0$ and $\alpha\in(0,1)$, suppose $w\in C^{\alpha}(Q_{T}\cup\Omega_{0})$ 
and $w_{0}\in C^{\alpha}(S_{T})$ are functions such that, 
for any $(x,t)\in Q_{T}\cup\Omega_{0}$ and $(x_{0},t)\in S_{T}$ with $d(x)=|x-x_{0}|$,
\begin{equation}\label{2b}
|w(x,t)-w_{0}(x_{0},t)|\leq A d^{\alpha}(x),
\end{equation}
and
\begin{equation}\label{2c}
[w]_{C^{\alpha}(B_{d(x)/2}(x)\times[0,T])}\leq A.
\end{equation}
Then, $w$ can be extended to a $C^{\alpha}(\bar{Q}_{T})$-function such that $w=w_{0}$ on $S_{T}$ and
$$
\|w\|_{C^{\alpha}(\bar{Q}_{T})}\leq C\big\{A+\|w_{0}\|_{ C^{\alpha}(S_{T})}\big\},
$$
where $C$ is a positive constant depending only on the diameter of $\Omega$.
\end{lemma}

\begin{proof}
First, \eqref{2b} implies
$$
\|w\|_{L^{\infty}(Q_{T})}\leq A\|d\|^{\alpha}_{L^{\infty}(\Omega)}+\|w_{0}\|_{L^{\infty}(S_{T})}.
$$
We note that $w$ is defined only in $Q_{T}\cup\Omega_{0}$. 
Define $w(x,t) = w_{0}(x,t)$ for any $(x,t)\in S_{T}$. 
Take any $X=(x,t),Y=(y,\tau)\in\bar{Q}_{T}$ with $X\neq Y$, and claim
$$
|w(X)-w(Y)|\leq (4 A+5[w_{0}]_{ C^{\alpha}(S_{T})})s^{\alpha}(X,Y).
$$
We consider three cases.

Case 1. If $X,Y\in S_{T}$, then $w(X) = w_{0}(X)$ and $w(Y) = w_{0}(Y)$. Hence,
$$
|w(X)-w(Y)| = |w_{0}(X)- w_{0}(Y)| \leq  [w_{0}]_{ C^{\alpha}(S_{T})}s^{\alpha}(X,Y).
$$

Case 2. If $X\in Q_{T}\cup\Omega_{0}$ and $Y\in S_{T}$, we take $X_{0}=(x_{0},t)\in S_{T}$  with $d(x) =|x -x_{0}|$. Then, by \eqref{2b},
$$
|w(X)-w(X_{0})|\leq Ad^{\alpha}(x)=A s^{\alpha}(X,X_{0}),
$$
and, by Case 1,
$$
|w(X_{0})-w(Y)| \leq  [w_{0}]_{ C^{\alpha}(S_{T})}s^{\alpha}(X_{0},Y).
$$
With $|x-x_{0}|\leq |x-y|$ and $|x_{0}-y|\leq |x_{0}-x|+|x-y|\leq 2|x-y|$, we obtain
$$
|w(X)-w(Y)|\leq (A+2[w_{0}]_{ C^{\alpha}(S_{T})})s^{\alpha}(X,Y).
$$

Case 3. If  $X,Y\in Q_{T}\cup\Omega_{0}$, without loss of generality, we assume $d(x)\geq d(y)$. If $|x-y|\leq d(x)/2$, then by \eqref{2c},
$$
|w(X)-w(Y)|\leq [w]_{C^{\alpha}(B_{d(x)/2}(x)\times[0,T])} s^{\alpha}(X,Y)\leq A s^{\alpha}(X,Y).
$$
If $|x-y|\geq d(x)/2$, take $X_{0}=(x_{0},t)\in S_{T}$ and $Y_{0}=(y_{0},\tau)\in S_{T}$ such that $d(x)=|x-x_{0}|$ and $d(y)=|y-y_{0}|$. Then
$$
\begin{aligned}
|x-x_{0}|&=d(x)\leq 2|x-y|,\\
|y-y_{0}|&=d(y)\leq 2|x-y|,
\end{aligned}
$$
and hence
$$
|x_{0}-y_{0}|\leq |x_{0}-x|+|x-y|+|y-y_{0}|\leq 5|x-y|.
$$
Arguing as in Case 2, we have
$$
\begin{aligned}
|w(X)-w(Y)|&\leq |w(X)-w(X_{0})|+|w(X_{0})-w(Y_{0})|+|w(Y_{0})-w(Y)|\\
&\leq Ad^{\alpha}(x)+[w_{0}]_{ C^{\alpha}(S_{T})}s^{\alpha}(X_{0},Y_{0})+Ad^{\alpha}(y)\\
&\leq (4 A+5[w_{0}]_{ C^{\alpha}(S_{T})})s^{\alpha}(X,Y).
\end{aligned}
$$
By combining Cases 1-3, we have the desired result.
\end{proof}

Next, we construct supersolutions near the lateral boundary.

\begin{lemma}\label{2D}
Let  $\mu>0$ and $K\geq 0$ be constants, 
$\Omega$ be a bounded domain in $\mathbb R^n$ with a $C^{1}$-boundary $\partial\Omega$, 
and $\rho$ be a $C^{1}(\bar\Omega)\cap C^2(\Omega)$-defining function with $\rho\nabla^2\rho\in L^{\infty}(\Omega)$. 
Assume $a_{ij},b_{i},c\in C(\bar{Q}_{T})$ with \eqref{1a}. 
Then, there exist positive constants $A$ and $C_{0}$, depending only on $\mu$, $K$, $\Omega$, $\rho$, 
and the $L^{\infty}$-norms of $a_{ij}$, $b_{i}$, and $c$ in $Q_{T}$, such that, for any $x_{0}\in \partial\Omega$ and any $(x,t)\in Q_T$,
$$
L[e^{At}(|x-x_{0}|^{\mu}+K\rho^{\mu})]\leq -C_{0}|x-x_{0}|^{\mu}.
$$
\end{lemma}

\begin{proof}
Without loss of generality, we assume $x_{0}=0$. Set
$$
w(x)=|x|^{\mu}+K \rho^{\mu}.
$$
Then,
$$
\partial_{i}w=\mu|x|^{\mu-2}x_{i}+K\mu \rho^{\mu-1}\rho_{i},
$$
and
\begin{align*}
\partial_{ij}w&=\mu(\mu-2)|x|^{\mu-4}x_{i}x_{j}+\mu|x|^{\mu-2}\delta_{ij}\\
&\qquad+K\mu(\mu-1)\rho^{\mu-2}\rho_{i}\rho_{j}+K\mu \rho^{\mu-1}\rho_{ij}.
\end{align*}
For some $A>0$ to be determined, a straightforward computation yields
$$
L(e^{At}w)=e^{At}(I_{1}+I_{2}),
$$
where
\begin{align*}
I_{1}&=\mu(\mu-2)|x|^{\mu-4}\rho^{2}a_{ij}x_{i}x_{j}+\mu|x|^{\mu-2}\rho^{2}a_{ii}\\
&\qquad +K\mu(\mu-1)\rho^{\mu}a_{ij}\rho_{i}\rho_{j}
+K\mu \rho^{\mu+1}a_{ij}\rho_{ij},
\end{align*}
and
\begin{align*}
I_{2}=\mu|x|^{\mu-2}\rho b_{i}x_{i}+K\mu \rho^{\mu}b_{i}\rho_{i}+(c-A)(|x|^{\mu}+K\rho^{\mu}).
\end{align*}
Note that $\rho\leq Cd\leq C|x|$ by \eqref{defining-distance} and $0\in\partial\Omega$. Then, 
$$
I_{1}\leq C|x|^{\mu-2}\rho^{2}+C\rho^{\mu}\leq  C|x|^{\mu},
$$
and
$$
I_{2}\leq C|x|^{\mu}+(c-A)(|x|^{\mu}+K\rho^{\mu}).
$$
By taking $A$ sufficiently large, we obtain
\begin{align*}
L(e^{At}w)\leq e^{At}|x|^{\mu}(C+c-A)+e^{At}(c-A)K\rho^{\mu}\leq -C_{0}|x|^{\mu},
\end{align*}
where $C_{0}$ is a positive constant.
\end{proof}

We point out that $K$ is allowed to be zero in Lemma \ref{2D}. We also need the following result later on.

\begin{lemma}\label{2E}
Let $\mu\in (0,1)$ be a constant, 
$\Omega$ be a bounded domain in $\mathbb R^n$ with a $C^{1}$-boundary $\partial\Omega$, 
and $\rho$ be a $C^{1}(\bar\Omega)\cap C^{2}(\Omega)$-defining function with $\rho\nabla^2\rho\in C(\bar\Omega)$ and $\rho\nabla^2\rho=0$ on $\partial\Omega$. 
Then, there exist positive constants $r$ and $K$, 
depending only on $n$, $\mu$, $\Omega$, and $\rho$,
such that, for any $x_0\in\partial\Omega$ and 
any $x\in \Omega\cap B_r(x_0)$,  
\begin{align*}\Delta\big(|x-x_0|^{\mu}+K \rho^\mu\big)
\le0.\end{align*}
\end{lemma}

\begin{proof} 
Without loss of generality, we assume $x_{0}=0$. Set, for some $K>0$ to be determined,
$$
w(x)=|x|^{\mu}+K\rho^{\mu}.
$$
By \eqref{defining-distance}, $0\in\partial\Omega$, and $\mu\in(0,1)$, we have $|x|^{\mu-2}\leq d^{\mu-2}\leq C\rho^{\mu-2}$. 
A straightforward computation yields
\begin{align*}
\Delta w&=\mu(\mu-2+n)|x|^{\mu-2}+K\mu(\mu-1)\rho^{\mu-2}|\nabla\rho|^2+K\mu \rho^{\mu-1}\Delta \rho\\
&\leq C\rho^{\mu-2}+K\mu(\mu-1)\rho^{\mu-2}|\nabla\rho|^2+K\mu \rho^{\mu-1}\Delta \rho\\
&=\rho^{\mu-2}\big\{C+K\mu[(\mu-1)|\nabla\rho|^2+\rho\Delta \rho]\big\}.
\end{align*}
Note that $\mu\in(0,1)$ and $|\nabla\rho|=1$ on $\partial\Omega$. 
By taking $r>0$ sufficiently small, we obtain, for any $x\in \Omega\cap B_{r}$,
$$
(\mu-1)|\nabla\rho|^2+ \rho\Delta \rho<-\frac{1}{2}(1-\mu).
$$
We fix such $r$. Then, by taking $K>0$ sufficiently large, we have $\Delta w\leq 0$ in $\Omega\cap B_{r}$.
\end{proof} 

With Lemma \ref{2D}, we prove a decay estimate of solutions near the
lateral boundary $S_{T}$.

\begin{lemma}\label{2F}
Let $\Omega$ be a bounded domain in $\mathbb R^n$ with a $C^{1}$-boundary $\partial\Omega$ 
and $\rho$ be a $C^{1}(\bar\Omega)\cap C^2(\Omega)$-defining function with $\rho\nabla^2\rho\in L^{\infty}(\Omega)$. 
For some constant $\alpha\in(0,1)$, assume $a_{ij},b_{i},c\in C^{\alpha}(\bar{Q}_{T})$ with \eqref{1a}. 
Suppose, for some $f\in C^{\alpha}(\bar{Q}_{T})$ and $\phi\in C^{\alpha}(\bar{\Omega})$, 
the initial-boundary value problem \eqref{1c}-\eqref{1e} admits a solution $u\in C(\bar{Q}_{T})\cap C^{2}(Q_{T})$. 
Then, for any $(x,t)\in Q_{T}\cup\Omega_{0}$ and $(x_{0},t)\in S_{T}$,
\begin{equation}\label{2d}
|u(x,t)-h(x_{0},t)|\leq C\big\{\|f\|_{C^{\alpha}(\bar{Q}_{T})}+\|\phi\|_{ C^{\alpha}(\bar{\Omega})}\big\}|x-x_{0}|^{\alpha},
\end{equation}
where $C$ is a positive constant depending only on $n$, $T$, $\alpha$, $\Omega$, $\rho$, 
and the $C^{\alpha}$-norms of $a_{ij}$, $b_{i}$, and $c$ in $\bar{Q}_{T}$.
\end{lemma}

\begin{proof}
For convenience, we set
$$
F=\|f\|_{C^{\alpha}(\bar{Q}_{T})}+\|\phi\|_{ C^{\alpha}(\bar{\Omega})}.
$$
Fix a point $x_{0}\in \partial\Omega$. By \eqref{1f}, we have
\begin{equation}\label{2e}
L(u(x,t)-h(x_{0},t))=g,
\end{equation}
where
\begin{align*}
g(x,t)&=f(x,t)-c(x,t)h(x_{0},t)+\partial_{t}h(x_{0},t)\\
&=f(x,t)-f(x_{0},t)-(c(x,t)-c(x_{0},t))h(x_{0},t).
\end{align*}
Then,
$$
|g(x,t)|\leq CF|x-x_{0}|^{\alpha}.
$$
Hence, for any $(x,t)\in Q_{T}$,
$$
\pm L(u(x,t)-h(x_{0},t))\geq -CF|x-x_{0}|^{\alpha}.
$$
Next, we have, for any $x\in\Omega$,
$$
\pm (u(x,0)-h(x_{0},0))=\pm(\phi(x)-\phi(x_{0}))\leq CF|x-x_{0}|^{\alpha},
$$
and, for any $(x,t)\in S_{T}$,
$$
\pm (u(x,t)-h(x_{0},t))=\pm (h(x,t)-h(x_{0},t))\leq CF|x-x_{0}|^{\alpha}.
$$
For some $A>0$ to be determined, set
$$
w(x,t)=e^{At}|x-x_{0}|^{\alpha}.
$$
By Lemma \ref{2D}, we have, for some positive constants $A$ and $C_{0}$, and any $(x,t)\in Q_T,$
$$
Lw\leq -C_{0}|x-x_{0}|^{\alpha}.
$$
We also have  
$$
w\geq |x-x_{0}|^{\alpha}\quad\text{on }\partial_{p}Q_{T}.
$$
For $B>0$ sufficiently large, we obtain
\begin{align*}
\pm L(u(x,t)-h(x_{0},t))&\geq L(BFw)\quad\text{in }Q_T\\
\pm (u(x,t)-h(x_{0},t))&\leq BFw \quad\text{on }\partial_{p}Q_T.
\end{align*}
The maximum principle implies
$$
\pm (u(x,t)-h(x_{0},t))\leq BFw\quad\text{in }Q_T,
$$
and hence, for any $(x,t)\in Q_T$,
\begin{equation}\label{2f}
|u(x,t)-h(x_{0},t)|\leq BFe^{At}|x-x_{0}|^{\alpha}\leq B'F|x-x_{0}|^{\alpha}.
\end{equation}
We obtain the desired result.
\end{proof}

Next, we prove a regularity result for solutions of the initial-boundary value problem.
We need Corollary \ref{5G} in the proof. 

\begin{thm}\label{2G}
Let $\Omega$ be a bounded domain in $\mathbb R^n$ with a $C^{1}$-boundary $\partial\Omega$ 
and $\rho$ be a $C^{1}(\bar\Omega)\cap C^2(\Omega)$-defining function with $\rho\nabla^2\rho\in L^{\infty}(\Omega)$. 
For some constant $\alpha\in(0,1)$, assume $a_{ij},b_{i},c\in C^{\alpha}(\bar{Q}_{T})$ with \eqref{1a}. 
Suppose that, for some $f\in C^{\alpha}(\bar{Q}_{T})$ and $\phi\in C^{2+\alpha}(\bar{\Omega})$, 
the initial-boundary value problem \eqref{1c}-\eqref{1e} admits a solution $u\in C^{2}(Q_{T})\cap C(\bar{Q}_{T})$. 
Then, $u\in C^{2+\alpha}(\bar{Q}_{T})$ and
$$
\|u\|_{C^{2+\alpha}(\bar{Q}_{T})}\leq C\big\{\|\phi\|_{C^{2+\alpha}(\bar{\Omega})}+\|f\|_{C^{\alpha}(\bar{Q}_{T})}\big\},
$$
where $C$ is a positive constant depending only on $n,$ $T,$ $\lambda,$ $\alpha,$ $\Omega,$ $\rho$, 
and the $C^{\alpha}$-norms of $a_{ij},$ $b_{i},$ and $c$ in $\bar{Q}_{T}$.
\end{thm}

\begin{proof}
The Schauder theory (interior in space and global in time) for uniformly parabolic equations 
implies $u\in C^{2,\alpha}(Q_{T}\cup \Omega_{0})$. (Refer to Section 10 in Chapter IV
of \cite{L1968}.) 
Now we derive a weighted $C^{2,\alpha}$-estimate of $u$. For convenience, we set
$$
F=\|f\|_{C^{\alpha}(\bar{Q}_{T})}+\|\phi\|_{C^{2+\alpha}(\bar{\Omega})}.
$$
We claim, for any $(x,t)\in Q_{T}\cup\Omega_{0}$ and $(x_{0},t)\in S_{T}$ with $d(x)=|x-x_{0}|$,
\begin{equation}\label{2g}
|u(x,t)-h(x_{0},t)|+|\rho(x)D_{x}u(x,t)|+|\rho^{2}(x)D^{2}_{x}u(x,t)|\leq CF d^{\alpha}(x),
\end{equation}
and
\begin{align}\label{2h}\begin{split}
[u]_{C^{\alpha}(B_{d(x)/4}(x)\times[0,T])}+[\rho D_{x}u]_{C^{\alpha}(B_{d(x)/4}(x)\times[0,T])}&\\
+[\rho^{2}D^{2}_{x}u]_{C^{\alpha}(B_{d(x)/4}(x)\times[0,T])}    
&\leq CF.
\end{split}\end{align}
Note that $\|h\|_{C^\alpha(S_T)}\leq CF$. By 
Lemma \ref{2C}, we obtain
$$
\|u\|_{C^{\alpha}(\bar{Q}_{T})}+\|\rho D_{x}u\|_{C^{\alpha}(\bar{Q}_{T})}+\|\rho^{2}D^{2}_{x}u\|_{C^{\alpha}(\bar{Q}_{T})}\leq CF,
$$
$\rho D_{x}u|_{S_{T}}=0$, and $\rho^{2} D^{2}_{x}u|_{S_{T}}=0$. 
Also, by \eqref{1c},
we get
$$
\|\partial_{t}u\|_{C^{\alpha}(\bar{Q}_{T})}\leq CF.
$$
We hence obtain the desired result.

We take arbitrary 
$x\in\Omega$ and 
$x_0\in\partial\Omega$ with $d(x)=|x-x_{0}|$. 
Set $r=d(x)/2$. Then, for any $y\in B_{r}(x)$,
$$
|d(x)-d(y)|\leq \|Dd\|_{L^{\infty}(B_{r}(x))}|x-y|\leq \frac{1}{2}d(x),
$$
and hence
$$
\frac{1}{2}d(x)\leq d(y)\leq \frac{3}{2}d(x).
$$
Thus, by \eqref{defining-distance}, we have
$$
c_1\rho(x)\leq \rho(y)\leq c_2\rho(x),
$$
for some positive constants $c_1$ and $c_2$. 
We consider
$$
v(y,t)=u(x+ry,t)-h(x_{0},t)\quad\text{in }Q'_{T},
$$
where $Q'_{T}=B_{1}(0)\times (0,T]$. A straightforward computation yields 
\begin{equation}\label{2i}
\Tilde{a}_{ij}\partial_{ij}v+\Tilde{b}_{i}\partial_{i}v+\Tilde{c}v-\partial_{t}v=\Tilde{g}\quad\text{in }Q'_{T},
\end{equation}
and
$$
v(y,0)=\Tilde{\psi}(y)\quad\text{on }B_{1}(0),
$$
where
\begin{align*}
\Tilde{a}_{ij}(y,t)&=\frac{\rho^{2}(x+ry)}{r^{2}}a_{ij}(x+ry,t),\\
\Tilde{b}_{i}(y,t)&=\frac{\rho(x+ry)}{r}b_{i}(x+ry,t),\\
\Tilde{c}(y,t)&=c(x+ry,t),
\end{align*}
and
\begin{align*}
\Tilde{\psi}(y)&=\phi(x+ry)-\phi(x_{0}),\\
\Tilde{g}(y,t)&=f(x+ry,t)-c(x+ry,t)h(x_{0},t)+\partial_{t}h(x_{0},t)\\
&=f(x+ry,t)-f(x_{0},t)-(c(x+ry,t)-c(x_{0},t))h(x_{0},t).
\end{align*}
It is easy to verify that
$$
\|\Tilde{a}_{ij}\|_{C^{\alpha}(\bar{Q}'_{T})}+\|\Tilde{b}_{i}\|_{C^{\alpha}(\bar{Q}'_{T})}+\|\Tilde{c}\|_{C^{\alpha}(\bar{Q}'_{T})}\leq C,
$$
where $C$ is a positive constant independent of $x$. Moreover, for any $\xi\in\mathbb{R}^{n}$ and any $(y,t)\in Q'_{T}$,
$$
\Tilde{a}_{ij}(y,t)\xi_{i}\xi_{j}\geq \lambda \frac{\rho^{2}(x+ry)}{r^{2}}|\xi|^{2}\geq \lambda c_3 |\xi|^{2}.
$$
Set $Q''_{T}=B_{1/2}(0)\times(0,T]$. By applying Corollary \ref{5G} to \eqref{2i} in $Q'_{T}$, we obtain
\begin{equation}\label{2j}
\begin{aligned}
\|v\|_{C^{2,\alpha}_{\ast}(\bar{Q}''_{T})}\leq C\big\{\|v\|_{L^{\infty}(Q'_{T})}+\|\Tilde{\psi}\|_{C^{2,\alpha}(\bar{B}_{1}(0))}
+\|\Tilde{g}\|_{C^{\alpha}_{\ast}(\bar{Q}'_{T})}\big\}.
\end{aligned}
\end{equation}
Set $M'_{T}=B_{r}(x)\times (0,T]$ and $M''_{T}=B_{r/2}(x)\times (0,T]$. By Lemma \ref{2F}, we have
\begin{equation}\label{2k}
\|v\|_{L^{\infty}(Q'_{T})}=\|u-h(x_{0},\cdot)\|_{L^{\infty}(M'_{T})}\leq CFr^{\alpha}.
\end{equation}
By $\phi\in C^{2+\alpha}(\bar{\Omega})$, we get
\begin{align}\label{2l}
\begin{split}
\|\Tilde{\psi}\|_{C^{2,\alpha}(\bar{B}_{1}(0))}
&=\|\phi-\phi(x_{0})\|_{L^{\infty}(B_{r}(x))}+r\|D\phi\|_{L^{\infty}(B_{r}(x))}\\
&\qquad+r^{2}\|D^{2}\phi\|_{L^{\infty}(B_{r}(x))}
+r^{2+\alpha}[D^{2}\phi]_{C^{\alpha}(B_{r}(x))}\\
&\leq  CFr^{\alpha}.
\end{split}
\end{align}
Similarly, we also have
\begin{equation}\label{2m}
\|\Tilde{g}\|_{C^{\alpha}_{\ast}(\bar{Q}'_{T})}\leq CFr^{\alpha}.
\end{equation}
By combining \eqref{2j}-\eqref{2m}, we obtain
$$
\begin{aligned}
&r\|D_{x}u\|_{L^{\infty}(M''_{T})}+r^{2}\|D^{2}_{x}u\|_{L^{\infty}(M''_{T})}\\
&\qquad\qquad+r^{\alpha}[u]_{C^{\alpha}_{x}(\bar{M}''_{T})}+r^{1+\alpha}[D_{x}u]_{C^{\alpha}_{x}(\bar{M}''_{T})}
+r^{2+\alpha}[D^{2}_{x}u]_{C^{\alpha}_{x}(\bar{M}''_{T})}\\
&\qquad\qquad+[u-h(x_{0},\cdot)]_{C^{\alpha/2}_{t}(\bar{M}''_{T})}+r[D_{x}u]_{C^{\alpha/2}_{t}(\bar{M}''_{T})}
+r^{2}[D^{2}_{x}u]_{C^{\alpha/2}_{t}(\bar{M}''_{T})}\\
&\qquad\leq CFr^{\alpha}.
\end{aligned}
$$
With \eqref{2d}, we thus finish the proof of \eqref{2g} and \eqref{2h}.
\end{proof}

\begin{remark}\label{2G'}
If we use the classical Schauder estimate (interior in space and global in time) to \eqref{2i} instead of using Corollary \ref{5G}, 
then \eqref{2j} is replaced by
$$
\|v\|_{C^{2,\alpha}(\bar{Q}''_{T})}\leq C\big\{\|v\|_{L^{\infty}(Q'_{T})}
+\|\Tilde{\psi}\|_{C^{2,\alpha}(\bar{B}_{1}(0))}+\|\Tilde{g}\|_{C^{\alpha}(\bar{Q}'_{T})}\big\}.
$$
However, we can only obtain
$$
\|\Tilde{g}\|_{C^{\alpha}(\bar{Q}'_{T})}\leq CF.
$$
A factor $r^{\alpha}$ is missing in the right-hand side due to the presence of the H\"older semi-norm
$[\Tilde{g}]_{C^{\alpha/2}_{t}(\bar{Q}'_{T})}$. 

In addition, we emphasize that to apply Corollary \ref{5G}, 
we apply the classical Schauder estimate (interior in space and global in time) first to obtain $v\in C^{2}(\bar{Q}'_{T})$.
\end{remark}

We are ready to prove Theorem \ref{1A}. 

\begin{proof}[Proof of Theorem \ref{1A}.]
First, the maximum principle implies the uniqueness of solutions of \eqref{1c}-\eqref{1e} in the $C^{2}(Q_{T})\cap C(\bar{Q}_{T})$-category.

Next, we prove the existence of a $C^{2+\alpha}(\bar{Q}_{T})$-solution $u$ of \eqref{1c}-\eqref{1e} 
for the given $f\in C^{\alpha}(\bar{Q}_{T})$ and $\phi\in C^{2+\alpha}(\bar{\Omega})$. 
Take an arbitrary $\delta>0$. The operator $L_{\delta}=L+\delta \Delta$ is uniformly parabolic in $Q_{T}$. 
By the Schauder theory, there exists a unique solution $u_{\delta}\in C^{2,\alpha}(Q_{T}\cup\Omega_{0})\cap C(\bar{Q}_{T})$ 
of the initial-boundary problem
\begin{align}
L_{\delta}u_{\delta}&=f\quad\text{in }Q_{T},\label{2n}\\
u_{\delta}(\cdot,0)&=\phi\quad\text{on }\Omega,\label{2o}\\
u_{\delta}&=h\quad\text{on }S_{T}.\label{2p}
\end{align}
Note that $h$ and $\phi$ satisfy
\begin{align}\label{2q}\begin{split}
\partial_{t}h-ch+f&=0\quad\text{on } S_{T},\\
h(\cdot,0)&=\phi\quad\text{on }\partial\Omega.
\end{split}
\end{align}
We now derive estimates of $u_{\delta}$, independent of $\delta\in(0,1)$. For convenience, we set
$$
F=\|f\|_{C^{\alpha}(\bar{Q}_{T})}+\|\phi\|_{C^{2+\alpha}(\bar{\Omega})}.
$$

Step 1. We derive an $L^{\infty}$-estimate of $u_{\delta}$. This step is the same as the proof of Lemma \ref{2B}. We have
\begin{equation}\label{delta-infty-estimate}
\|u_{\delta}\|_{L^{\infty}(Q_{T})}\leq C\Big\{\sup_{\partial_{p}Q_{T}}|u_{\delta}|+\sup_{Q_{T}}|f|\Big\}\leq CF.
\end{equation}

Step 2. We derive a decay estimate near the lateral boundary $S_{T}$. 
Let $K$ and $r$ be given by Lemma \ref{2E} and $A$ be given by Lemma \ref{2D}. Set, for any $x_{0}\in \partial\Omega$ and any $(x,t)\in(\Omega\cap B_{r}(x_{0}))\times (0,T]$, 
$$
w= e^{At}(|x-x_{0}|^{\alpha}+K \rho^{\alpha}).
$$
Then, 
$$
L_{\delta}w\leq -C_{0}|x-x_{0}|^{\alpha}.
$$
Fix a point $x_{0}\in \partial\Omega$. By \eqref{2q}, we have
\begin{equation}\label{2r}
L_{\delta}(u_{\delta}(x,t)-h(x_{0},t))=g,
\end{equation}
where
\begin{align*}
g(x,t)&=f(x,t)-c(x,t)h(x_{0},t)+\partial_{t}h(x_{0},t)\\
&=f(x,t)-f(x_{0},t)-(c(x,t)-c(x_{0},t))h(x_{0},t).
\end{align*}
Then,
$$
|g(x,t)|\leq CF|x-x_{0}|^{\alpha}.
$$
By proceeding similarly as in the proof of Lemma \ref{2F}, we have, for any $(x,t)\in Q_{T}\cup\Omega_{0}$ and $(x_{0},t)\in S_{T}$,
\begin{equation}\label{2r'}
|u_{\delta}(x,t)-h(x_{0},t)|\leq CF|x-x_{0}|^{\alpha}.
\end{equation}
Different from the proof of Lemma \ref{2F}, here we use the maximum principle in $(\Omega\cap B_{r}(x_{0}))\times (0,T]$. 
We need \eqref{delta-infty-estimate} and note that $w\geq e^{At}r^\alpha$ on $(\Omega\cap \partial B_{r}(x_{0}))\times (0,T]$.

Step 3. We derive a weighted $C^{2,\alpha}$-estimate of $u_{\delta}$. Take arbitrary 
$x\in\Omega$
and 
$x_0\in\partial\Omega$ with $d(x)=|x-x_{0}|$. Set $r=d(x)/2$. 
We consider
$$
v_{\delta}(y,t)=u_{\delta}(x+ry,t)-h(x_{0},t)\quad\text{in }Q'_{T},
$$
where $Q'_{T}=B_{1}(0)\times (0,T]$.
A straightforward computation yields 
\begin{equation}\label{2s}
\bar{a}_{ij}\partial_{ij}v_{\delta}+\bar{b}_{i}\partial_{i}v_{\delta}+\bar{c}v_{\delta}-\partial_{t}v_{\delta}=\bar{g}\quad\text{in }Q'_{T},
\end{equation}
and
$$
v_{\delta}(y,0)=\bar{\psi}(y)\quad\text{on }B_{1}(0),
$$
where
\begin{align*}
\bar{a}_{ij}(y,t)&=\frac{\rho^{2}(x+ry)a_{ij}(x+ry,t)+\delta\delta_{ij}}{r^2},\\
\bar{b}_{i}(y,t)&=\frac{\rho(x+ry)}{r} b_{i}(x+ry,t),\\
\bar{c}(y,t)&=c(x+ry,t),
\end{align*}
and
\begin{align*}
\bar{\psi}(y)&=\phi(x+ry)-\phi(x_{0}),\\
\bar{g}(y,t)&=f(x+ry,t)-c(x+ry,t)h(x_{0},t)+\partial_{t}h(x_{0},t)\\
&=f(x+ry,t)-f(x_{0},t)-(c(x+ry,t)-c(x_{0},t))h(x_{0},t).
\end{align*}

Fix an $x$. It is easy to verify that
$$
\|\bar{a}_{ij}\|_{C^{\alpha}(\bar{Q}'_{T})}+\|\bar{b}_{i}\|_{C^{\alpha}(\bar{Q}'_{T})}
+\|\bar{c}\|_{C^{\alpha}(\bar{Q}'_{T})}\leq C_{x,\delta}:=C+C\frac{\delta}{r^2},
$$
where $C$ is a positive constant independent of $x$ and $\delta$. Note that
$C_{x,\delta}$ is a positive constant satisfying
\begin{equation}\label{2s'}
C_{x,\delta}\leq C_{x},\quad \lim_{\delta\rightarrow 0} C_{x,\delta}=C,
\end{equation}
where $C_x$ is a positive constant independent of $\delta\in(0,1)$ but depending on $x$.
Moreover, for any $\xi\in\mathbb{R}^{n}$ and any $(y,t)\in Q'_{T}$,
$$
\bar{a}_{ij}(y,t)\xi_{i}\xi_{j}\geq  \frac{\lambda \rho^{2}(x+ry)+\delta}{r^{2}}|\xi|^{2}\geq  \lambda c_4|\xi|^{2}.
$$
Set $Q''_{T}=B_{1/2}(0)\times(0,T]$. By applying Corollary \ref{5G} to \eqref{2s} in $Q'_{T}$, we obtain
\begin{align*}
\|v_\delta\|_{C^{2,\alpha}_{\ast}(\bar{Q}''_{T})}\leq C'_{x,\delta}\big\{\|v_{\delta}\|_{L^{\infty}(Q'_{T})}
+\|\bar{\psi}\|_{C^{2,\alpha}(\bar{B}_{1}(0))}+\|\bar{g}\|_{C^{\alpha}_{\ast}(\bar{Q}'_{T})}\big\},
\end{align*}
where $C'_{x,\delta}$ is a positive constant depending on $C_{x,\delta}$.
By proceeding as in the proof of Theorem \ref{2G}, we have
\begin{align}\label{2t'}\begin{split}
|u_{\delta}(x,t)-h(x_{0},t)|+|\rho(x)D_{x}u_{\delta}(x,t)|& \\
+|\rho^{2}(x)D^{2}_{x}u_{\delta}(x,t)|&\leq C''_{x,\delta}F d^{\alpha}(x),
\end{split}\end{align}
and
\begin{align}\label{2u'}\begin{split}
[u_{\delta}]_{C^{\alpha}(B_{d(x)/4}(x)\times[0,T])}+[\rho D_{x}u_{\delta}]_{C^{\alpha}(B_{d(x)/4}(x)\times[0,T])} &\\
+[\rho^{2}D^{2}_{x}u_{\delta}]_{C^{\alpha}(B_{d(x)/4}(x)\times[0,T])}
&\leq C''_{x,\delta}F,
\end{split}\end{align}
where $C''_{x,\delta}$ is a positive constant depending on $C_{x,\delta}$.
Set $Q_{T,x}=B_{d(x)/4}(x)\times(0,T]$. 
By \eqref{2s'}-\eqref{2u'} and \eqref{2n}, there exists a sequence $\delta=\delta_{i}(x)\rightarrow 0$ such that
\begin{align*}
u_{\delta}\rightarrow u,\quad\rho D_{x}u_{\delta}\rightarrow u^{(1)},\quad\rho^{2} D^{2}_{x}u_{\delta}\rightarrow u^{(2)},
\quad \partial_{t}u_{\delta}\rightarrow u^{(3)} &\\
\quad\text{uniformly in }\bar{Q}_{T,x},&
\end{align*}
for some $u,u^{(3)}\in C^{\alpha}(\bar{Q}_{T,x}) $, $u^{(1)}\in C^{\alpha}(\bar{Q}_{T,x};\mathbb{R}^{n})$, 
$u^{(2)}\in C^{\alpha}(\bar{Q}_{T,x};\mathbb{R}^{n\times n})$.
Hence, we have $u\in C^{2,\alpha}(\bar{Q}_{T,x})$, with $\rho D_{x}u=u^{(1)}$, $\rho^{2} D^{2}_{x}u=u^{(2)}$, and $\partial_{t}u=u^{(3)}$. 
By letting $\delta=\delta_{i}(x)\rightarrow 0$ in \eqref{2t'} and \eqref{2u'} and using \eqref{2s'}, we obtain \begin{equation}\label{2v'}
|u(x,t)-h(x_{0},t)|+|\rho(x)D_{x}u(x,t)|+|\rho^{2}(x)D^{2}_{x}u(x,t)|\leq CF d^{\alpha}(x),
\end{equation}
and
\begin{align}\label{2w'}\begin{split}
[u]_{C^{\alpha}(B_{d(x)/4}(x)\times[0,T])}+[\rho D_{x}u]_{C^{\alpha}(B_{d(x)/4}(x)\times[0,T])} &\\
+[\rho^{2}D^{2}_{x}u]_{C^{\alpha}(B_{d(x)/4}(x)\times[0,T])}
&\leq CF,
\end{split}
\end{align}
where $C$ is a positive constant independent of $x$. We cover $\Omega$ by
$$
\{B_{d(x_{i})/4}(x_{i}):x_{i}\in\Omega,i=1,2,\cdots\}.$$ By proceeding as in the above discussion in $B_{d(x_{i})/4}(x_{i})$, 
$i=1,2,\cdots$, successively and using a standard diagonalization,
 there exist a sequence $\delta=\delta_{i}\rightarrow 0$ and $u\in C^{2,\alpha}(Q_{T}\cup\Omega_{0})$ such that
\begin{align*}
&u_{\delta}\rightarrow u,\quad D_{x}u_{\delta}\rightarrow D_{x}u,\quad D^{2}_{x}u_{\delta}
\rightarrow D^{2}_{x}u,\quad \partial_{t}u_{\delta}\rightarrow \partial_{t}u\\
&\text{uniformly in any compact subsets of }Q_{T}\cup\Omega_{0}.
\end{align*}
In addition, by \eqref{2s'}-\eqref{2u'}, 
we have \eqref{2v'} and \eqref{2w'} for any $(x,t)\in Q_{T}\cup\Omega_{0}$ and $(x_{0},t)\in S_{T}$ with $d(x)=|x-x_{0}|$.
By Lemma \ref{2C} and letting $\delta=\delta_{i}\rightarrow 0$ in \eqref{2n}-\eqref{2p}, 
we conclude that $u\in C^{2+\alpha}(\bar{Q}_{T})$ is a solution of \eqref{1c}-\eqref{1e} with
$$
\|u\|_{C^{2+\alpha}(\bar{Q}_{T})}\leq CF.
$$
This is the desired result.
\end{proof}

\section{The Higher Regularity of Solutions}\label{sec-Higher-regularity}

In this section, we study the H\"{o}lder continuity of solutions and their derivatives up to the boundary.
There are three subsections. We will study the tangential spatial derivatives, normal spatial derivatives,
and time derivatives in these subsections, respectively.
We will do this in a special setting and consider domains with a piece of flat boundary.

Denote by $x=(x',x_{n})$ points in $\mathbb{R}^{n}$ and $X=(x,t)=(x',x_{n},t)$ points in $\mathbb{R}^{n}\times[0,\infty)$. 
Set, for any $r,T>0$,
\begin{align*}
G_{r}&=\{x\in\mathbb{R}^{n}:|x'|<r,0<x_{n}<r\},\\
B'_{r}&=\{x'\in\mathbb{R}^{n-1}:|x'|<r\},
\end{align*}
and
\begin{align*}
Q_{r,T}&=\{(x,t)\in\mathbb{R}^{n}\times[0,\infty):x\in G_{r},0<t\leq T\},\\
\Omega_{r,0}&=\{(x,0)\in\mathbb{R}^{n}\times[0,\infty):x\in G_{r}\},\\
S_{r,T}&=\{(x',0,t)\in\mathbb{R}^{n}\times[0,\infty):x'\in B'_{r},0\leq t\leq T\}.
\end{align*}
We also define the weighted H\"older function spaces locally. 
For a function $u:Q_{r,T}\rightarrow \mathbb{R}$ with continuous derivatives $D_{x}u$, $D^{2}_{x}u$, and $\partial_{t}u$, 
we define
\begin{align*}
\|u\|_{C^{2+\alpha}(\bar{Q}_{r,T})}&=\|u\|_{C^{\alpha}(\bar{Q}_{r,T})}+\|x_n D_{x}u\|_{C^{\alpha}(\bar{Q}_{r,T})}\\
&\qquad+\|x^{2}_n D^{2}_{x}u\|_{C^{\alpha}(\bar{Q}_{r,T})}+\|\partial_{t}u\|_{C^{\alpha}(\bar{Q}_{r,T})},
\end{align*}
and
\begin{align*}
C^{2+\alpha}(\bar{Q}_{r,T})&=\{u\in C^{2}(Q_{r,T}):\|u\|_{C^{2+\alpha}(\bar{Q}_{r,T})}<\infty,\\
&\qquad x_n D_{x}u|_{S_{r,T}}=0, \text{ and }x^{2}_n D^{2}_{x}u|_{S_{r,T}}=0\}.
\end{align*}
Similarly, we can define $C^{k,2+\alpha}(\bar{G}_r)$, $C^{k,2+\alpha}(\bar{Q}_{r,T})$, $C^{k,2+\alpha}_{\ast}(\bar{Q}_{r,T})$, $\cdots$.

Let $a_{i j}$, $b_{i}$, and $c$ be continuous functions in $\bar{Q}_{1,T}$, with $a_{i j}=a_{j i}$ and, for any $(x,t)\in\bar{Q}_{1,T}$ 
and any $\xi\in\mathbb{R}^{n}$,
\begin{equation}\label{3a}
\lambda|\xi|^{2}\leq a_{ij}(x,t)\xi_{i}\xi_{j}\leq \Lambda|\xi|^{2},
\end{equation}
for some positive constants $\lambda$ and $\Lambda$. We consider the operator
\begin{equation}\label{3b}
L=x^{2}_{n} a_{ij}\partial_{ij}+x_{n} b_{i}\partial_{i}+c-\partial_{t}\quad\text{in }Q_{1,T}.
\end{equation}
Let $f$ be a continuous function in $\bar{Q}_{1,T}$. We consider
\begin{align}
Lu&=f\quad\text{in }Q_{1,T},\label{3c}\\
u(\cdot,0)&=\phi\quad\text{on }G_{1},\label{3d}\\
u&=h\quad\text{on } S_{1,T},\label{3e}
\end{align}
where $\phi\in C(\bar{G}_{1})$ and
$$
h(x,t)=\phi(x)\exp\left\{\int^{t}_{0}c(x,s)ds\right\}-\int^{t}_{0}\exp\left\{\int^{t}_{s}c(x,\tau)d\tau\right\}f(x,s) d s.
$$
In fact, $h$ and $\phi$ satisfy
\begin{align}\label{3f}\begin{split}
\partial_{t}h-ch+f&=0\quad\text{on } S_{1,T},\\
h(\cdot,0,0)&=\phi(\cdot,0)\quad\text{on }  B'_{1} .
\end{split}
\end{align}
By proceeding similarly as in Section \ref{sec-Existence}, we can prove the local versions of Lemma \ref{2C}, 
Lemma \ref{2D}, Lemma \ref{2F}, and Theorem \ref{2G}.

\begin{lemma}\label{3A}
Let $A>0$ and $\alpha\in(0,1)$ be constants and $w\in C^{\alpha}(Q_{1,T}\cup\Omega_{1,0})$ 
and $w_{0}\in C^{\alpha}(S_{1,T})$ be functions. Suppose, for any $(x,t)\in Q_{1,T}\cup\Omega_{1,0}$ with $B_{x_n/2}(x)\subset G_1$,
\begin{equation}\label{3g}
|w(x,t)-w_{0}(x',0,t)|\leq A x^{\alpha}_{n},
\end{equation}
and
\begin{equation}\label{3h}
[w]_{C^{\alpha}(B_{x_{n}/2}(x)\times[0,T])}\leq A.
\end{equation}
Then, $w\in C^{\alpha}(\bar{Q}_{r,T})$, for any $r\in(0,1)$, and
$$
\|w\|_{C^{\alpha}(\bar{Q}_{1/2,T})}\leq 5A+\|w_{0}\|_{ C^{\alpha}(S_{1,T})}.
$$
\end{lemma}

\begin{lemma}\label{3B}
Let $\mu>0$ be a constant. Assume $a_{ij},b_{i},c\in C(\bar{Q}_{1,T})$ with \eqref{3a}. 
Then, there exist positive constants $A$ and $C_{0}$, depending only on $\mu$, 
and the $L^{\infty}$-norms of $a_{ij}$, $b_{i}$, and $c$ in $Q_{1,T}$, such that, for any $x'_{0}\in B'_{1}$ and
any $(x,t)\in Q_{1,T}$,
$$
L(e^{At}|x-(x'_{0},0)|^{\mu})\leq -C_{0}|x-(x'_{0},0)|^{\mu}.
$$
\end{lemma}

\begin{lemma}\label{3C}
For some constant $\alpha\in(0,1)$, assume $a_{ij},b_{i},c\in C^{\alpha}(\bar{Q}_{1,T})$ with \eqref{3a}. 
Suppose, for some $f\in C^{\alpha}(\bar{Q}_{1,T})$ and $\phi\in C^{\alpha}(\bar{G}_{1})$, \eqref{3c}-\eqref{3e} 
admits a solution $u\in C(\bar{Q}_{1,T})\cap C^{2}(Q_{1,T})$. Then, for any $(x,t)\in Q_{1/2,T}\cup\Omega_{1/2,0}$ and $x'_{0}\in B'_{1/2}$,
\begin{align}\label{3i}\begin{split}
&|u(x,t)-h(x'_{0},0,t)|\\
&\qquad\leq C\big\{\|u\|_{L^{\infty}(Q_{1,T})}+\|f\|_{C^{\alpha}(\bar{Q}_{1,T})}+\|\phi\|_{ C^{\alpha}(\bar{G}_{1})}\big\}|x-(x'_{0},0)|^{\alpha},
\end{split}\end{align}
where $C$ is a positive constant depending only on $n$, $T$, $\alpha$, and the $C^{\alpha}$-norms 
of $a_{ij}$, $b_{i}$, and $c$ in $\bar{Q}_{1,T}$.
\end{lemma}

\begin{thm}\label{3D}
For some constant $\alpha\in(0,1)$, assume $a_{ij},b_{i},c\in C^{\alpha}(\bar{Q}_{1,T})$ with \eqref{3a}. 
Suppose that, for some $f\in C^{\alpha}(\bar{Q}_{1,T})$ and $\phi\in C^{2+\alpha}(\bar{G}_{1})$, \eqref{3c}-\eqref{3e} 
admits a solution $u\in C^{2}(Q_{1,T})\cap C(\bar{Q}_{1,T})$. Then, for any $r\in(0,1)$, $u\in C^{2+\alpha}(\bar{Q}_{r,T})$ and
$$
\|u\|_{C^{2+\alpha}(\bar{Q}_{1/2,T})}\leq C\big\{\|u\|_{L^{\infty}(Q_{1,T})}+\|f\|_{C^{\alpha}(\bar{Q}_{1,T})}+\|\phi\|_{C^{2+\alpha}(\bar{G}_{1})}\big\},
$$
where $C$ is a positive constant depending only on $n,$ $T,$ $\lambda,$ $\alpha,$  and the $C^{\alpha}$-norms 
of $a_{ij},$ $b_{i},$ and $c$ in $\bar{Q}_{1,T}$.
\end{thm}

\begin{remark}\label{boundedness}
In Theorem \ref{3D}, it suffices to assume $u\in C^{2}(Q_{1,T})\cap C(Q_{1,T}\cup\Omega_{1,0}) \cap L^{\infty}(Q_{1,T})$. 
Then, we can prove that $u$ is continuous up to $S_{1,T}$ and $u=h$ on $S_{1,T}$. 
The study is similar to the case of elliptic equations. Refer to Lemmas 3.1 and 3.2 in \cite{HanXie2024}.
\end{remark}

Throughout this section, summations over Latin letters are from $1$ to $n$
and those over Greek letters are from $1$ to $n-1$. However, an unrepeated $\alpha\in (0,1)$ is reserved for
the H\"older index.

\subsection{The Regularity of the Derivative in $x'$}\label{tan}
We first prove a decay estimate of the derivative in $x'$.

\begin{lemma}\label{3E}
For some constant $\alpha\in(0,1)$, assume $a_{ij},b_{i},c\in C^{\alpha}(\bar{Q}_{1,T})$ 
with $D_{x'}a_{ij},D_{x'}b_{i},D_{x'}c\in C^{\alpha}(\bar{Q}_{1,T})$ and \eqref{3a}. 
Suppose, for some $f\in C^{\alpha}(\bar{Q}_{1,T})$ with $D_{x'}f\in C^{\alpha}(\bar{Q}_{1,T})$, 
and $\phi\in C^{2+\alpha}(\bar{G}_{1})$ with $D_{x'}\phi\in C^{\alpha}(\bar{G}_{1})$, \eqref{3c}-\eqref{3e} 
admits a solution $u\in C(\bar{Q}_{1,T})\cap C^{2}(Q_{1,T})$. 
Then, $D_{x'}u$ is continuous up to $S_{1,T}$ and, for any $(x,t)\in Q_{1/2,T}\cup \Omega_{1/2,0}$ and $(x'_{0},0,t)\in S_{1/2,T}$,
\begin{align}\label{3j}\begin{split}
&|D_{x'}u(x,t)-D_{x'}h(x'_{0},0,t)|\\
&\qquad\leq C\big\{\|u\|_{L^{\infty}(Q_{1,T})}+\|f\|_{C^{\alpha}(\bar{Q}_{1,T})}
+\|D_{x'}f\|_{C^{\alpha}(\bar{Q}_{1,T})}\\
&\qquad\qquad+\|\phi\|_{ C^{2+\alpha}(\bar{G}_{1})}+\|D_{x'}\phi\|_{C^{\alpha}(\bar{G}_{1})}\big\}\cdot|x-(x'_{0},0)|^{\alpha},
\end{split}\end{align}
where $C$ is a positive constant depending only on $n$, $T$, $\lambda$, $\alpha$, and the $C^{\alpha}$-norms of $a_{ij},b_{i},c$ 
and $D_{x'}a_{ij},D_{x'}b_{i},D_{x'}c$ in $\bar{Q}_{1,T}$.
\end{lemma}

\begin{proof}
Fix a $k=1,\cdots,n-1$ and set
\begin{align*}
F&=\|u\|_{L^{\infty}(Q_{1,T})}+\|f\|_{C^{\alpha}(\bar{Q}_{1,T})}+\|D_{x'}f\|_{C^{\alpha}(\bar{Q}_{1,T})}\\
&\qquad+\|\phi\|_{ C^{2+\alpha}(\bar{G}_{1})}+\|D_{x'}\phi\|_{C^{\alpha}(\bar{G}_{1})}.
\end{align*}
Take any $\tau$ small and define
$$
u_{\tau}(x,t)=\frac{1}{\tau}[u(x+\tau e_{k},t)-u(x,t)],
$$
and similarly $a_{ij,\tau}$, $b_{i,\tau}$, $c_{\tau}$, and $f_{\tau}$. A straightforward computation yields
\begin{equation}\label{3k}
L(u_{\tau}(x,t)-u_{\tau}(0,t))=g_{1}=I_1+I_2,
\end{equation}
where
\begin{align*}
I_1&=-x^{2}_{n}a_{ij,\tau}\partial_{ij}u(x+\tau e_{k},t)-x_{n}b_{i,\tau}\partial_{i}u(x+\tau e_{k},t),\\
I_2&=f_{\tau}-c_{\tau}u(x+\tau e_{k},t)-cu_{\tau}(0,t)+\partial_{t}u_{\tau}(0,t).
\end{align*}
By Theorem \ref{3D}, we have, for any $(x,t)\in Q_{7/8,T}$,
\begin{equation}\label{3l}
x_{n}|D_{x}u(x,t)|+x^{2}_{n}|D^{2}_{x}u(x,t)|\leq  CF x^{\alpha}_{n}.
\end{equation}
Hence, for any $(x,t)\in Q_{3/4,T}$,
\begin{equation}\label{3m}
|I_1|\leq CFx^{\alpha}_{n}.
\end{equation}
We write $I_2$ as
\begin{align*}
I_2&=f_{\tau}(x,t)-c_{\tau}(x,t)u(x+\tau e_{k},t)-c(x,t)u_{\tau}(0,t)+\partial_{t}u_{\tau}(0,t)\\
&= f_{\tau}(x,t)-f_{\tau}(0,t)-[c_{\tau}(x,t)-c_{\tau}(0,t)]u(x+\tau e_{k},t)\\
&\qquad-c_{\tau}(0,t)[u(x+\tau e_{k},t)-u(\tau e_{k},t)]-[c(x,t)-c(0,t)]u_{\tau}(0,t)\\
&\qquad+f_{\tau}(0,t)-c_{\tau}(0,t)u(\tau e_{k},t)-c(0,t)u_{\tau}(0,t)+\partial_{t}u_{\tau}(0,t).
\end{align*}
By \eqref{3f}, we have
$$
f_{\tau}(0,t)-c_{\tau}(0,t)u(\tau e_{k},t)-c(0,t)u_{\tau}(0,t)+\partial_{t}u_{\tau}(0,t)=0.
$$
We write
$$
f_{\tau}(x,t)=\frac{1}{\tau}[f(x+\tau e_{k},t)-f(x,t)]=\int^{1}_{0}\partial_{k}f(x+s\tau e_{k},t)ds.
$$
Then,
\begin{align*}
|f_{\tau}(x,t)-f_{\tau}(0,t)|&=\Big|\int^{1}_{0}[\partial_{k}f(x+s\tau e_{k},t)-\partial_{k}f(s\tau e_{k},t)]ds\Big|\\
&\leq [D_{x'}f]_{C^{\alpha}(\bar{Q}_{1,T})}|x|^{\alpha}.
\end{align*}
Similarly,
$$
|c_{\tau}(x,t)-c_{\tau}(0,t)|\leq[D_{x'}c]_{C^{\alpha}(\bar{Q}_{1,T})}|x|^{\alpha}.
$$
By Lemma \ref{3C}, we have, for any $(x,t)\in Q_{3/4,T}$,
$$
|u(x+\tau e_{k},t)-u(\tau e_{k},t)|\leq CF|x|^{\alpha}.
$$
Therefore, for any $(x,t)\in Q_{3/4,T}$,
\begin{equation}\label{3n}
|I_2|\leq CF|x|^{\alpha}.
\end{equation}
By combining \eqref{3m} and \eqref{3n}, we obtain, for any $(x,t)\in Q_{3/4,T}$,
$$
|g_{1}(x,t)|\leq CF|x|^{\alpha}.
$$
Take a cutoff function $\varphi=\varphi(x')\in C^{\infty}_{0}(B'_{3/4})$, with $\varphi=1$ in $B'_{1/2}$. Then,
\begin{equation}\label{3o}
L(\varphi (u_{\tau}-u_{\tau}(0,\cdot)))=\varphi g_{1}+g_{2},
\end{equation}
where
$$
g_{2}=2x^{2}_{n}a_{ij}\partial_{i}\varphi\partial_{j}u_{\tau}+(x_{n} a_{ij}\partial_{ij}\varphi+b_{i}\partial_{i}\varphi)x_{n}(u_{\tau}-u_{\tau}(0,t)).
$$
Similarly, we have, for any $(x,t)\in Q_{3/4,T}$,
$$
|g_{2}(x,t)|\leq CFx^{\alpha}_{n}.
$$
We now examine $\varphi (u_{\tau}-u_{\tau}(0,t))$ on $\partial_{p} Q_{3/4,T}$. First, $\varphi (u_{\tau}-u_{\tau}(0,t))=0$ on $\partial B'_{3/4}\times (0,\frac{3}{4})\times [0,T]$. Next, the $C^{1}$-estimate interior in space and global in time (applied to the equation $Lu=f$) implies
$$
\|D_x u\|_{L^{\infty}(B'_{7/8}\times \{3/4\}\times [0,T])}\leq CF.
$$
(In fact, this $C^{1}$-estimate can be obtained by combining (10.12) in Chapter IV of \cite{L1968} and the Sobolev embedding theorem.)
Hence,
$$
\|\varphi (u_{\tau}-u_{\tau}(0,\cdot))\|_{L^{\infty}(B'_{3/4}\times \{3/4\}\times[0,T])}\leq CF.
$$
Last, for any $x'\in B'_{3/4}$ and any $t\in[0,T]$,
$$
|\varphi(x') (u_{\tau}(x',0,t)-u_{\tau}(0,t))|=|\varphi(x') (h_{\tau}(x',0,t)-h_{\tau}(0,t))|\leq CF|x'|^{\alpha},
$$
and, for any $x\in G_{3/4}$,
$$
|\varphi(x') (u_{\tau}(x,0)-u_{\tau}(0,0))|=|\varphi(x') (\phi_{\tau}(x)-\phi_{\tau}(0))|\leq CF|x|^{\alpha}.
$$
By Lemma \ref{3B} with $\mu=\alpha$ and $x'_0=0$ and the maximum principle, we have, for any $(x,t)\in Q_{3/4,T}$, 
$$
|\varphi(x') (u_{\tau}(x,t)-u_{\tau}(0,t))|\leq CF|x|^{\alpha},
$$
and hence, for any $(x,t)\in Q_{1/2,T}$,
$$
|u_{\tau}(x,t)-u_{\tau}(0,t)|\leq CF|x|^{\alpha}.
$$
Similarly, for any fixed $x'_{0}\in B'_{1/2}$, we have, for any $(x,t)\in B'_{1/2}(x'_{0})\times (0,1/2)\times(0,T]$,
$$
|u_{\tau}(x,t)-u_{\tau}(x'_{0},0,t)|\leq CF|x-(x'_{0},0)|^{\alpha},
$$
where $C$ is a positive constant independent of $\tau$. By letting $\tau\rightarrow 0$, 
we obtain, for any $(x,t)\in B'_{1/2}(x'_{0})\times (0,1/2)\times(0,T]$,
$$
|\partial_{k}u(x,t)-\partial_{k}h(x'_{0},0,t)|\leq CF|x-(x'_{0},0)|^{\alpha}.
$$
This implies \eqref{3j}. By $\partial_k h\in C^{\alpha}(S_{1,T})$, we obtain the continuity of $\partial_{k}u$ at $(x'_{0},0,t)\in S_{1/2,T}$, with
$$\partial_{k}u(x'_{0},0,t)=\partial_{k}h(x'_{0},0,t).$$
We conclude the desired result.
\end{proof}

We are ready to prove the regularity of the derivative in $x'$.

\begin{thm}\label{3F}
For some constant $\alpha\in(0,1)$, assume $a_{ij},b_{i},c\in C^{\alpha}(\bar{Q}_{1,T})$ 
with $D_{x'}a_{ij},D_{x'}b_{i},D_{x'}c\in C^{\alpha}(\bar{Q}_{1,T})$ and \eqref{3a}. 
Suppose that, for some $f\in C^{\alpha}(\bar{Q}_{1,T})$ with $D_{x'}f\in C^{\alpha}(\bar{Q}_{1,T})$, 
and $\phi\in C^{2+\alpha}(\bar{G}_{1})$ with $D_{x'}\phi\in C^{2+\alpha}(\bar{G}_{1})$, 
\eqref{3c}-\eqref{3e} admits a solution $u\in C^{2}(Q_{1,T})\cap C(\bar{Q}_{1,T})$. 
Then, for any $r\in(0,1)$, $D_{x'}u\in C^{2+\alpha}(\bar{Q}_{r,T})$ and
\begin{align*}
\|D_{x'}u\|_{C^{2+\alpha}(\bar{Q}_{1/2,T})}&\leq C\big\{\|u\|_{L^{\infty}(Q_{1,T})}+\|f\|_{C^{\alpha}(\bar{Q}_{1,T})}+\|D_{x'}f\|_{C^{\alpha}(\bar{Q}_{1,T})}\\
&\qquad+\|\phi\|_{C^{2+\alpha}(\bar{G}_{1})}+\|D_{x'}\phi\|_{C^{2+\alpha}(\bar{G}_{1})}\big\},
\end{align*}
where $C$ is a positive constant depending only on $n,$ $T,$ $\lambda,$ $\alpha,$  and the $C^{\alpha}$-norms 
of $a_{ij},b_{i},c$ and $D_{x'}a_{ij},D_{x'}b_{i},D_{x'}c$ in $\bar{Q}_{1,T}$.
\end{thm}

\begin{proof}
Fix a $k=1,\cdots,n-1$. By a simple differentiation of $Lu=f$ with respect to $x_{k}$ 
(The differentiation is valid by the interior Schauder theory along the $x'$ directions.) 
and using Theorem \ref{3D} and Lemma \ref{3E}, we have, 
for any $r\in(0,1)$, $\partial_{k}u\in C^{2}(Q_{r,T})\cap C(\bar{Q}_{r,T})$ satisfies
\begin{align}
L(\partial_{k}u)&=f_{k}\quad\text{in }Q_{r,T},\label{3p}\\
\partial_{k}u(\cdot,0)&=\partial_{k}\phi\quad\text{on } G_{r},\label{3q}\\
\partial_{k}u&=\partial_{k}h\quad\text{on }  S_{r,T},\label{3r}
\end{align}
where
$$
f_{k}=\partial_{k}f-x^{2}_{n}\partial_{k}a_{ij}\partial_{ij}u-x_{n}\partial_{k}b_{i}\partial_{i}u-\partial_{k}cu.
$$
Then, $f_k\in C^{\alpha}(\bar{Q}_{r,T})$,
$$
f_{k}=\partial_{k}f-\partial_{k}c h\quad\text{on } S_{r,T},
$$
and
\begin{align}\label{3s}\begin{split}
\partial_{t}(\partial_{k}h)-c\partial_{k}h+f_{k}&=0\quad\text{on } S_{r,T},\\
\partial_{k}h(\cdot,0,0)&=\partial_{k}\phi(\cdot,0)\quad\text{on }  B'_{r}.
\end{split}\end{align}
Thus, we can apply Theorem \ref{3D} to \eqref{3p}-\eqref{3r} and obtain the desired regularity of $\partial_{k}u$.
\end{proof}

We have the following more general result.

\begin{thm}\label{3G}
For some integer $\ell\geq 0$ and some constant $\alpha\in(0,1)$, 
assume $D^{\tau}_{x'}a_{ij},D^{\tau}_{x'}b_{i},D^{\tau}_{x'}c\in C^{\alpha}(\bar{Q}_{1,T})$ for any $\tau\leq \ell$, 
with \eqref{3a}. 
Suppose that, for some $f$ with $D^{\tau}_{x'}f\in C^{\alpha}(\bar{Q}_{1,T})$ for any $\tau\leq \ell$, and $\phi$ 
with $D^{\tau}_{x'}\phi\in C^{2+\alpha}(\bar{G}_{1})$ for any $\tau\leq \ell$, 
\eqref{3c}-\eqref{3e} admits a solution $u\in C^{2}(Q_{1,T})\cap C(\bar{Q}_{1,T})$. 
Then, for any $\tau\leq \ell$ and any $r\in(0,1)$, $D^{\tau}_{x'}u\in C^{2+\alpha}(\bar{Q}_{r,T})$ and
\begin{align*}
\|D^{\tau}_{x'}u\|_{C^{2+\alpha}(\bar{Q}_{1/2,T})}
&\leq C\Big\{\|u\|_{L^{\infty}(Q_{1,T})}+\sum^{\ell}_{i=0}\|D^{i}_{x'}f\|_{C^{\alpha}(\bar{Q}_{1,T})}\\
&\qquad\qquad+\sum^{\ell}_{i=0}\|D^{i}_{x'}\phi\|_{C^{2+\alpha}(\bar{G}_{1})}\Big\},
\end{align*}
where $C$ is a positive constant depending only on $n,$ $T,$ $\ell,$ $\lambda,$ $\alpha,$  
and the $C^{\alpha}$-norms of $D^{\tau}_{x'}a_{ij},D^{\tau}_{x'}b_{i},D^{\tau}_{x'}c$ in $\bar{Q}_{1,T}$ for $\tau\leq\ell$.
\end{thm}

\begin{proof}
The proof is based on induction. Note that Theorem \ref{3F} corresponds to $\ell=1$. 
For some fixed integer $k \geq 1$, assume Theorem \ref{3G} holds for $\ell=1, \cdots, k$ and we consider $\ell=k+1$.

By differentiating \eqref{3c}-\eqref{3e}, we have, for any $r\in(0,1)$,
\begin{align}
L(D^{k}_{x'}u)&=f_{k}\quad\text{in }Q_{r,T},\label{3t}\\
D^{k}_{x'}u(\cdot,0)&=D^{k}_{x'}\phi\quad\text{on }  G_{r},\label{3u}\\
D^{k}_{x'}u&=D^{k}_{x'}h\quad\text{on } S_{r,T},\label{3v}
\end{align}
where
\begin{align*}
f_k&=D_{x'}^k f-\big[x^2_{n} D_{x'}^k(a_{i j} \partial_{i j} u)-x^2_{n} a_{i j} D_{x'}^k \partial_{i j} u\big] \\
&\qquad-\big[x_{n} D_{x'}^k(b_i \partial_i u)-x_{n} b_i D_{x'}^k \partial_i u\big]-\big[D_{x'}^k(c u)-c D_{x'}^k u\big].
\end{align*}
It is obvious that $f_k$ is a linear combination of
$$
D_{x'}^k f,\ x^2_{n} D^2_{x} D_{x'}^\tau u,\ x_{n}D_{x} D_{x'}^\tau u,\ D_{x'}^\tau u \quad\text{for } \tau \leq k-1,
$$
with coefficients given by $a_{i j}, b_i, c$, and their $x'$-derivatives up to order $k$.
By the induction hypothesis, we have $f_{k},D_{x'}f_{k}\in C^{\alpha}(\bar{Q}_{r,T})$, for any $r\in(0,1)$, and
\begin{align*}
&\|f_{k}\|_{C^{\alpha}(\bar{Q}_{r,T})}+\|D_{x'}f_{k}\|_{C^{\alpha}(\bar{Q}_{r,T})}\\
&\qquad\leq C\Big\{\|u\|_{L^{\infty}(Q_{1,T})}+\sum^{k+1}_{i=0}\|D^{i}_{x'}f\|_{C^{\alpha}(\bar{Q}_{1,T})}
+\sum^{k}_{i=0}\|D^{i}_{x'}\phi\|_{C^{2+\alpha}(\bar{G}_{1})}\Big\}.
\end{align*}
Moreover,
$$
f_{k}=D_{x'}^k f-\left[D_{x'}^k(c h)-c D_{x'}^k h\right]\quad\text{on } S_{r,T},
$$
and
\begin{align}\label{3w}\begin{split}
\partial_{t}(D^{k}_{x'}h)-cD^{k}_{x'}h+f_{k}&=0\quad\text{on }S_{r,T},\\
D^{k}_{x'}h(\cdot,0,0)&=D^{k}_{x'}\phi(\cdot,0)\quad\text{on } B'_{r}.
\end{split}\end{align}
By applying Theorem \ref{3F} to \eqref{3t}-\eqref{3v}, we obtain, for any $r \in(0,1)$,
$D^{k+1}_{x'}u\in C^{2+\alpha}(\bar{Q}_{r,T})$, and
\begin{align*}
&\|D^{k+1}_{x'}u\|_{C^{2+\alpha}(\bar{Q}_{1/2,T})}\\
&\qquad\leq C\big\{\|D^{k}_{x'}u\|_{L^{\infty}(Q_{3/4,T})}+\|f_{k}\|_{C^{\alpha}(\bar{Q}_{3/4,T})}+\|D_{x'}f_{k}\|_{C^{\alpha}(\bar{Q}_{3/4,T})}\\
&\qquad\qquad+\|D^{k}_{x'}\phi\|_{C^{2+\alpha}(\bar{G}_{3/4})}+\|D^{k+1}_{x'}\phi\|_{C^{2+\alpha}(\bar{G}_{3/4})}\big\}\\
&\qquad\leq C\Big\{\|u\|_{L^{\infty}(Q_{1,T})}+\sum^{k+1}_{i=0}\|D^{i}_{x'}f\|_{C^{\alpha}(\bar{Q}_{1,T})}
+\sum^{k+1}_{i=0}\|D^{i}_{x'}\phi\|_{C^{2+\alpha}(\bar{G}_{1})}\Big\}.
\end{align*}
This ends the proof by induction.
\end{proof}

\subsection{The Regularity of the Derivative in $x_{n}$}\label{nor}
We first prove a decay estimate of the derivative in $x_{n}$.

\begin{lemma}\label{3H}
For some constant $\alpha\in(0,1)$, assume $a_{ij},b_{i},c\in C^{1,\alpha}(\bar{Q}_{1,T})$ with \eqref{3a}. 
Suppose, for some $f\in C^{1,\alpha}(\bar{Q}_{1,T})$ and $\phi\in C^{1,2+\alpha}(\bar{G}_{1})$,
 \eqref{3c}-\eqref{3e} admits a solution $u\in C(\bar{Q}_{1,T})\cap C^{2}(Q_{1,T})$. 
 Then, $\partial_{n}u$ is continuous up to $S_{1,T}$ and, for any $(x,t)\in Q_{1/2,T}\cup \Omega_{1/2,0}$ and $(x'_{0},0,t)\in S_{1/2,T}$,
\begin{align}\label{3x}\begin{split}
&|\partial_{n}u(x,t)-u_{1}(x'_{0},0,t)|\\
&\qquad\leq C\big\{\|u\|_{L^{\infty}(Q_{1,T})}+\|f\|_{C^{1,\alpha}(\bar{Q}_{1,T})}\\
&\qquad\qquad+\|\phi\|_{ C^{1,2+\alpha}(\bar{G}_{1})}\big\}|x-(x'_{0},0)|^{\alpha},
\end{split}\end{align}
where
\begin{align}\label{3y}\begin{split}
u_{1}(x,t)&= \partial_{n}\phi(x)\exp\Big\{\int^{t}_{0}(c+b_{n})(x,s)ds\Big\}\\
&\qquad-\int^{t}_{0}\exp\Big\{\int^{t}_{s}(c+b_{n})(x,\tau)d\tau\Big\}f_{n}(x,s) d s\quad\text{on }S_{1,T},\\
f_{n}&=\partial_{n}f-\partial_{n}c h-b_{\alpha}\partial_{\alpha}h\quad\text{on }S_{1,T},
\end{split}\end{align}
and $C$ is a positive constant depending only on $n$, $T$, $\lambda$, $\alpha$, 
and the $C^{1,\alpha}$-norms of $a_{ij}$, $b_{i}$, and $c$ in $\bar{Q}_{1,T}$.
\end{lemma}

\begin{proof}
Set
$$
F=\|u\|_{L^{\infty}(Q_{1,T})}+\|f\|_{C^{1,\alpha}(\bar{Q}_{1,T})}+\|\phi\|_{ C^{1,2+\alpha}(\bar{G}_{1})}.
$$

Step 1. We first prove a decay estimate of $u$. Consider a function linear in $x$ given by
$$
l(x,t)=l_{0}(t)+l_{\alpha}(t)x_{\alpha}+l_{n}(t)x_{n},
$$
where $l_{0},l_{1},\cdots,l_{n}$ are differentiable functions on $[0,T]$, to be determined.
Then,
$$
L(u-l)=f-x_{n}(b_{\alpha}l_{\alpha}+b_{n}l_{n})-cl_{0}-cl_{\alpha}x_{\alpha}-cl_{n}x_{n}+l'_{0}+l'_{\alpha}x_{\alpha}+l'_{n}x_{n}.
$$
For fixed $t\in [0,T]$, the linear part of the expression in $x$ in the right-hand side is given by
\begin{align*}
&l'_{0}+f(0,t)-c(0,t)l_{0}\\
&\qquad+[l'_{\alpha}+\partial_{\alpha}f(0,t)-\partial_{\alpha}c(0,t)l_{0}-c(0,t)l_{\alpha}]x_{\alpha}\\
&\qquad+[l'_{n}+\partial_{n}f(0,t)-\partial_{n}c(0,t)l_{0}-c(0,t)l_{n}-b_{\alpha}(0,t)l_{\alpha}-b_{n}(0,t)l_{n}]x_{n}.
\end{align*}
We now make this equal to zero by choosing $l_{0},l_{\alpha}$, and $l_{n}$ to satisfy the following ODEs on $[0, T]$:
\begin{align*}
l'_{0}+f(0,t)-c(0,t)l_{0}&=0,\\
l'_{\alpha}+\partial_{\alpha}f(0,t)-\partial_{\alpha}c(0,t)l_{0}-c(0,t)l_{\alpha}&=0,
\end{align*}
and
\begin{align*}
l'_{n}+\partial_{n}f(0,t)-\partial_{n}c(0,t)l_{0}-c(0,t)l_{n}-b_{\alpha}(0,t)l_{\alpha}-b_{n}(0,t)l_{n}=0.
\end{align*}
We take the initial values at $t=0$ given by
\begin{align*}
l_{0}(0)=\phi(0), \quad l_{\alpha}(0)=\partial_{\alpha}\phi(0),
\quad l_{n}(0)=\partial_{n}\phi(0).
\end{align*}
First, we have
$$
l_{0}(t)=\phi(0)e^{\int^{t}_{0}c(0,s)ds}-\int^{t}_{0}e^{\int^{t}_{s}c(0,\tau)d\tau}f(0,s) d s=h(0,t).
$$
Next, we get
$$
l_{\alpha}(t)=\partial_{\alpha}\phi(0)e^{\int^{t}_{0}c(0,s)ds}
-\int^{t}_{0}e^{\int^{t}_{s}c(0,\tau)d\tau}\big(\partial_{\alpha}f(0,s)-\partial_{\alpha}c(0,s) l_{0}(s)\big) d s.
$$
We compute the above integral involving $l_0$ separately. By a simple substitution, we have
\begin{align*}
&\int^{t}_{0}e^{\int^{t}_{s}c(0,\tau)d\tau}\partial_{\alpha}c(0,s) l_{0}(s) d s\\
&\quad=\phi(0)e^{\int^{t}_{0}c(0,\tau)d\tau}\int^{t}_{0}\partial_{\alpha}c(0,s)d s
-\int^{t}_{0}\int^{s}_{0}e^{\int^{t}_{\rho}c(0,\tau)d\tau}\partial_{\alpha}c(0,s)f(0,\rho) d\rho d s\\
&\quad=\phi(0)e^{\int^{t}_{0}c(0,\tau)d\tau}\int^{t}_{0}\partial_{\alpha}c(0,s)d s\\
&\quad\qquad-\int^{t}_{0}e^{\int^{t}_{\rho}c(0,\tau)d\tau}\Big(\int^{t}_{\rho}\partial_{\alpha}c(0,s)d s\Big) f(0,\rho)  d\rho.
\end{align*}
Hence,
$$l_{\alpha}(t)=\partial_{\alpha}h(0,t).$$
Last, we get
\begin{align*}
l_{n}(t)&=\partial_{n}\phi(0)e^{\int^{t}_{0}(c+b_{n})(0,s)ds}\\
&\quad-\int^{t}_{0}e^{\int^{t}_{s}(c+b_{n})(0,\tau)d\tau}(\partial_{n}f(0,s)-\partial_{n}c(0,s)l_{0}(s)-b_{\alpha}(0,s)l_{\alpha}(s)) d s.
\end{align*}
By substituting $l_0$ and $l_\alpha$, we obtain
\begin{align*}
l_{n}(t)=u_{1}(0,t),
\end{align*}
where $u_{1}$ is given by \eqref{3y}. Hence,
$$
l(x,t)=h(0,t)+\partial_{\alpha}h(0,t)x_{\alpha}+u_{1}(0,t)x_{n}.
$$
With such a choice of $l$, we have
$$
|L(u-l)|\leq CF|x|^{1+\alpha}\quad\text{in }Q_{1/2,T}.
$$
Moreover, we have, for any $(x',0,t)\in S_{1/2,T}$,
$$
|u(x',0,t)-l(x',0,t)|=|h(x',0,t)-h(0,t)-\partial_{\alpha}h(0,t)x_{\alpha}|\leq CF|x'|^{1+\alpha},
$$
and, for any $x\in G_{1/2}$,
$$
|u(x,0)-l(x,0)|=|\phi(x)-\phi(0)-\partial_{i}\phi(0) x_{i}|\leq CF|x|^{1+\alpha}.
$$
In general, for any fixed $x'_{0}\in B'_{1/2}$, we set
$$
l_{x'_{0}}(x,t)=h(x'_{0},0,t)+\partial_{\alpha}h(x'_{0},0,t)(x_{\alpha}-x_{0\alpha})+u_{1}(x'_{0},0,t)x_{n}.
$$
Similarly, we obtain
\begin{equation}\label{3z}
L(u-l_{x'_{0}})=g,
\end{equation}
such that, for any $(x,t)\in B'_{1/2}(x'_{0})\times(0,1/2)\times(0,T]$,
\begin{equation}\label{3aa}
|g(x,t)|\leq  CF|x-(x'_{0},0)|^{1+\alpha},
\end{equation}
for any $x'\in B'_{1/2}(x'_{0})$ and $t\in[0,T]$,
\begin{equation}\label{3ab}
|u(x',0,t)-l_{x'_{0}}(x',0,t)|\leq CF|x'-x'_{0}|^{1+\alpha},
\end{equation}
and, for any $x\in B'_{1/2}(x'_{0})\times(0,1/2)$,
\begin{equation}\label{3ac}
|u(x,0)-l_{x'_{0}}(x,0)|\leq CF|x-(x'_{0},0)|^{1+\alpha}.
\end{equation}
Next, take $w$ as in Lemma \ref{3B} with $\mu=1+\alpha$. 
By applying the maximum principle, we have, for any $(x,t)\in B'_{1/2}(x'_{0})\times (0,\frac{1}{2})\times(0,T]$,
\begin{equation}\label{3ad}
|u(x,t)-l_{x'_{0}}(x,t)|\leq CF|x-(x'_{0},0)|^{1+\alpha}.
\end{equation}
With $x'=x'_{0}$, we obtain, for any $(x,t)\in Q_{1/2,T}$,
$$
|u(x,t)-h(x',0,t)-u_{1}(x',0,t)x_{n}|\leq CFx^{1+\alpha}_{n}.
$$
Dividing by $x_{n}$ and letting $x_n\rightarrow 0$, we conclude that $\partial_{n}u(x',0,t)$ exists and is equal to $u_{1}(x',0,t)$, 
for any $x'\in B'_{1/2}$ and $t\in (0,T]$.

Step 2. We prove \eqref{3x}. We take any $x_{0}\in G_{1/2}$ and consider \eqref{3z} in $B_{3x_{0n}/4}(x_{0})\times(0,T]$. 
Set $r=3x_{0n}/4$, $M'=B_{3x_{0n}/4}(x_{0})\times(0,T]$, and $M''=B_{x_{0n}/2}(x_{0})\times(0,T]$. 
Note that, for any $(x,t)\in B_{3x_{0n}/4}(x_{0})\times(0,T]$,
$$
\frac{1}{4}x_{0n}\leq x_{n}\leq \frac{7}{4}x_{0n}.
$$
We consider
$$
v(y,t)=(u-l_{x'_{0}})(x_{0}+ry,t)\quad\text{in } Q',
$$
where $Q'=B_{1}(0)\times (0,T]$. A straightforward computation yields 
\begin{equation}\label{3ae}
\Tilde{a}_{ij}\partial_{ij}v+\Tilde{b}_{i}\partial_{i}v+\Tilde{c}v-\partial_{t}v=\Tilde{g}\quad\text{in }Q',
\end{equation}
and
$$
v(\cdot,0)=\Tilde{\psi}\quad\text{on }B_{1}(0),
$$
where
\begin{align*}
\Tilde{a}_{ij}(y,t)&=\frac{(x_{0n}+ry_{n})^{2}}{r^{2}}a_{ij}(x_{0}+ry,t),\\
\Tilde{b}_{i}(y,t)&=\frac{x_{0n}+ry_{n}}{r}b_{i}(x_{0}+ry,t),\\
\Tilde{c}(y,t)&=c(x_{0}+ry,t),
\end{align*}
and
\begin{align*}
\Tilde{g}(y,t)&=g(x_0+ry,t),\\
\Tilde{\psi}(y)&=\phi(x_{0}+ry)-\phi(x'_{0},0)-\partial_{\alpha}\phi(x'_{0},0)(ry_{\alpha})-\partial_{n}\phi(x'_{0},0)(x_{0n}+ry_{n}).
\end{align*}
By \eqref{3aa}, we have
\begin{equation}\label{3af}
\|\Tilde{g}\|_{L^{\infty}(Q')}\leq CFr^{1+\alpha}.
\end{equation}
It is easy to verify that
$$
\|\Tilde{a}_{ij}\|_{C^{\alpha}(\bar{Q}')}+\|\Tilde{b}_{i}\|_{L^{\infty}(Q')}+\|\Tilde{c}\|_{L^{\infty}(Q')}\leq C,
$$
where $C$ is a positive constant independent of $x_{0}$. Moreover, for any $\xi\in\mathbb{R}^{n}$ and any $(y,t)\in Q'$,
$$
\Tilde{a}_{ij}(y,t)\xi_{i}\xi_{j}\geq \lambda \frac{(x_{0n}+ry_{n})^{2}}{r^{2}}|\xi|^{2}\geq \frac{1}{9}\lambda |\xi|^{2}.
$$
Set $Q''=B_{2/3}(0)\times(0,T]$. 
By applying the $C^{1}$-estimate interior in space and global in time for uniformly parabolic equations to \eqref{3ae} in $Q'$,
we obtain
\begin{equation}\label{3ag}
\|v\|_{C^{1}(\bar{Q}'')}\leq C\big\{\|v\|_{L^{\infty}(Q')}+\|\Tilde{\psi}\|_{C^{2}(\bar{B}_{1}(0))}+\|\Tilde{g}\|_{L^{\infty}(Q')}\big\}.
\end{equation}
By \eqref{3ad}, we have
\begin{equation}\label{3ah}
\|v\|_{L^{\infty}(Q')}=\|u-l_{x'_{0}}\|_{L^{\infty}(M')}\leq CFr^{1+\alpha}.
\end{equation}
By $\phi\in C^{1,2+\alpha}(\bar{G}_{1})$, we have
\begin{equation}\label{3ai}
\|\Tilde{\psi}\|_{C^{2}(\bar{B}_{1}(0))}\leq  CFr^{1+\alpha}.
\end{equation}
By combining \eqref{3af}-\eqref{3ai}, we obtain
$$
r\|D_x(u-l_{x'_{0}})\|_{L^{\infty}(M'')}\leq CFr^{1+\alpha},
$$
and hence
$$
\|D_x(u-l_{x'_{0}})\|_{L^{\infty}(M'')}\leq CFr^{\alpha}.
$$
By considering $\partial_{n}$ only and evaluating at $(x_{0},t)$, we get, for any $(x,t)\in Q_{1/2,T}$,
$$
|\partial_{n}u(x,t)-u_{1}(x',0,t)|\leq CFx^{\alpha}_{n}.
$$
By $u_{1}\in C^{\alpha}(S_{1,T})$, we obtain \eqref{3x} easily. As a consequence, $\partial_{n}u$ is continuous up to $S_{1,T}$.
\end{proof}

We are ready to prove the regularity of the derivative in $x_{n}$.

\begin{thm}\label{3I}
For some constant $\alpha\in(0,1)$, assume $a_{ij},b_{i},c\in C^{1,\alpha}(\bar{Q}_{1,T})$ with \eqref{3a}. 
Suppose that, for some $f\in C^{1,\alpha}(\bar{Q}_{1,T})$ and $\phi\in C^{1,2+\alpha}(\bar{G}_{1})$, 
\eqref{3c}-\eqref{3e} admits a solution $u\in C^{2}(Q_{1,T})\cap C(\bar{Q}_{1,T})$. 
Then, for any $r\in(0,1)$, $\partial_{n}u\in C^{2+\alpha}(\bar{Q}_{r,T})$ and
$$
\|\partial_{n}u\|_{C^{2+\alpha}(\bar{Q}_{1/2,T})}\leq C\big\{\|u\|_{L^{\infty}(Q_{1,T})}+\|f\|_{C^{1,\alpha}(\bar{Q}_{1,T})}
+\|\phi\|_{C^{1,2+\alpha}(\bar{G}_{1})}\big\},
$$
where $C$ is a positive constant depending only on $n,$ $T,$ $\lambda,$ $\alpha,$  
and the $C^{1,\alpha}$-norms of $a_{ij},$ $b_{i},$ and $c$ in $\bar{Q}_{1,T}$.
\end{thm}

\begin{proof}
By a simple differentiation of $Lu=f$ with respect to $x_{n}$ 
(The differentiation is valid by the interior Schauder theory along the $x_n$ direction.) 
and using Theorem \ref{3D}, Theorem \ref{3F}, and Lemma \ref{3H}, 
we have, for any $r\in(0,1)$, $\partial_{n}u\in C^{2}(Q_{r,T})\cap C(\bar{Q}_{r,T})$ satisfies
\begin{align}
L^{(1)}(\partial_{n}u)&=f_{n}\quad\text{in }Q_{r,T},\label{3aj}\\
\partial_{n}u(\cdot,0)&=\partial_{n}\phi\quad\text{on } G_{r},\label{3ak}\\
\partial_{n}u&=u_{1}\quad\text{on }  S_{r,T},\label{3al}
\end{align}
where
\begin{align*}
L^{(1)}&=x^{2}_{n}a_{ij}\partial_{ij}+x_{n}(b_{i}+2a_{in})\partial_{i}+(c+b_{n})-\partial_{t},\\
f_{n}&=\partial_{n}f-x^{2}_{n}\partial_{n}a_{ij}\partial_{ij}u-x_{n}\partial_{n}b_{i}\partial_{i}u-\partial_{n}c u\\
&\qquad-2x_{n} a_{i\alpha}\partial_{i\alpha}u-b_{\alpha}\partial_{\alpha}u.
\end{align*}
Then, $f_n\in C^{\alpha}(\bar{Q}_{r,T})$,
$$
f_{n}=\partial_{n}f-\partial_{n}c h-b_{\alpha}\partial_{\alpha}h\quad \text{on } S_{r,T},
$$
and
\begin{align}\label{3am}\begin{split}
\partial_{t}u_{1}-(c+b_{n})u_{1}+f_{n}&=0\quad\text{on } S_{r,T},\\
u_{1}(\cdot,0,0)&=\partial_{n}\phi(\cdot,0)\quad\text{on }  B'_{r}.
\end{split}\end{align}
Thus, we can apply Theorem \ref{3D} to \eqref{3aj}-\eqref{3al} and obtain the desired regularity of $\partial_{n}u$.
\end{proof}

By combining Theorem \ref{3G} and Theorem \ref{3I}, we have the following more general result.

\begin{thm}\label{3J}
For some integer $\ell\geq 0$ and some constant $\alpha\in(0,1)$, 
assume $D^{\tau}_{x}a_{ij},D^{\tau}_{x}b_{i},D^{\tau}_{x}c\in C^{\alpha}(\bar{Q}_{1,T})$ for any $\tau\leq \ell$, with \eqref{3a}. 
Suppose that, for some
$f$ with $D^{\tau}_{x}f\in C^{\alpha}(\bar{Q}_{1,T})$ for any $\tau\leq \ell$, and $\phi\in C^{\ell,2+\alpha}(\bar{G}_{1})$, 
\eqref{3c}-\eqref{3e} admits a solution $u\in C^{2}(Q_{1,T})\cap C(\bar{Q}_{1,T})$. 
Then, for any $\tau\leq\ell$ and any $r\in(0,1)$, $D^{\tau}_{x}u\in C^{2+\alpha}(\bar{Q}_{r,T})$ and
$$
\|D^{\tau}_{x}u\|_{C^{2+\alpha}(\bar{Q}_{1/2,T})}\leq C\Big\{\|u\|_{L^{\infty}(Q_{1,T})}
+\sum^{\ell}_{i=0}\|D^{i}_{x}f\|_{C^{\alpha}(\bar{Q}_{1,T})}+\|\phi\|_{C^{\ell,2+\alpha}(\bar{G}_{1})}\Big\},
$$
where $C$ is a positive constant depending only on $n,$ $T,$ $\ell,$ $\lambda,$ $\alpha,$  and the $C^{\alpha}$-norms of
$D^{\tau}_{x}a_{ij},D^{\tau}_{x}b_{i},D^{\tau}_{x}c$ in $\bar{Q}_{1,T}$ for $\tau\leq\ell$.
\end{thm}

\begin{proof}
The proof is based on induction. 
First, we claim that, for any $\tau\leq\ell$ and any $r\in(0,1)$, $\partial^{\tau}_{n}u\in C^{2+\alpha}(\bar{Q}_{r,T}) $ and
$$
\|\partial^{\tau}_{n}u\|_{C^{2+\alpha}(\bar{Q}_{1/2,T})}\leq C\Big\{\|u\|_{L^{\infty}(Q_{1,T})}
+\sum^{\ell}_{i=0}\|D^{i}_{x}f\|_{C^{\alpha}(\bar{Q}_{1,T})}+\|\phi\|_{C^{\ell,2+\alpha}(\bar{G}_{1})}\Big\}.
$$
Note that Theorem \ref{3I} corresponds to $\ell=1$. 
For some fixed integer $k \geq 1$, assume the claim holds for $\ell=1, \cdots, k$ and we consider $\ell=k+1$.

For $\nu=0,1,\cdots,k$, set
$$
L^{(\nu)}=x^2_{n} a_{i j} \partial_{i j}+x_{n} b_i^{(\nu)} \partial_i+c^{(\nu)}-\partial_{t}u,
$$
where
\begin{align*}
&b^{(0)}_{i}=b_{i},\quad b_i^{(\nu)}=b_i^{(\nu-1)}+2 a_{i n}\quad\text{for }\nu=1,\cdots,k, \\
&c^{(0)}=c,\quad c^{(\nu)}=c^{(\nu-1)}+b_n^{(\nu-1)}\quad\text{for }\nu=1,\cdots,k.
\end{align*}
Set $f^{(0)}=f$, and for $\nu=1,\cdots,k,$
\begin{align*}
f^{(\nu)}&=\partial_n f^{(\nu-1)}-x^2_{n} \partial_n a_{i j} \partial_{i j} \partial_n^{\nu-1} u
-x_{n} \partial_n b_i^{(\nu-1)} \partial_i \partial_n^{\nu-1} u\\
&\qquad -\partial_n c^{(\nu-1)} \partial_n^{\nu-1} u 
-2 x_{n} a_{i \alpha} \partial_{i \alpha} \partial_n^{\nu-1} u-b_\alpha^{(\nu-1)} \partial_\alpha \partial_n^{\nu-1} u.
\end{align*}
For $\nu=0,1,\cdots,k$, set
\begin{align*}
h^{(\nu)}(x,t)&=\partial^{\nu}_{n}\phi(x)\exp\Big\{\int^{t}_{0}c^{(\nu)}(x,s)ds\Big\}\\
&\qquad-\int^{t}_{0}\exp\Big\{\int^{t}_{s}c^{(\nu)}(x,\tau)d\tau\Big\}f^{(\nu)}(x,s) d s.
\end{align*}
By differentiating \eqref{3c}-\eqref{3e} with respect to $x_{n}$ successively and using the induction hypothesis, we have, for any $r\in(0,1)$,
\begin{align}
L^{(k)}(\partial^{k}_{n}u)&=f^{(k)}\quad\text{in }Q_{r,T},\label{3an}\\
\partial^{k}_{n}u(\cdot,0)&=\partial^{k}_{n}\phi\quad\text{on } G_{r},\label{3ao}\\
\partial^{k}_{n}u&=h^{(k)}\quad\text{on }  S_{r,T},\label{3ap}
\end{align}
and
\begin{align}\label{3aq}\begin{split}
\partial_{t}h^{(k)}-c^{(k)}h^{(k)}+f^{(k)}&=0\quad\text{on }S_{r,T},\\
h^{(k)}(\cdot,0,0)&=\partial^{k}_{n}\phi(\cdot,0)\quad\text{on } B'_{r}.
\end{split}\end{align}
A simple induction yields
$$
b^{(k)}_{i}=b_{i}+2ka_{in},\quad c^{(k)}=c+k b_{n}+k(k-1)a_{nn},
$$
and
\begin{align*}
f^{(k)}&=\partial_n^k  f-\left[\partial_n^{k}(x^2_{n} a_{i j} \partial_{i j} u)-a_{i j} \partial_n^k(x^2_{n} \partial_{i j} u)\right] \\
&\qquad -\left[\partial_n^k(x_{n} b_i \partial_i u)-b_i \partial_n^k(x_{n} \partial_i u)\right]-\left[\partial_n^k(c u)
-c \partial_n^k u\right]-k b_\alpha \partial_\alpha \partial_n^{k-1} u \\
&\qquad-2 k x_{n} a_{i \alpha} \partial_{i \alpha} \partial_n^{k-1} u  -\sum_{(i, j) \neq(n, n)} k(k-1) a_{i j} \partial_{i j} \partial_n^{k-2} u.
\end{align*}
It is obvious that $f^{(k)}$ is a linear combination of
\begin{align*}
&\partial^{k}_{n} f,\ x^2_{n} D^2_{x} \partial^{\tau}_{n}u,\ x_{n}D_{x} \partial^{\tau}_{n}u,
\ x_{n}D_{x} \partial^{\tau}_{n}D_{x'}u,\  \partial^{\tau}_{n} u,\\
&\qquad\partial^{\tau}_{n} D_{x'}u,\ \partial^{\tau-1}_{n} D^{2}_{x'}u \quad\text{for }\tau \leq k-1,
\end{align*}
with coefficients given by $a_{i j}, b_i, c$, and their $x$-derivatives up to order $k$. 
Recall \eqref{3t}-\eqref{3w}. 
By Theorem \ref{3G} and the induction hypothesis, we have, for any $\tau\leq k$ and $r\in(0,1)$,
$$\partial^{\tau}_{n}u,\partial^{\tau}_{n} D_{x'} u,\partial^{\tau-1}_{n} D^{2}_{x'} u,\partial^{\tau-2}_{n} D^{3}_{x'} u\in C^{2+\alpha}(\bar{Q}_{r,T}),$$
and
\begin{align*}
&\|\partial^{\tau}_{n}u\|_{C^{2+\alpha}(\bar{Q}_{r,T})}+\|\partial^{\tau}_{n}D_{x'}u\|_{C^{2+\alpha}(\bar{Q}_{r,T})}
+\|\partial^{\tau-1}_{n}D^{2}_{x'}u\|_{C^{2+\alpha}(\bar{Q}_{r,T})}\\
&\qquad\qquad+\|\partial^{\tau-2}_{n}D^{3}_{x'}u\|_{C^{2+\alpha}(\bar{Q}_{r,T})}\\
&\qquad\leq C\Big\{\|u\|_{L^{\infty}(Q_{1,T})}+\sum^{k+1}_{i=0}\|D^{i}_{x}f\|_{C^{\alpha}(\bar{Q}_{1,T})}+\|\phi\|_{C^{k+1,2+\alpha}(\bar{G}_{1})}\Big\}.
\end{align*}
Hence, $f^{(k)}\in  C^{1,\alpha}(\bar{Q}_{r,T})$, for any $r\in(0,1)$, and
$$
\|f^{(k)}\|_{C^{1,\alpha}(\bar{Q}_{r,T})}\leq C\Big\{\|u\|_{L^{\infty}(Q_{1,T})}
+\sum^{k+1}_{i=0}\|D^{i}_{x}f\|_{C^{\alpha}(\bar{Q}_{1,T})}+\|\phi\|_{C^{k+1,2+\alpha}(\bar{G}_{1})}\Big\}.
$$
For an illustration, we consider $x^2_{n} D^2_{x} \partial^{\tau}_{n}u$ for any $\tau\leq k-1$. Note that
\begin{align*}
D_{x'}(x^2_{n} D^2_{x} \partial^{\tau}_{n}u)&=x^2_{n} D^2_{x}( \partial^{\tau}_{n}D_{x'}u),\\
\partial_n(x^2_{n} D_{x}D_{x'} \partial^{\tau}_{n}u)&=x^2_{n} D_{x} D_{x'} (\partial^{\tau+1}_{n}u)+2x_n D_x(\partial^{\tau}_{n}D_{x'}u),
\end{align*}
and
$$
\partial_n(x^2_{n} D_{x}\partial_n\partial^{\tau}_{n}u)=x^2_{n} D_{x}\partial_n(\partial^{\tau+1}_{n}u)+2x_n D_{x}(\partial^{\tau+1}_{n}u).
$$
Thus, we have  $x^2_{n} D^2_{x} \partial^{\tau}_{n}u\in  C^{1,\alpha}(\bar{Q}_{r,T})$, 
for any $\tau\leq k-1$ and any $r\in(0,1)$. We discuss other terms in $f^{(k)}$ similarly.
By applying Theorem \ref{3I} to \eqref{3an}-\eqref{3ap}, we obtain, for any $r \in(0,1)$,
$\partial^{k+1}_{n}u\in C^{2+\alpha}(\bar{Q}_{r,T})$ and
\begin{align*}
&\|\partial^{k+1}_{n}u\|_{C^{2+\alpha}(\bar{Q}_{1/2,T})}\\
&\qquad \leq C\big\{\|\partial^{k}_{n}u\|_{L^{\infty}(Q_{3/4,T})}+\|f^{(k)}\|_{C^{1,\alpha}(\bar{Q}_{3/4,T})}
+\|\partial^{k}_{n}\phi\|_{C^{1,2+\alpha}(\bar{G}_{3/4})}\big\}\\
&\qquad\leq C\Big\{\|u\|_{L^{\infty}(Q_{1,T})}+\sum^{k+1}_{i=0}\|D^{i}_{x}f\|_{C^{\alpha}(\bar{Q}_{1,T})}
+\|\phi\|_{C^{k+1,2+\alpha}(\bar{G}_{1})}\Big\}.
\end{align*}
Hence, the claim holds by induction.

Similarly, by induction and Theorem \ref{3G}, we can prove, for any nonnegative integers $\tau$ 
and $\sigma$ with $\tau+\sigma\leq k+1$, and any $r\in(0,1)$, $\partial^{\tau}_{n}D^{\sigma}_{x'}u\in C^{2+\alpha}(\bar{Q}_{r,T})$ and
\begin{align*}
\|\partial^{\tau}_{n}D^{\sigma}_{x'}u\|_{C^{2+\alpha}(\bar{Q}_{1/2,T})}
&\leq C\Big\{\|u\|_{L^{\infty}(Q_{1,T})}
+\sum^{k+1}_{i=0}\|D^{i}_{x}f\|_{C^{\alpha}(\bar{Q}_{1,T})}\\
&\qquad+\|\phi\|_{C^{k+1,2+\alpha}(\bar{G}_{1})}\Big\}.
\end{align*}
We obtain the desired result.
\end{proof}

\subsection{The Regularity of the Derivative in $t$}\label{time}
By Theorem \ref{3F} and Theorem \ref{3I}, we can directly obtain the following results.

\begin{thm}\label{3K}
For some constant $\alpha\in(0,1)$, assume $a_{ij},b_{i},c\in C^{1,\alpha}(\bar{Q}_{1,T})$ with \eqref{3a}. 
Suppose that, for some $f\in C^{1,\alpha}(\bar{Q}_{1,T})$ and $\phi\in C^{1,2+\alpha}(\bar{G}_{1})$, 
\eqref{3c}-\eqref{3e} admits a solution $u\in C^{2}(Q_{1,T})\cap C(\bar{Q}_{1,T})$. 
Then, for any $r\in(0,1)$, $u\in C^{1,2+\alpha}(\bar{Q}_{r,T})$ and
$$
\|u\|_{C^{1,2+\alpha}(\bar{Q}_{1/2,T})}\leq C\big\{\|u\|_{L^{\infty}(Q_{1,T})}+\|f\|_{C^{1,\alpha}(\bar{Q}_{1,T})}
+\|\phi\|_{C^{1,2+\alpha}(\bar{G}_{1})}\big\},
$$
where $C$ is a positive constant depending only on $n,$ $T,$ $\lambda,$ $\alpha,$  
and the $C^{1,\alpha}$-norms of $a_{ij},$ $b_{i},$ and $c$ in $\bar{Q}_{1}$.
\end{thm}

By combining \eqref{3c} and Theorem \ref{3J}, we can prove the regularity of the derivative in $t$.

\begin{thm}\label{3L}
For some constant $\alpha\in(0,1)$, assume $a_{ij},b_{i},c\in C^{2,\alpha}(\bar{Q}_{1,T})$ with \eqref{3a}. 
Suppose that, for some $f\in C^{2,\alpha}(\bar{Q}_{1,T})$ and $\phi\in C^{2,2+\alpha}(\bar{G}_{1})$, 
\eqref{3c}-\eqref{3e} admits a solution $u\in C^{2}(Q_{1,T})\cap C(\bar{Q}_{1,T})$. 
Then, for any $r\in(0,1)$, $u\in C^{2,2+\alpha}(\bar{Q}_{r,T})$ and
$$
\|u\|_{C^{2,2+\alpha}(\bar{Q}_{1/2,T})}\leq C\big\{\|u\|_{L^{\infty}(Q_{1,T})}+\|f\|_{C^{2,\alpha}(\bar{Q}_{1,T})}
+\|\phi\|_{C^{2,2+\alpha}(\bar{G}_{1})}\big\},
$$
where $C$ is a positive constant depending only on $n,$ $T,$ $\lambda,$ $\alpha,$  
and the $C^{2,\alpha}$-norms of $a_{ij},$ $b_{i},$ and $c$ in $\bar{Q}_{1,T}$.
\end{thm}

\begin{proof}
By the Schauder theory for uniformly parabolic equations, we have $u\in C^{4,\alpha}(Q_{1,T})$. 
We need to consider the $C^{\alpha}$-norm of $\partial^2_{t}u$. By $Lu=f$, we get
$$
\partial_{t}u=x^{2}_{n} a_{ij}\partial_{ij}u+x_{n} b_{i}\partial_{i}u+cu-f,
$$
and 
\begin{align*}
\partial_{t}(\partial_{t}u)&=x^{2}_{n} \partial_{t}a_{ij}\partial_{ij}u+x_{n} \partial_{t}b_{i}\partial_{i}u+\partial_{t}cu-\partial_{t}f\\
&\qquad +x^{2}_{n} a_{ij}\partial_{ijt}u+x_{n}b_{i}\partial_{it}u+c\partial_{t}u.
\end{align*}
Then, by Theorem \ref{3J}, we have, for any $r\in(0,1)$, $\partial^2_{t}u\in C^{\alpha}(\bar{Q}_{r,T})$ and
$$
\|\partial^2_{t}u\|_{ C^{\alpha}(\bar{Q}_{1/2,T})}\leq C\big\{\|u\|_{L^{\infty}(Q_{1,T})}+\|f\|_{C^{2,\alpha}(\bar{Q}_{1,T})}
+\|\phi\|_{C^{2,2+\alpha}(\bar{G}_{1})}\big\}.
$$
We obtain the desired result.
\end{proof}

By induction, we have the following more general result.

\begin{thm}\label{3M}
For some integer $\ell\geq 0$ and some constant $\alpha\in(0,1)$, assume $a_{ij},b_{i},c\in C^{\ell,\alpha}(\bar{Q}_{1,T})$ with \eqref{3a}. 
Suppose that, for some $f\in C^{\ell,\alpha}(\bar{Q}_{1,T})$ and $\phi\in C^{\ell,2+\alpha}(\bar{G}_{1})$,
\eqref{3c}-\eqref{3e} admits a solution $u\in C^{2}(Q_{1,T})\cap C(\bar{Q}_{1,T})$. 
Then, for any $r\in(0,1)$, $u\in C^{\ell,2+\alpha}(\bar{Q}_{r,T})$ and
$$
\|u\|_{C^{\ell,2+\alpha}(\bar{Q}_{1/2,T})}\leq C\big\{\|u\|_{L^{\infty}(Q_{1,T})}+\|f\|_{C^{\ell,\alpha}(\bar{Q}_{1,T})}
+\|\phi\|_{C^{\ell,2+\alpha}(\bar{G}_{1})}\big\},
$$
where $C$ is a positive constant depending only on $n,$ $T,$ $\ell,$ $\lambda,$ $\alpha,$  
and the $C^{\ell,\alpha}$-norms of $a_{ij},$ $b_{i},$ and $c$ in $\bar{Q}_{1,T}$.
\end{thm}

\begin{proof}
We will only discuss the case that $\ell$ is an even number. 
The discussion is similar if $\ell$ is an odd number.  
By the Schauder theory for uniformly parabolic equations, we have $u\in C^{\ell+2,\alpha}(Q_{1,T})$. We need to prove, for any nonnegative integers $\tau$ and $\nu$ with $\tau+2\nu\leq\ell+2$ and $\nu>1$, and any $r\in(0,1)$,
\begin{equation}\label{3ar}
D^{\tau}_{x}\partial^{\nu}_{t}u\in C^{\alpha}(\bar{Q}_{r,T}),
\end{equation}
and
\begin{equation}
\begin{aligned}\label{3as}
&\|D^{\tau}_{x}\partial^{\nu}_{t}u\|_{C^{\alpha}(\bar{Q}_{1/2,T})}\\
&\qquad\leq C\big\{\|u\|_{L^{\infty}(Q_{1,T})}
+\|f\|_{C^{\ell,\alpha}(\bar{Q}_{1,T})}+\|\phi\|_{C^{\ell,2+\alpha}(\bar{G}_{1})}\big\}.
\end{aligned}
\end{equation}
First, we consider $\nu=2$ and $\tau\leq\ell-2$. Take $\beta\in\mathbb{Z}^{n}_{+}$ with $|\beta|=\tau$. 
By the discussion in Subsection \ref{tan} and Subsection \ref{nor} and using 
Theorem \ref{3J}, 
we know that $D^{\beta}_{x}u$ satisfies a problem similar to \eqref{3c}-\eqref{3e} 
and the assumption of Theorem \ref{3L} is satisfied. 
By applying Theorem \ref{3L} to $D^{\beta}_{x}u$, we have, 
for any $r\in(0,1)$, $D^{\beta}_{x}\partial^2_{t}u\in C^{\alpha}(\bar{Q}_{r,T})$ and \eqref{3as} holds.

We fix an integer $k$ with $3\leq k\leq m$, for $m=\ell/2+1$.
We assume \eqref{3ar} and \eqref{3as} for any $2\leq\nu\leq k-1$ and $\tau\leq\ell+2-2\nu$, 
and proceed to prove \eqref{3ar} and \eqref{3as} for $\nu=k$ and $\tau\leq\ell+2-2k$. 
By $Lu=f$, we have, for any $0\leq \sigma\leq \tau$,
$$
D^{\tau-\sigma}_{x'}\partial^{\sigma}_{n}\partial^{\nu}_{t}u
=D^{\tau-\sigma}_{x'}\partial^{\sigma}_{n}\partial^{\nu-1}_{t}(x^{2}_{n} a_{ij}\partial_{ij}u+x_{n} b_{i}\partial_{i}u+cu-f).
$$
A straightforward computation yields
\begin{align*}
D^{\tau-\sigma}_{x'}\partial^{\sigma}_{n}\partial^{\nu-1}_{t}(x^{2}_{n}\partial_{ij}u)
&=x^{2}_{n}D^{\tau-\sigma}_{x'}\partial_{ij}\partial^{\sigma}_{n}\partial^{\nu-1}_{t}u
+2\sigma x_{n}D^{\tau-\sigma}_{x'}\partial_{ij}\partial^{\sigma-1}_{n}\partial^{\nu-1}_{t}u\\
&\qquad+\sigma(\sigma-1)D^{\tau-\sigma}_{x'}\partial_{ij}\partial^{\sigma-2}_{n}\partial^{\nu-1}_{t}u,
\end{align*}
and
$$
D^{\tau-\sigma}_{x'}\partial^{\sigma}_{n}\partial^{\nu-1}_{t}(x_{n}\partial_{i}u)
=x_{n}D^{\tau-\sigma}_{x'}\partial_{i}\partial^{\sigma}_{n}\partial^{\nu-1}_{t}u
+\sigma D^{\tau-\sigma}_{x'}\partial_{i}\partial^{\sigma-1}_{n}\partial^{\nu-1}_{t}u.
$$
By \eqref{3ar} and \eqref{3as} for $2\leq \nu\leq k-1$ and $\tau\leq \ell+2-2\nu$ and Theorem \ref{3J}, 
we establish \eqref{3ar} and \eqref{3as} for $D^{\tau-\sigma}_{x'}\partial^{\sigma}_{n}\partial^{\nu}_{t}u$,
 for $2\leq \nu\leq k$, $\tau\leq \ell+2-2\nu$, and $0\leq \sigma\leq \tau$. 
 Hence, we obtain \eqref{3ar} and \eqref{3as} for $2\leq \nu\leq k$ and $\tau\leq \ell+2-2\nu$.
\end{proof}

Theorem \ref{3M} implies Theorem \ref{1B} easily.

\section{The Convergence of Solutions}\label{sec-Convergence}

In this section, we discuss solutions $u=u(x,t)$ for large $t$.
We first study the convergence of the boundary values, and then discuss the $L^{\infty}$-convergence, 
the $C^{2+\alpha}_{\ast}$-convergence, and the $C^{k,2+\alpha}_{\ast}$-convergence of solutions in sequence. 
Based on Sections 2 and 3, we will use standard interior estimates for parabolic equations, not estimates global in time.

First, we prove the following lemma concerning the 
convergence of the boundary values.

\begin{lemma}\label{4A}
Let $\Omega$ be a bounded domain in $\mathbb{R}^{n}$ with a $C^{1}$-boundary $\partial\Omega$ 
and $S=\partial\Omega\times[0,\infty)$. 
Let $c,f\in C(S)\cap L^{\infty}(S)$ and $\bar{c},\bar{f},\phi\in C(\partial\Omega)$, with $\bar{c}<0$ on $\partial\Omega$. 
Define, for any $(x,t)\in S$,
$$
h(x,t)=\phi(x)\exp\left\{\int^{t}_{0}c(x,s)ds\right\}-\int^{t}_{0}\exp\left\{\int^{t}_{s}c(x,\tau)d\tau\right\}f(x,s) d s.
$$


$\mathrm{(i)}$  Assume $c(\cdot,t),f(\cdot,t)$ converge to $\bar{c},\bar{f}$ in $L^{\infty}(\partial\Omega)$ 
as $t\rightarrow\infty$, respectively, then
$h(\cdot,t)$ converges to $\bar{f}/\bar{c}$ in $L^{\infty}(\partial\Omega)$ as $t\rightarrow\infty$.

$\mathrm{(ii)}$  For some constant $\alpha\in(0,1)$, assume $c,f\in C^{\alpha}(S)$ 
and $\bar{c},\bar{f},\phi\in C^{\alpha}(\partial\Omega)$. 
Assume $c(\cdot,t),f(\cdot,t)$ converge to $\bar{c},\bar{f}$ in $C^{\alpha}(\partial\Omega)$ as $t\rightarrow\infty$, respectively, then
$h(\cdot,t)$ converges to $\bar{f}/\bar{c}$ in $C^{\alpha}(\partial\Omega)$ as $t\rightarrow\infty$.

$\mathrm{(iii)}$ Assume $c(\cdot,t),f(\cdot,t)$ converge to $\bar{c},\bar{f}$ 
in $L^{\infty}(\partial\Omega)$ as $t\rightarrow\infty $, respectively, 
then $\partial_{t}h(\cdot,t)$ converges to $0$ in $L^{\infty}(\partial\Omega)$ as $t\rightarrow\infty$. 
In particular, $[h(x,\cdot)]_{C^{0,1}([T,T+1])}$ converges to $0$ uniformly in $x\in\partial\Omega$, as $T\rightarrow\infty$.
\end{lemma}

\begin{proof}
(i) Fix an $\varepsilon>0$. By assumptions, we can take positive constants $\sigma$ and $c_{0}$ such that
\begin{align*}
&c(x,t)<-c_0,\\
&|c(x,t)-\bar{c}(x)|<\varepsilon,\quad|f(x,t)-\bar{f}(x)|<\varepsilon\quad\text{for any }(x,t)\in S\backslash S_{\sigma}.
\end{align*}
Here, $S_\sigma=\partial\Omega\times [0, \sigma]$ as introduced before. 
Note that $c_{0}$ does not need to depend on $\varepsilon$.
We write
\begin{align*}
h(x,t)=J_1+J_2+J_3,
\end{align*}
where
\begin{align*}
J_1&=\phi(x)e^{\int^{t}_{0}c(x,s)ds},\\
J_2&=-\int^{\sigma}_{0}e^{\int^{t}_{s}c(x,\tau)d\tau}f(x,s) d s,\\
J_3&=-\int^{t}_{\sigma}e^{\int^{t}_{s}c(x,\tau)d\tau}f(x,s) d s.
\end{align*}
Then, for any $t\geq\sigma$, we have
\begin{align*}
|J_1|&=|\phi(x)|e^{\int^{\sigma}_{0}c(x,s)ds}e^{\int^{t}_{\sigma}c(x,s)ds}
\leq Ce^{-c_{0}(t-\sigma)}\rightarrow 0\quad\text{as }t\rightarrow\infty,\\
|J_2|&=\Big|\int^{\sigma}_{0}e^{\int^{\sigma}_{s}c(x,\tau)d\tau}e^{\int^{t}_{\sigma}c(x,\tau)d\tau}f(x,s) d s\Big|
\leq Ce^{-c_{0}(t-\sigma)}\quad\text{as }t\rightarrow\infty.
\end{align*}
Hence, there exists a constant $t_{0}\geq\sigma$ such that, for any $(x,t)\in S\backslash S_{t_{0}}$, 
\begin{equation}\label{j1j2}
|J_1|+|J_2|<\varepsilon.
\end{equation}
Now we consider $J_3$. For any $t\geq\sigma$, we write
$$
J_3=I_{1}+I_{2}+I_{3},
$$
where
\begin{align*}
I_{1}&=\int^{t}_{\sigma}e^{\int^{t}_{s}c(x,\tau)d\tau}[\bar{f}(x)-f(x,s)]d s,\\
I_{2}&=\bar{f}(x)\int^{t}_{\sigma}\Big[e^{\int^{t}_{s}\bar{c}(x)d\tau}-e^{\int^{t}_{s}c(x,\tau)d\tau}\Big]ds,\\
I_{3}&=-\bar{f}(x)\int^{t}_{\sigma}e^{\int^{t}_{s}\bar{c}(x)d\tau}ds.
\end{align*}
For $I_{1}$, we have
\begin{equation}\label{4c}
|I_{1}|\leq \varepsilon\int^{t}_{\sigma}e^{-c_{0}(t-s)}ds\leq \frac{\varepsilon}{c_{0}}.
\end{equation}
For $I_{2}$, we get
\begin{align*}
I_{2}&=\bar{f}(x)\int^{t}_{\sigma}\int^{1}_{0}\frac{d}{d\rho}e^{\rho\int^{t}_{s}\bar{c}(x)d\tau+(1-\rho)\int^{t}_{s}c(x,\tau)d\tau}d\rho d s\\
&=\bar{f}(x)\int^{t}_{\sigma}\int^{1}_{0}e^{\rho\int^{t}_{s}\bar{c}(x)d\tau+(1-\rho)\int^{t}_{s}c(x,\tau)d\tau}
\Big[\int^{t}_{s}(\bar{c}(x)-c(x,\tau))d\tau\Big]d\rho d s.
\end{align*}
Hence,
\begin{equation}\label{4d}
|I_{2}|\leq C\varepsilon\int^{t}_{\sigma}(t-s)e^{-c_{0}(t-s)}ds\leq C\varepsilon.
\end{equation}
For $I_{3}$, we have
$$
\Big|I_{3}-\frac{\bar{f}}{\bar{c}}\Big|
=\Big|\frac{\bar{f}}{\bar{c}}\Big|e^{\bar{c}(t-\sigma)}\leq \frac{C}{c_{0}}e^{-c_{0}(t-\sigma)}\rightarrow 0\quad\text{as }t\rightarrow\infty.
$$
Hence, there exists a constant $t_{1}\geq\sigma$ such that, for any $(x,t)\in S\backslash S_{t_{0}}$, 
\begin{equation}\label{4e}
\Big|I_{3}-\frac{\bar{f}}{\bar{c}}\Big|<\varepsilon.
\end{equation}
By combining \eqref{j1j2}-\eqref{4e}, we have, for any $t\geq \max\{t_0,t_1\}$,
$$
\Big\|h(\cdot,t)-\frac{\bar{f}}{\bar{c}}(\cdot)\Big\|_{L^{\infty}(\partial\Omega)}\leq C\varepsilon,
$$
where $C$ is a positive constant independent of $\varepsilon$. This finishes the proof of (i).

The proof of (ii) is similar to that of (i) but more complicated. We omit the proof.

(iii) By (i), we have $h(\cdot,t)$ converging to $\bar{f}/\bar{c}$ in $L^{\infty}(\partial\Omega)$ as $t\rightarrow\infty$. 
By \eqref{1l}, we have, for any $(x,t)\in S$,
\begin{align*}
\partial_{t}h(x,t)&=c(x,t)h(x,t)-f(x,t)\\
&=c(x,t)\Big(h(x,t)-\frac{\bar{f}(x)}{\bar{c}(x)}\Big)+\big(c(x,t)-\bar{c}(x)\big)\frac{\bar{f}(x)}{\bar{c}(x)}\\
&\qquad+(\bar{f}(x)-f(x,t)).
\end{align*}
Hence, we conclude that $\partial_{t}h(\cdot,t)$ converges to $0$ in $L^{\infty}(\partial\Omega)$ as $t\rightarrow\infty$.
We have the desired result. 
\end{proof}

\subsection{$L^{\infty}$-Convergence}\label{infty-convergence}

We now derive the $L^{\infty}$-decay of solutions of the initial-boundary value problem 
by modifying the proof of Theorem 1 in Chapter 6 of \cite{F1964}.

\begin{thm}\label{4C}
Let $\Omega$ be a bounded domain in $\mathbb{R}^n$, with a $C^{1}$-boundary $\partial\Omega$ 
and a $C^{1}(\bar{\Omega})$-defining function $\rho$. 
Assume that $a_{ij},b_{i},c,f\in C(\bar{Q})\cap L^{\infty}(Q)$ and $\bar{c}\in C(\bar{\Omega})$, with \eqref{1g} and \eqref{1o}, 
and that $c(\cdot,t)$ converges to $\bar{c}$ in $L^{\infty}(\Omega)$ 
and $f(\cdot,t)$ converges to $0$ in $L^{\infty}(\Omega)$ as $t\rightarrow\infty$. 
Suppose, for some $\phi\in C(\bar{\Omega})$, the  initial-boundary value problem \eqref{1i}-\eqref{1k} 
admits a solution $u\in  C^{2}(Q)\cap C(\bar{Q})$. 
Then, $u(\cdot,t)$ converges to $0$ in $L^{\infty}(\Omega)$ as $t\rightarrow\infty$.
\end{thm}

\begin{proof}
Without loss of generality, we assume $$\Omega\subset \{0<x_{1}<r\},$$
for some constant $r>0$. For some constant $\mu>0$ to be chosen later, set
$$
z(x)=r^{2}(e^{2\mu}-e^{\frac{\mu x_{1}}{r}}).
$$
A straightforward calculation yields
$$
L z(x)=-e^{\frac{\mu x_{1}}{r}}[\rho^{2}a_{11}\mu^{2}+r\rho b_{1}\mu-cr^{2}(e^{\mu(2-\frac{x_{1}}{r})}-1)].
$$
We claim that, by choosing $\mu>0$ and $T_0>0$ sufficient large, 
\begin{equation}\label{4i}
Lz\leq -1\quad\text{in }Q\backslash Q_{T_{0}}.
\end{equation}
We first use the claim to prove the desired result, and then prove the claim.
Fix an $\varepsilon>0$. By Lemma \ref{4A}(i) and choosing $T_{0}$ larger if necessary, we have
\begin{equation}\label{4j}
|f|<\varepsilon\quad\text{in }Q\backslash Q_{T_{0}},
\end{equation}
and
\begin{equation}\label{4k}
|h|<\varepsilon\quad\text{on }S\backslash S_{T_{0}}.
\end{equation}
For some positive constants $A$, $B$, and $H$ to be determined, we set
$$
\psi(x,t)=A\varepsilon z(x)+B z(x)e^{-H(t-T_{0})}\quad\text{in }Q\backslash Q_{T_{0}}.
$$
By \eqref{4i}, we have
$$
L\psi\leq -A\varepsilon-Be^{-H(t-T_{0})}+BH z e^{-H(t-T_{0})}\quad\text{in }Q\backslash Q_{T_{0}}.
$$
We take $H=1/\sup_{\Omega}z>0$. Then, 
$$
L\psi\leq -A\varepsilon\quad\text{in }Q\backslash Q_{T_{0}}.
$$
Note that
$$
\psi\geq A\varepsilon\inf_{\Omega}z \quad\text{on }S\backslash S_{T_{0}},
$$
and
$$
\psi\geq B\inf_{\Omega}z\quad\text{on }\Omega_{T_{0}}=\Omega\times\{T_{0}\}.
$$
By Lemma \ref{2B}, we have
\begin{equation}\label{4l}
|u(x,t)|\leq C_{0}\quad\text{in }Q_{T_{0}}.
\end{equation}
We now take $A=\max\{1,1/\inf_{\Omega}z\}$ and $B=C_{0}/\inf_{\Omega}z$. Then, 
$$
L(\pm u)=\pm f\geq L\psi \quad\text{in }Q\backslash Q_{T_{0}},
$$
and
$$
\pm u\leq \psi\quad\text{on }\partial_{p}(Q\backslash Q_{T_{0}}).
$$
By the maximum principle, we obtain
$$
|u|\leq \psi\quad\text{in }Q\backslash Q_{T_{0}}.
$$
Hence,
$$
\|u(\cdot,t)\|_{L^{\infty}(\Omega)}\leq (A\varepsilon+Be^{-H(t-T_{0})})\sup_{\Omega}z \leq 2A\varepsilon\sup_{\Omega}z,
$$
if
$$
t\geq \max\Big\{T_{0},T_{0}-\frac{1}{H}\log\frac{A\varepsilon}{B}\Big\}.
$$
This yields the desired result.

We now prove \eqref{4i}.  
By $\bar{c}<0$ on $\partial\Omega$ and $c(\cdot,t)$ converging to $\bar{c}$ in $L^{\infty}(\Omega)$ as $t\rightarrow\infty$, 
we can take positive constants $\rho_{0}, c_{0}$, and $T_{1}$ such that
\begin{equation}\label{4f}
c\leq -c_{0}\quad\text{in }\{0\leq \rho\leq\rho_{0}\}\cap Q\backslash Q_{T_{1}}.
\end{equation}
Next, we choose $\mu$ sufficient large such that
\begin{equation}\label{mu}
r\rho b_{1}\mu+c_{0}r^{2}(e^{\mu}-1)\geq 1,\quad \rho^{2}_{0}\lambda\mu^{2}+r\rho b_{1}\mu\geq 2\quad\text{in }Q.
\end{equation}
We fix such a $\mu$. Then, by \eqref{1o}, we take $T_0>T_1$ such that
\begin{equation}\label{limc}
cr^{2}(e^{\mu(2-\frac{x_{1}}{r})}-1)\leq 1\quad\text{in }Q\backslash Q_{T_{0}}.
\end{equation}
By \eqref{4f} and \eqref{mu}, we have, for any $(x,t)\in\{0\leq \rho\leq\rho_{0}\}\cap Q\backslash Q_{T_{0}}$,
$$
\rho^{2}a_{11}\mu^{2}+r\rho b_{1}\mu-cr^{2}(e^{\mu(2-\frac{x_{1}}{r})}-1)\geq r\rho b_{1}\mu+c_{0}r^{2}(e^{\mu}-1)\geq 1.
$$
Next, by \eqref{limc} and \eqref{mu}, we get, for any $(x,t)\in\{ \rho\geq\rho_{0}\}\cap Q\backslash Q_{T_{0}}$,
$$
\rho^{2}a_{11}\mu^{2}+r\rho b_{1}\mu-cr^{2}(e^{\mu(2-\frac{x_{1}}{r})}-1)\geq \rho^{2}_{0}\lambda\mu^{2}+r\rho b_{1}\mu-1\geq 1.
$$
This finishes the proof of \eqref{4i}.
\end{proof}

\begin{remark}
We further assume that $c\leq 0$ in $Q$. 
Similar to the proof of Corollaries 4 and 5 in Chapter 6 of \cite{F1964}, 
we can establish explicit convergence rates of the solutions.

$\mathrm{(i)}$ If $\|f(\cdot,t)\|_{L^{\infty}(\Omega)}=O((1+t)^{-\mu})$ 
and $\|h(\cdot,t)\|_{L^{\infty}(\Omega)}=O((1+t)^{-\mu})$ as $t\rightarrow\infty$, for some positive constant $\mu$, 
then $\|u(\cdot,t)\|_{L^{\infty}(\Omega)}=O((1+t)^{-\mu})$.  

$\mathrm{(ii)}$ If $\|f(\cdot,t)\|_{L^{\infty}(\Omega)}=O(e^{-\mu t})$ 
and $\|h(\cdot,t)\|_{L^{\infty}(\Omega)}=O(e^{-\mu t})$ as $t\rightarrow\infty$, for some positive constant $\mu$, 
then $\|u(\cdot,t)\|_{L^{\infty}(\Omega)}=O(e^{-\nu t})$, for some positive constant $\nu\leq\mu$. 
\end{remark}

Note that the parabolic operator $L$ in Theorem \ref{4C} does not need to have a limit elliptic operator as $t\rightarrow\infty$.

As a direct consequence, we obtain an $L^{\infty}$-convergence result.

\begin{corollary}\label{4D}
Let $\Omega$ be a bounded domain in $\mathbb{R}^{n}$, 
with a $C^{1}$-boundary $\partial\Omega$ and a $C^{1}(\bar{\Omega})$-defining function $\rho$. 
Assume that $a_{ij},b_{i},c,f\in C(\bar{Q})\cap L^{\infty}(Q)$ and $\bar{a}_{ij},\bar{b}_{i},\bar{c},\bar{f}\in C(\bar{\Omega})$, 
with \eqref{1g} and \eqref{1o}, 
and that $a_{i j}(\cdot,t)$, $b_{i}(\cdot,t)$, $c(\cdot,t)$, $f(\cdot,t)$ 
converge to $\bar{a}_{i j},\bar{b}_{i},\bar{c},\bar{f}$ in $L^{\infty}(\Omega)$ as $t\rightarrow\infty$, respectively. 
Suppose that, for some $\phi\in C(\bar{\Omega})$, the initial-boundary value problem \eqref{1i}-\eqref{1k} 
admits a solution $u\in C^{2}(Q)\cap C(\bar{Q})$ and the Dirichlet problem \eqref{1m}-\eqref{1n} 
admits a solution $v\in C^{2}(\Omega)\cap C(\bar{\Omega})$, with $\rho Dv,\rho^{2}D^{2}v\in L^{\infty}(\Omega)$. 
Then, $u(\cdot,t)$ converges to $v$ in $L^{\infty}(\Omega)$ as $t\rightarrow\infty$.
\end{corollary}

\begin{proof}
Let $w=u-v$ in $Q$. Then, $w$ satisfies the initial-boundary value
problem
\begin{align}
Lw&=f_{0}\quad\text{in }Q,\label{4m}\\
w(\cdot,0)&=\phi_{0}=\phi-v\quad\text{on }\Omega,\label{4n}\\
w&=h_{0}=h-\bar{f}/\bar{c}\quad\text{on }S,\label{4o}
\end{align}
where
\begin{align*}
f_{0}&=Lu-L_{0}v+(L_{0}-L)v\\
&=f-\bar{f}-\rho^{2}(a_{ij}-\bar{a}_{ij})\partial_{ij}v-\rho(b_{i}-\bar{b}_{i})\partial_{i}v-(c-\bar{c})v.
\end{align*}
By $a_{ij}(\cdot,t),b_{i}(\cdot,t),c(\cdot,t),f(\cdot,t)$ converging to 
$\bar{a}_{ij},\bar{b}_{i},\bar{c},\bar{f}$ in $L^{\infty}(\Omega)$ as $t\rightarrow\infty$, respectively, 
and $\rho Dv,\rho^{2}D^{2}v\in L^{\infty}(\Omega)$, 
we have $f_{0}(\cdot,t)$ converges to $0$ in $L^{\infty}(\Omega)$ as $t\rightarrow\infty$. 
We cannot apply Theorem \ref{4C} to \eqref{4m}-\eqref{4o} directly since it is not clear that
$$
h_{0}(x,t)=\phi_{0}(x)\exp\Big\{\int^{t}_{0}c(x,s)ds\Big\}-\int^{t}_{0}\exp\Big\{\int^{t}_{s}c(x,\tau)d\tau\Big\}f_{0}(x,s) ds.
$$
However, in the proof of Theorem \ref{4C}, the above relation is only used to obtain \eqref{4k} and \eqref{4l}. 
By Lemma \ref{4A}(i), we have $h_{0}(\cdot,t)$ converges to $0$ in $L^{\infty}(\partial\Omega)$ as $t\rightarrow\infty$. 
For some constant $T_{0}>0$, by Lemma \ref{2B}, we have
$$
|w(x,t)|\leq |u(x,t)|+|v(x)|\leq C\quad\text{in }Q_{T_{0}}.
$$
Hence, by a similar argument as in the proof of Theorem \ref{4C}, 
we conclude that $w(\cdot,t)$ converges to $0$ in $L^{\infty}(\Omega)$ as $t\rightarrow\infty$ and obtain the desired result.
\end{proof}

\subsection{$C^{2+\alpha}_{\ast}$-Convergence}\label{2+alpha convergence}
First, we construct supersolutions independent of $t$ when $t$ is large, which improves Lemma \ref{2D}.

\begin{lemma}\label{4E}
Let $\mu>0$ be a constant, 
$\Omega$ be a bounded domain in $\mathbb R^n$ with a $C^{1}$-boundary $\partial\Omega$, and $\rho$ be a $C^{1}(\bar\Omega)\cap C^{2}(\Omega)$-defining function with $\rho\nabla^2\rho\in C(\bar\Omega)$ and $\rho\nabla^2\rho=0$ on $\partial\Omega$. 
Assume that $a_{ij},b_{i},c\in C(\bar{Q})$ and $\bar{a}_{ij},\bar{b}_{i},\bar{c}\in C(\bar{\Omega})$, 
with \eqref{1g}, and $\bar{c}<0$ and $P(\mu)<0$ on $\partial\Omega$, 
and that $a_{ij}(\cdot,t),b_{i}(\cdot,t),c(\cdot,t)$ converge to 
$\bar{a}_{ij},\bar{b}_{i},\bar{c}$ in $L^{\infty}(\Omega)$ as $t\rightarrow\infty$, respectively. 
Then, there exist positive constants $r$, $K$, $T$, and $C_{0}$, depending only on $\mu$, $\Omega$, $\rho$, $a_{i j},b_{i},c$, and $\bar{a}_{i j},\bar{b}_{i},\bar{c}$,
such that, for any $x_{0}\in \partial\Omega$ and
any $(x,t)\in(\Omega\cap B_{r}(x_{0}))\times [T,\infty)$,
$$
L(|x-x_{0}|^{\mu}+K\rho^{\mu})\leq -C_{0}|x-x_{0}|^{\mu}.
$$
\end{lemma}

\begin{proof}
Without loss of generality, we assume $x_{0}=0$. A straightforward computation yields
$$
L(\rho^{\mu})=\eta \rho^{\mu}\quad\text{in } Q,
$$
where
$$
\eta=\mu(\mu-1)a_{ij}\rho_{i}\rho_{j}+\mu b_{i}\rho_{i}+c +\mu \rho a_{ij}\rho_{ij}.
$$
By $a_{ij}(\cdot,t),b_{i}(\cdot,t),c(\cdot,t)$ converging to $\bar{a}_{ij},\bar{b}_{i},\bar{c}$ 
in $L^{\infty}(\Omega)$ as $t\rightarrow\infty$, respectively, we have
$$
\eta(\cdot,t)\rightarrow \bar{\eta}\quad\text{in }L^{\infty}(\Omega)\text{ as }t\rightarrow\infty,
$$
where
$$
\bar{\eta}=\mu(\mu-1)\bar{a}_{ij}\rho_{i}\rho_{j}+\mu \bar{b}_{i}\rho_{i}+\bar{c} +\mu \rho \bar{a}_{ij}\rho_{ij}.
$$
Note that $\bar\eta=P(\mu)$ on $\partial\Omega$. By $P(\mu)<0$ on $\partial\Omega$, 
we can take positive constants $T,c_{\mu}$, and $r$ such that
\begin{equation}\label{4p}
L(\rho^{\mu})\leq -c_{\mu}\rho^{\mu}\quad\text{in }(\Omega\cap B_{r})\times [T,\infty).
\end{equation}
By $\bar{c}<0$ on $\partial\Omega$ and taking $r$ smaller and $T$ larger if necessary, there exists a positive constant $c_{0}$ such that
\begin{equation}\label{4q}
c<-c_{0}\quad\text{in }(\Omega\cap B_{r})\times [T,\infty).
\end{equation}
By \eqref{4q} and the Cauchy inequality, we have, for any $(x,t)\in(\Omega\cap B_{r})\times [T,\infty)$, 
\begin{align}\label{4r}\begin{split}
L(|x|^{\mu})&=\mu(\mu-2)|x|^{\mu-4}\rho^{2}a_{ij}x_{i}x_{j}+\mu|x|^{\mu-2}\rho^{2}a_{ii}\\
&\qquad+\mu|x|^{\mu-2}\rho b_{i}x_{i}+c|x|^{\mu}\\
&\leq C(|x|^{\mu-2}\rho^{2}+|x|^{\mu-1}\rho)-c_{0}|x|^{\mu}\\
&\leq C|x|^{\mu-2}\rho^{2}-\frac{c_{0}}{2}|x|^{\mu}\\
&=|x|^{\mu}\Big(C\frac{\rho^{2}}{|x|^{2}}-\frac{c_{0}}{2}\Big).
\end{split}\end{align}
For some constant $K\geq 0$ to be determined, set
$$
w(x)=|x|^{\mu}+K\rho^{\mu}.
$$
By \eqref{4p} and \eqref{4r}, we get
$$
Lw\leq|x|^{\mu}\Big(C\frac{\rho^{2}}{|x|^{2}}-\frac{c_{0}}{2}-c_{\mu}K\frac{\rho^{\mu}}{|x|^{\mu}}\Big)
\quad\text{in }(\Omega\cap B_{r})\times [T,\infty).
$$
We take $\delta=\sqrt{c_{0}/(4C)}$. For $\rho\leq \delta |x|$, we have
$$
Lw\leq -\frac{c_{0}}{4}|x|^{\mu}.
$$
Note that we always have $\rho\leq C_1d\leq C_1|x|$. Next, for $\rho\geq \delta |x|$, we have
$$
Lw\leq |x|^{\mu}\left(CC^2_1-\frac{c_{0}}{2}-c_{\mu}K\delta^{\mu}\right)\leq-\frac{c_{0}}{4}|x|^{\mu},
$$
by taking $K$ sufficiently large. We obtain the desired result by taking $C_{0}=c_{0}/4$.
\end{proof}

Next, by Lemma \ref{4E}, we can improve Lemma \ref{2F} when $t$ is large. 

\begin{lemma}\label{4F}
Let $\Omega$ be a bounded domain in $\mathbb R^n$ with a $C^{1}$-boundary $\partial\Omega$
and $\rho$ be a $C^{1}(\bar\Omega)\cap C^{2}(\Omega)$-defining function with $\rho\nabla^2\rho\in C(\bar\Omega)$ and $\rho\nabla^2\rho=0$ on $\partial\Omega$. 
For some constant $\alpha\in(0,1)$, assume that $a_{ij},b_{i},c,f\in C^{\alpha}(\bar{Q})$ 
and $\bar{a}_{ij},\bar{b}_{i},\bar{c}\in C^{\alpha}(\bar{\Omega})$, 
with \eqref{1g}, \eqref{1o}, and $P(\alpha)<0$ on $\partial\Omega$, 
and that $a_{ij}(\cdot,t),b_i(\cdot,t)$ converge to $\bar{a}_{ij},\bar{b}_{i}$ in $L^{\infty}(\Omega)$ as $t\rightarrow\infty$, respectively, 
and $f(\cdot,t),c(\cdot,t)$ converge to $0,\bar{c}$ in $C^{\alpha}(\bar{\Omega})$ as $t\rightarrow\infty$, respectively. 
Suppose, for some $\phi\in C^{\alpha}(\bar{\Omega})$, the initial-boundary value problem \eqref{1i}-\eqref{1k} 
admits a solution $u\in C(\bar{Q})\cap C^{2}(Q)$. 
Then, for any $\varepsilon>0$, there exists a positive constant $T$, 
such that for any $(x,t)\in Q\backslash Q_{T}$ and $(x_{0},t)\in S\backslash S_{T}$,
\begin{equation}\label{4s}
|u(x,t)-h(x_{0},t)|\leq \varepsilon |x-x_{0}|^{\alpha}.
\end{equation}
\end{lemma}

\begin{proof}
By Lemma \ref{4E}, we can take positive constants $r,K,T,$ and $C_{0}$ such that, 
for any $x_0\in\partial\Omega$ and any $(x,t)\in(\Omega\cap B_{r}(x_{0}))\times [T,\infty)$,
\begin{equation}\label{4t}
L(|x-x_{0}|^{\alpha}+K\rho^{\alpha})\leq -C_{0}|x-x_{0}|^{\alpha}.
\end{equation}
Fix an $\varepsilon>0$. By $f(\cdot,t),c(\cdot,t)$ converging to $0,\bar{c}$ 
in $C^{\alpha}(\bar{\Omega})$ as $t\rightarrow\infty$, respectively, Lemma \ref{4A}(ii), and Theorem \ref{4C}, we have, for any $t\geq T$,
\begin{equation}\label{4u}
\|f(\cdot,t)\|_{C^{\alpha}(\bar{\Omega})}<\varepsilon,\quad \|c(\cdot,t)-\bar{c}(\cdot)\|_{C^{\alpha}(\bar{\Omega})}
<\varepsilon,\quad\|h(\cdot,t)\|_{C^{\alpha}(\partial\Omega)}<\varepsilon,
\end{equation}
and
\begin{equation}\label{4v}
\|u(\cdot,t)\|_{L^{\infty}(\Omega)}<\varepsilon,
\end{equation}
by taking $T$ larger if necessary.

Fix a point $x_{0}\in \partial\Omega$. By \eqref{1l}, we get
\begin{equation}\label{4w}
L(u(x,t)-h(x_{0},t))=g,
\end{equation}
where
\begin{align*}
g(x,t)&=f(x,t)-c(x,t)h(x_{0},t)+\partial_{t}h(x_{0},t)\\
&=f(x,t)-f(x_{0},t)-(c(x,t)-c(x_{0},t))h(x_{0},t).
\end{align*}
Then, by \eqref{4u}, we have
$$
|g(x,t)|\leq C_{1}\varepsilon|x-x_{0}|^{\alpha}\quad\text{for any }(x,t)\in Q\backslash Q_{T}.
$$
Hence, we obtain, for any $(x,t)\in Q\backslash Q_{T}$,
$$
\pm L(u(x,t)-h(x_{0},t))\geq -C_{1}\varepsilon|x-x_{0}|^{\alpha}.
$$
Next, by Lemma \ref{2F}, \eqref{4u}, and \eqref{4v}, we get, for any $x\in\Omega$,
$$
\pm (u(x,T)-h(x_{0},T))\leq C_{T}|x-x_{0}|^{\alpha},
$$
for any $(x,t)\in S\backslash S_{T}$,
$$
\pm (u(x,t)-h(x_{0},t))=\pm (h(x,t)-h(x_{0},t))\leq \varepsilon|x-x_{0}|^{\alpha},
$$
and, for any $(x,t)\in Q\backslash Q_{T}$,
$$
\pm (u(x,t)-h(x_{0},t))\leq 2|u|_{L^{\infty}(Q\backslash Q_{T})}\leq 2\varepsilon.
$$
For some positive constants $A,B,$ and $H$ to be determined, we set
$$
\psi(x,t)=A\varepsilon w(x)+Bw(x)e^{-H(t-T)},
$$
where
$$
w(x)=|x-x_{0}|^{\alpha}+K \rho^{\alpha}.
$$
By \eqref{4t} and $\rho(x)\leq c_1 d(x)\leq c_1|x-x_{0}|$ for some positive constant $c_1$, 
we have, for any $(x,t)\in (\Omega\cap B_{r}(x_{0}))\times [T,\infty),$
$$
\begin{aligned}
L\psi&\leq -C_{0}A\varepsilon|x-x_{0}|^{\alpha}-C_{0}B|x-x_{0}|^{\alpha}e^{-H(t-T)}+BHwe^{-H(t-T)}\\
&\leq -C_{0}A\varepsilon|x-x_{0}|^{\alpha}-[C_{0}-H(Kc^\alpha_1+1)]B|x-x_{0}|^{\alpha}e^{-H(t-T)}.
\end{aligned}
$$
We take $H=C_{0}/(Kc^\alpha_1+1)$. Then,
$$
L\psi\leq -C_{0}A\varepsilon|x-x_{0}|^{\alpha}\quad\text{in } (\Omega\cap B_{r}(x_{0}))\times [T,\infty).
$$
Next, we have, for any $(x,t)\in S\backslash S_{T}$,
$$
\psi(x,t)\geq A\varepsilon |x-x_{0}|^{\alpha},
$$
for any $(x,t)\in \bar{Q}\cap (\partial B_{r}(x_{0})\times [T,\infty))$,
$$
\psi(x,t)\geq A\varepsilon r^{\alpha},
$$
and for any $x\in B_{r}(x_{0})\cap\Omega$,
$$
\psi(x,T)\geq B|x-x_{0}|^{\alpha}.
$$
We now take $A=\max\{C_{1}/C_{0},1,2r^{-\alpha}\}$ and $B=C_{T}$. Then, 
$$
\begin{aligned}
\pm L(u(x,t)-h(x_{0},t))&\geq L\psi\quad\text{in }(\Omega\cap B_{r}(x_{0}))\times (T,\infty)\\
\pm (u(x,t)-h(x_{0},t))&\leq \psi\quad\text{on }\partial_{p}\{(\Omega\cap B_{r}(x_{0}))\times (T,\infty)\}.
\end{aligned}
$$
The maximum principle implies
$$
\pm (u(x,t)-h(x_{0},t))\leq \psi\quad\text{in }(\Omega\cap B_{r}(x_{0}))\times [T,\infty).
$$
Hence,
we have, for any $(x,t)\in (\Omega\cap B_{r}(x_{0}))\times [T,\infty)$,
\begin{equation}\label{4x}
\begin{aligned}
|u(x,t)-h(x_{0},t)|&\leq (A\varepsilon +Be^{-H(t-T)})(Kc^\alpha_1+1)|x-x_{0}|^{\alpha}\\
&\leq2A(Kc^\alpha_1+1)\varepsilon |x-x_{0}|^{\alpha},
\end{aligned}
\end{equation}
if
$$
t\geq T-\frac{1}{H}\log\frac{A\varepsilon}{B}.
$$
By \eqref{4v}, we obtain, for any $(x,t)\in (\Omega\backslash B_{r}(x_{0}))\times [T,\infty)$.
\begin{equation}\label{4y}
|u(x,t)-h(x_{0},t)|\leq 2\varepsilon\leq \frac{2\varepsilon}{r^{\alpha}}|x-x_{0}|^{\alpha}.
\end{equation}
Note that $A,K,c_1$, and $r$ are independent of $\varepsilon$. 
By combining \eqref{4x} and \eqref{4y}, and renaming $T$ and $\varepsilon$, we conclude the desired result.
\end{proof}

Lemma \ref{4F} and its local version Lemma \ref{4J} 
will play an important role in proving the $C^{k,2+\alpha}_*$-convergence of solutions.

Now, we prove the $C^{2+\alpha}_{\ast}$-decay for solutions of the initial-boundary value problem.

\begin{thm}\label{4G}
Let $\Omega$ be a bounded domain in $\mathbb R^n$ with a $C^{1}$-boundary $\partial\Omega$ 
and $\rho$ be a $C^{1}(\bar\Omega)\cap C^{2}(\Omega)$-defining function 
with $\rho\nabla^2\rho\in C(\bar\Omega)$ and $\rho\nabla^2\rho=0$ on $\partial\Omega$. 
For some constant $\alpha\in(0,1)$, assume that $a_{ij},b_{i},c,f\in C^{\alpha}(\bar{Q})$ 
and $\bar{a}_{ij},\bar{b}_{i},\bar{c}\in C^{\alpha}(\bar{\Omega})$, 
with \eqref{1g}, \eqref{1o}, and $P(\alpha)<0$ on $\partial\Omega$, 
and that $a_{ij}(\cdot,t),b_i(\cdot,t),c(\cdot,t),f(\cdot,t)$ converge to $\bar{a}_{ij},\bar{b}_{i},\bar{c},0$ 
in $C^{\alpha}(\bar{\Omega})$ as $t\rightarrow\infty$, respectively. 
Suppose, for some $\phi\in C^{2+\alpha}(\bar{\Omega})$, the initial-boundary value problem \eqref{1i}-\eqref{1k} 
admits a solution $u\in C(\bar{Q})\cap C^{2}(Q)$. Then, $u$ $C^{2+\alpha}_{\ast}$-converges to $0$ in $\bar{Q}$ at infinity.
\end{thm}

\begin{proof}
The interior Schauder theory for uniformly parabolic equations implies $u\in C^{2,\alpha}(Q_{T})$ for any $T>0$. 
Now we derive a weighted $C^{2,\alpha}$-estimate of $u$.
Fix an $\varepsilon>0$. For any $T>0$ and $0<a<b$, set 
$$Q^{a,b}=\Omega\times (a,b] \quad\text{and}\quad S^{a,b}=\partial\Omega\times[a,b].$$ 
We claim, there exists a constant $T_{0}>0$ such that, for any $T>T_{0}$, 
and any  $(x,t)\in Q^{T,T+1}$ and $(x_{0},t)\in S^{T,T+1}$, with $d(x)=|x-x_{0}|$,
\begin{equation}\label{4z}
|u(x,t)-h(x_{0},t)|+|\rho(x)D_{x}u(x,t)|+|\rho^{2}(x)D^{2}_{x}u(x,t)|\leq C\varepsilon d^{\alpha}(x),
\end{equation}
and
\begin{align}\label{4aa}\begin{split}
[u]_{C^{\alpha}(B_{d(x)/4}(x)\times[T,T+1])}+[\rho D_{x}u]_{C^{\alpha}(B_{d(x)/4}(x)\times[T,T+1])} &\\
+[\rho^{2}D^{2}_{x}u]_{C^{\alpha}(B_{d(x)/4}(x)\times[T,T+1])}
&\leq C\varepsilon,
\end{split}\end{align}
where $C$ is a positive constant independent of $\varepsilon$.
By Lemma \ref{2C}, we get
\begin{align*}
&\|u\|_{C^{\alpha}(\bar{Q}^{T,T+1})}+\|\rho D_{x}u\|_{C^{\alpha}(\bar{Q}^{T,T+1})}+\|\rho^{2}D^{2}_{x}u\|_{C^{\alpha}(\bar{Q}^{T,T+1})}\\
&\qquad\leq C\big\{\|h\|_{C^{\alpha}(S^{T,T+1})}+\varepsilon\big\}.
\end{align*}
By Lemma \ref{4A}(ii)-(iii) and taking $T_{0}$ larger if necessary, we have, for any $T>T_{0}$,
$$
\|f(\cdot,T)\|_{C^{\alpha}(\bar{\Omega})}< \varepsilon,\quad \|h\|_{C^{\alpha}(S^{T,T+1})}<\varepsilon.
$$
Hence,
\begin{equation}\label{4aa'}
\|u\|_{C^{\alpha}(\bar{Q}^{T,T+1})}+\|\rho D_{x}u\|_{C^{\alpha}(\bar{Q}^{T,T+1})}
+\|\rho^{2}D^{2}_{x}u\|_{C^{\alpha}(\bar{Q}^{T,T+1})}\leq C\varepsilon,
\end{equation}
and, by \eqref{1i}, we obtain
$$
\|\partial_{t}u(\cdot,T)\|_{C^{\alpha}(\bar{\Omega})}\leq C\varepsilon.
$$
We hence conclude the desired result.

We now prove \eqref{4z} and \eqref{4aa}. 
We take arbitrary $T>1$, $x\in \Omega $, and $x_{0}\in \partial\Omega$ with $d(x)=|x-x_{0}|$. 
Set $r=d(x)/2$. 
We consider
$$
\Tilde{u}(y,t)=u(x+ry,t)-h(x_{0},t)\quad\text{in }Q'_{T},
$$
where $Q'_{T}=B_{1}(0)\times (T-1,T+1]$. A straightforward computation yields 
\begin{equation}\label{4ad}
\Tilde{a}_{ij}(y,t)\partial_{ij}\Tilde{u}+\Tilde{b}_{i}(y,t)\partial_{i}\Tilde{u}+\Tilde{c}(y,t)\Tilde{u}-\partial_{t}\Tilde{u}=\Tilde{g}(y,t)\quad\text{in }Q'_{T},
\end{equation}
where
\begin{align*}
\Tilde{a}_{ij}(y,t)&=\frac{\rho^{2}(x+ry)}{r^{2}}a_{ij}(x+ry,t),\\
\Tilde{b}_{i}(y,t)&=\frac{\rho(x+ry)}{r}b_{i}(x+ry,t),\\
\Tilde{c}(y,t)&=c(x+ry,t),
\end{align*}
and
\begin{align*}
\Tilde{g}(y,t)&=f(x+ry,t)-c(x+ry,t)h(x_{0},t)+\partial_{t}h(x_{0},t)\\
&=f(x+ry,t)-f(x_{0},t)-(c(x+ry,t)-c(x_{0},t))h(x_{0},t).
\end{align*}
Note that
$$
\|a_{ij}\|_{C^{\alpha}_{\ast}(\bar{Q}^{T-1,T+1})}+\|b_{i}\|_{C^{\alpha}_{\ast}(\bar{Q}^{T-1,T+1})}+\|c\|_{C^{\alpha}_{\ast}(\bar{Q}^{T-1,T+1})}\leq C,
$$
where $C$ is a positive constant independent of $T$. Hence, it is easy to verify that
$$
\|\Tilde{a}_{ij}\|_{C^{\alpha}_{\ast}(\bar{Q}'_{T})}+\|\Tilde{b}_{i}\|_{C^{\alpha}_{\ast}(\bar{Q}'_{T})}+\|\Tilde{c}\|_{C^{\alpha}_{\ast}(\bar{Q}'_{T})}\leq C,
$$
where $C$ is a positive constant independent of $x$ and $T$. Moreover, for any $\xi\in\mathbb{R}^{n}$ and any $(y,t)\in Q'_{T}$,
$$
\Tilde{a}_{ij}(y,t)\xi_{i}\xi_{j}\geq \lambda \frac{\rho^{2}(x+ry)}{r^{2}}|\xi|^{2}\geq \lambda c_1 |\xi|^{2}.
$$
Set $Q''_{T}=B_{1/2}(0)\times(T,T+1]$. By applying Theorem \ref{5A} to \eqref{4ad} in $Q'_{T}$, we obtain
\begin{equation}\label{4ae}
\begin{aligned}
\|\Tilde{u}\|_{C^{2,\alpha}_{\ast}(\bar{Q}''_{T})}\leq C\big\{\|\Tilde{u}\|_{L^{\infty}(Q'_{T})}+\|\Tilde{g}\|_{C^{\alpha}_{\ast}(\bar{Q}'_{T})}\big\}.
\end{aligned}
\end{equation}
Set $M'_{T}=B_{r}(x)\times (T-1,T+1]$ and $M''_{T}=B_{r/2}(x)\times (T,T+1]$.
By Lemma \ref{4F}, there exists a constant $T_{1}>1$ such that, for any $z\in B_{r}(x)$ and any $t>T_1$, 
\begin{equation}\label{decay-C2+alpha}
|u(z,t)-h(x_{0},t)|\leq  \varepsilon|z-x_{0}|^{\alpha}\leq C\varepsilon r^{\alpha}.
\end{equation}
Here, $T_1$ plays the role of $T$ in Lemma \ref{4F}. 
Hence, we obtain, for any $T>T_1+1$,
\begin{equation}\label{4ag}
\|\Tilde{u}\|_{L^{\infty}(Q'_{T})}=\|u-h(x_{0},\cdot)\|_{L^{\infty}(M'_{T})}\leq C\varepsilon r^{\alpha}.
\end{equation}
By the assumption and Lemma \ref{4A}(ii)-(iii), there exists a positive constant $T_{2}>1$ such that, for any $T>T_{2}$,
\begin{equation}\label{4af}
\|f(\cdot,T)\|_{C^{\alpha}(\bar{\Omega})}<\varepsilon,\quad \|h\|_{C^{\alpha}(S^{T-1,T+1})}< \varepsilon.
\end{equation}
Thus, 
for any $T>T_{2}+1$,
\begin{equation}\label{4ah}
\|\Tilde{g}\|_{C^{\alpha}_{\ast}(\bar{Q}'_{T})}\leq C\varepsilon r^{\alpha}.
\end{equation}
We take $T_{0}>\max\{T_{1},T_{2}\}+1$. By combining \eqref{4ae}, \eqref{4ag}, and \eqref{4ah}, we obtain, for any $T>T_{0}$,
\begin{align*}
&r\|D_{x}u\|_{L^{\infty}(M''_{T})}+r^{2}\|D^{2}_{x}u\|_{L^{\infty}(M''_{T})}\\
&\qquad\qquad+r^{\alpha}[u]_{C^{\alpha}_{x}(\bar{M}''_{T})}
+r^{1+\alpha}[D_{x}u]_{C^{\alpha}_{x}(\bar{M}''_{T})}+r^{2+\alpha}[D^{2}_{x}u]_{C^{\alpha}_{x}(\bar{M}''_{T})}\\
&\qquad\qquad+[u-h(x_{0},t)]_{C^{\alpha/2}_{t}(\bar{M}''_{T})}+r[D_{x}u]_{C^{\alpha/2}_{t}(\bar{M}''_{T})}
+r^{2}[D^{2}_{x}u]_{C^{\alpha/2}_{t}(\bar{M}''_{T})}\\
&\qquad\leq C\varepsilon r^{\alpha}.
\end{align*}
Recall \eqref{decay-C2+alpha} and \eqref{4af}. We finish the proof of \eqref{4z} and \eqref{4aa}.
\end{proof}

\begin{remark}\label{4G'}
After proving \eqref{4aa'}, if we further assume that $a_{ij},b_{i},c,f$ $C^{\alpha}$-converge 
to $\bar{a}_{ij},\bar{b}_{i},\bar{c},0$ in $\bar{Q}$ at infinity, respectively,  we can take $T_{0}$ larger such that
$$
\|f\|_{C^{\alpha}(\bar{Q}^{T,T+1})}<\varepsilon\quad\text{for }T>T_{0}.
$$
By \eqref{1i}, we have
$$
\|\partial_{t}u\|_{C^{\alpha}(\bar{Q}^{T,T+1})}\leq C\varepsilon\quad\text{for }T>T_{0}.
$$
Therefore, $u$ $C^{2+\alpha}$-converges to $0$ in $\bar{Q}$ at infinity.
\end{remark}

As a direct consequence, we obtain a $C^{2+\alpha}_{\ast}$-convergence result.

\begin{corollary}\label{4H}
Let $\Omega$ be a bounded domain in $\mathbb R^n$ with a $C^{1}$-boundary $\partial\Omega$
and $\rho$ be a $C^{1}(\bar\Omega)\cap C^{2}(\Omega)$-defining function with $\rho\nabla^2\rho\in C(\bar\Omega)$ and $\rho\nabla^2\rho=0$ on $\partial\Omega$. 
For some constant $\alpha\in(0,1)$, assume that $a_{ij},b_{i},c,f\in C^{\alpha}(\bar{Q})$ 
and $\bar{a}_{ij},\bar{b}_{i},\bar{c},\bar{f}\in C^{\alpha}(\bar{\Omega})$, 
with \eqref{1g}, \eqref{1o}, and $P(\alpha)<0$ on $\partial\Omega$, 
and that $a_{ij}(\cdot,t),b_i(\cdot,t),c(\cdot,t),f(\cdot,t)$ converge to $\bar{a}_{ij},\bar{b}_{i},\bar{c},\bar{f}$ 
in $C^{\alpha}(\bar{\Omega})$ as $t\rightarrow\infty$, respectively.  
Suppose that, for some $\phi\in C^{2+\alpha}(\bar{\Omega})$, the initial-boundary value problem \eqref{1i}-\eqref{1k} 
admits a solution $u\in C^{2}(Q)\cap C(\bar{Q})$. Then, $u$ $C^{2+\alpha}_{\ast}$-converges 
to the unique solution $v\in C^{2}(\Omega)\cap C(\bar{\Omega})$ of the Dirichlet problem \eqref{1m}-\eqref{1n} in $\bar{Q}$ at infinity.
\end{corollary}

\begin{proof}
By Theorem \ref{1C'}, the Dirichlet problem \eqref{1m}-\eqref{1n} 
admits a unique solution $v\in C^{2+\alpha}(\bar\Omega)$. Let $w=u-v$ in $Q$. Then, $w$ satisfies 
\begin{align}
Lw&=f_{0}\quad\text{in }Q,\label{4ai}\\
w(\cdot,0)&=\phi_{0}=\phi-v\quad\text{on }\Omega,\label{4aj}\\
w&=h_{0}=h-\bar{f}/\bar{c}\quad\text{on }S,\label{4ak}
\end{align}
where
\begin{align*}
f_{0}&=Lu-L_{0}v+(L_{0}-L)v\\
&=f-\bar{f}-\rho^{2}(a_{ij}-\bar{a}_{ij})v_{ij}-\rho(b_{i}-\bar{b}_{i})v_{i}-(c-\bar{c})v.
\end{align*}
By $v\in C^{2+\alpha}(\bar\Omega)$, we get
$$
f_{0}= f-\bar{f}-(c-\bar{c})\frac{\bar{f}}{\bar{c}}=f-\frac{c\bar{f}}{\bar{c}}\quad\text{on }S,
$$
and, by \eqref{1l},
$$
\partial_{t}h_{0}-ch_{0}+f_{0}=\partial_{t}h-c\left(h-\frac{\bar{f}}{\bar{c}}\right)+f_{0}=-f+\frac{c\bar{f}}{\bar{c}}+f_{0}=0\quad\text{on }S.
$$
Hence, 
$$
h_{0}(x,t)=\phi_{0}(x)\exp\Big\{\int^{t}_{0}c(x,s)ds\Big\}-\int^{t}_{0}\exp\Big\{\int^{t}_{s}c(x,\tau)d\tau\Big\}f_{0}(x,s) d s.
$$
By $v\in C^{2+\alpha}(\bar\Omega)$ and $a_{ij}(\cdot,t),b_i(\cdot,t),c(\cdot,t),f(\cdot,t)$ 
converging to $\bar{a}_{ij},\bar{b}_{i},\bar{c},\bar{f}$ in $C^{\alpha}(\bar{\Omega})$ as $t\rightarrow\infty$, respectively, 
we have $f_{0}(\cdot,t)$ converges to $0$ in $C^{\alpha}(\bar{\Omega})$ as $t\rightarrow\infty$. 
By applying Theorem \ref{4G} to \eqref{4ai}-\eqref{4ak}, we conclude the desired result.
\end{proof}

\subsection{$C^{k,2+\alpha}_{\ast}$-Convergence}\label{k-2+alpla convergence}
In this subsection, we study the higher order convergence of solutions of the initial-boundary problem \eqref{1i}-\eqref{1k}.  
We do this in a special setting and consider domains with a piece of flat boundary.

Denote by $x=(x',x_{n})$ points in $\mathbb{R}^{n}$ and $X=(x,t)=(x',x_{n},t)$ points in $\mathbb{R}^{n}\times[0,\infty)$. 
In addition to continuing to use the notations from Section \ref{sec-Higher-regularity}, we also set, for any $r>0$ and $0<a<b\leq \infty$,
\begin{align*}
Q_{r,\infty}&=\{(x,t)\in\mathbb{R}^{n}\times (0,\infty):x\in G_{r}\},\\
S_{r,\infty}&=\{(x',0,t)\in\mathbb{R}^{n}\times[0,\infty):x'\in B'_{r}\},
\end{align*}
and 
\begin{align*}
Q^{a,b}_{r}&=\{(x,t)\in\mathbb{R}^{n}\times(0,\infty):x\in G_{r},a< t\leq b\},\\
S^{a,b}_{r}&=\{(x',0,t)\in\mathbb{R}^{n}\times[0,\infty):x'\in B'_{r},a\leq t\leq b\}.
\end{align*}

Let $a_{i j}$, $b_{i}$, and $c$ be bounded and continuous functions in $\bar{Q}_{1,\infty}$, with $a_{i j}=a_{j i}$ and, 
for any $(x,t)\in\bar{Q}_{1,\infty}$ and any $\xi\in\mathbb{R}^{n}$,
\begin{equation}\label{4am}
\lambda|\xi|^{2}\leq a_{ij}(x,t)\xi_{i}\xi_{j}\leq \Lambda|\xi|^{2},
\end{equation}
for some positive constants $\lambda$ and $\Lambda$. We consider the operator
\begin{equation}\label{4an}
L=x^{2}_{n} a_{ij}\partial_{ij}+x_{n} b_{i}\partial_{i}+c-\partial_{t}\quad\text{in }Q_{1,\infty},
\end{equation}
and its limit elliptic operator $L_{0}$ given by
$$
L_{0}=x^{2}_{n}\bar{a}_{ij}\partial_{ij}+x_{n} \bar{b}_{i}\partial_{i}+\bar{c}\quad\text{in }G_{1},
$$
where $\bar{a}_{i j},\bar{b}_{i},\bar{c}\in C(\bar{G}_{1})$ are the limits of $a_{i j}(\cdot,t),b_{i}(\cdot,t),c(\cdot,t)$ 
in certain sense as $t\rightarrow\infty$, respectively, with $\bar{c}(\cdot,0)<0$ on $B'_1$ and, 
for any $x\in\bar{G}_{1}$ and any $\xi\in\mathbb{R}^{n}$,
$$
\lambda|\xi|^{2}\leq \bar{a}_{ij}(x)\xi_{i}\xi_{j}\leq \Lambda|\xi|^{2}.
$$
We now define
$$
Q(\mu)=\mu(\mu-1)\bar{a}_{nn}+\mu \bar{b}_{n}+\bar{c}\quad\text{in }G_{1}.
$$
We point out that $Q(\mu)$ here is defined in the entire domain $G_{1}$, 
instead of $P(\mu)$ only on the boundary in Subsection \ref{Main Results}. Note that $Q(0)=\bar{c}$.

Let $f$ be a bounded continuous function in $\bar{Q}_{1,\infty}$ and $\bar{f}\in C(\bar{G}_{1})$ be the limit of $f(\cdot,t)$ 
in certain sense as $t\rightarrow\infty$. Let $u$ be a solution of the problem
\begin{align}
Lu&=f\quad\text{in }Q_{1,\infty},\label{4ao}\\
u(\cdot,0)&=\phi\quad\text{on }G_{1},\label{4ap}\\
u&=h\quad \text{on }S_{1,\infty},\label{4aq}
\end{align}
where $\phi\in C(\bar{G}_{1})$ and
$$
h(x,t)=\phi(x)\exp\Big\{\int^{t}_{0}c(x,s)ds\Big\}-\int^{t}_{0}\exp\Big\{\int^{t}_{s}c(x,\tau)d\tau\Big\}f(x,s) d s.
$$
Note that $h$ and $\phi$ satisfy
\begin{align}\label{4ar}\begin{split}
\partial_{t}h-ch+f&=0\quad\text{on } S_{1,\infty},\\
h(\cdot,0,0)&=\phi(\cdot,0)\quad\text{on }B'_{1}.
\end{split}\end{align}
Let $v$ be a solution of the problem
\begin{align}
L_{0}v&=\bar{f}\quad\text{in }G_{1},\label{4as}\\
v(\cdot,0)&=\frac{\bar{f}}{\bar{c}}(\cdot,0)\quad\text{on }B'_{1}.\label{4at}
\end{align}
In this local setting, we may assume that
\begin{equation}\label{4au}
\bar{c}< 0\quad\text{in } \bar{G}_{1},
\end{equation}
and
\begin{equation}\label{4av}
u(\cdot,t)\rightarrow 0\quad\text{in }L^{\infty}(G_{1})\quad\text{as }t\rightarrow\infty.
\end{equation}

By proceeding similarly as in Subsection \ref{2+alpha convergence}, 
we can prove the local versions of Lemma \ref{4E}, Lemma \ref{4F}, and Theorem \ref{4G}. 

\begin{lemma}\label{4I}
Let $\mu>0$ be a constant. Assume that $a_{ij},b_{i},c\in C(\bar{Q}_{1,\infty})$ 
and $\bar{a}_{ij},\bar{b}_{i},\bar{c}\in C(\bar{G}_{1})$, with \eqref{4am}, \eqref{4au}, and $Q(\mu)<0$ in $\bar{G}_{1}$, 
and that $a_{ij}(\cdot,t),b_{i}(\cdot,t)$, $c(\cdot,t)$ converge to $\bar{a}_{ij},\bar{b}_{i},\bar{c}$ 
in $L^{\infty}(G_{1})$ as $t\rightarrow\infty$, respectively. 
Then, there exist positive constants $K$, $T$, and $C_{0}$, depending only on $\mu$, $a_{i j},b_{i},c$, 
and $\bar{a}_{i j},\bar{b}_{i},\bar{c}$, such that, for any $x'_{0}\in B'_{1}$ and any $(x,t)\in Q^{T,\infty}_{1}$,
$$
L(|x-(x'_{0},0)|^{\mu}+Kx^{\mu}_{n})\leq -C_{0}|x-(x'_{0},0)|^{\mu}.
$$
\end{lemma}

\begin{lemma}\label{4J}
For some constant $\alpha\in(0,1)$, assume that $a_{ij},b_{i},c,f\in C^{\alpha}(\bar{Q}_{1,\infty})$ 
and $\bar{a}_{ij},\bar{b}_{i},\bar{c}\in C^{\alpha}(\bar{G}_{1})$, with \eqref{4am}, \eqref{4au}, and $Q(\alpha)<0$ 
in $\bar{G}_{1}$, and that $a_{ij}(\cdot,t),b_i(\cdot,t)$ converge to $\bar{a}_{ij},\bar{b}_{i}$ in $L^{\infty}(G_{1})$ 
as $t\rightarrow\infty$, respectively, and $f(\cdot,t)$, $c(\cdot,t)$ converge to $0,\bar{c}$ in $C^{\alpha}(\bar{G}_{1})$ 
as $t\rightarrow\infty$, respectively. 
Suppose, for some $\phi\in C^{\alpha}(\bar{G}_{1})$, the problem \eqref{4ao}-\eqref{4aq} 
admits a solution $u\in C(\bar{Q}_{1,\infty})\cap C^{2}(Q_{1,\infty})$ with \eqref{4av}. 
Then, for any $\varepsilon>0$, there exists a positive constant $T$, such that for any $(x,t)\in Q^{T,\infty}_{1/2}$ and $x'_{0}\in B'_{1/2}$,
\begin{equation}\label{4aw}
|u(x,t)-h(x'_{0},0,t)|\leq \varepsilon |x-(x'_{0},0)|^{\alpha}.
\end{equation}
\end{lemma}

\begin{thm}\label{4K}
For some constant $\alpha\in(0,1)$, assume that $a_{ij}$, $b_{i}$, $c$, $f\in C^{\alpha}(\bar{Q}_{1,\infty})$ 
and $\bar{a}_{ij},\bar{b}_{i},\bar{c}\in C^{\alpha}(\bar{G}_{1})$, with \eqref{4am}, \eqref{4au}, 
and $Q(\alpha)<0$ in $\bar{G}_{1}$, and that $a_{ij}(\cdot,t),b_i(\cdot,t),c(\cdot,t),f(\cdot,t)$ converge to 
$\bar{a}_{ij},\bar{b}_{i},\bar{c},0$ in $C^{\alpha}(\bar{G}_{1})$ as $t\rightarrow\infty$, respectively. 
Suppose, for some $\phi\in C^{2+\alpha}(\bar{G}_{1})$, the problem \eqref{4ao}-\eqref{4aq} admits 
a solution $u\in C(\bar{Q}_{1,\infty})\cap C^{2}(Q_{1,\infty})$ with \eqref{4av}. 
Then, for any $r\in(0,1)$, $u$ $C^{2+\alpha}_{\ast}$-converges to $0$ in $\bar{Q}_{r,\infty}$ at infinity.
\end{thm}

In Lemma \ref{4J} and Theorem \ref{4K},  we assume \eqref{4av}. 
In applications, \eqref{4av} is a consequence of Theorem \ref{4C}.

Now, we study the $C^{2+\alpha}_{\ast}$-decay of derivatives of solutions of the problem \eqref{4ao}-\eqref{4aq}.
First, we consider the $C^{2+\alpha}_{\ast}$-decay of the $x'$-derivative.

\begin{thm}\label{4L}
For some constant $\alpha\in(0,1)$, assume that $D_{x'}^{\sigma}a_{ij}$, $D_{x'}^{\sigma}b_{i}$, 
$D_{x'}^{\sigma}c$, $D_{x'}^{\sigma}f\in C^{\alpha}(\bar{Q}_{1,\infty})$ 
and $D_{x'}^{\sigma}\bar{a}_{ij},D_{x'}^{\sigma}\bar{b}_{i},D_{x'}^{\sigma}\bar{c}\in C^{\alpha}(\bar{G}_{1})$, 
for any $\sigma\leq 1$, with \eqref{4am}, \eqref{4au}, and $Q(\alpha)<0$ in $\bar{G}_{1}$, 
and that $D_{x'}^{\sigma}a_{ij}(\cdot,t)$, $D_{x'}^{\sigma}b_{i}(\cdot,t)$, $D_{x'}^{\sigma}c(\cdot,t)$, 
$D_{x'}^{\sigma}f(\cdot,t)$ converge to $D_{x'}^{\sigma}\bar{a}_{ij},D_{x'}^{\sigma}\bar{b}_{i},D_{x'}^{\sigma}\bar{c},0$ 
in $C^{\alpha}(\bar{G}_{1})$ as $t\rightarrow\infty$, for any $\sigma\leq 1$, respectively. 
Suppose that, for some $\phi\in C^{2+\alpha}(\bar{G}_{1})$ with $D_{x'}\phi\in C^{2+\alpha}(\bar{G}_{1})$, the problem \eqref{4ao}-\eqref{4aq} admits a solution $u\in C^{2}(Q_{1,\infty})\cap C(\bar{Q}_{1,\infty})$ with \eqref{4av}. 
Then, for any $\sigma\leq 1$ and any $r\in(0,1)$, $D^{\sigma}_{x'}u$ $C^{2+\alpha}_{\ast}$-converges 
to $0$ in $\bar{Q}_{r,\infty}$ at infinity.
\end{thm}

\begin{proof}
Fix $k=1,\cdots,n-1$. As in the proof of Theorem \ref{3F}, we have, 
for any $r\in(0,1)$, $\partial_{k}u\in C^{2}(Q_{r,\infty})\cap C(\bar{Q}_{r,\infty})$ satisfies
\begin{align}
L(\partial_{k}u)&=f_{k}\quad\text{in }Q_{r,\infty},\label{4ax}\\
\partial_{k}u(\cdot,0)&=\partial_{k}\phi\quad\text{on } G_{r},\label{4ay}\\
\partial_{k}u&=\partial_{k}h\quad\text{on }  S_{r,\infty},\label{4az}
\end{align}
where
$$
f_{k}=\partial_{k}f-x^{2}_{n}\partial_{k}a_{ij}\partial_{ij}u-x_{n}\partial_{k}b_{i}\partial_{i}u-\partial_{k}cu,
$$
with
$$
f_{k}=\partial_{k}f-\partial_{k}c h\quad\text{on } S_{r,\infty},
$$
and
\begin{align}\label{4ba}\begin{split}
\partial_{t}(\partial_{k}h)-c\partial_{k}h+f_{k}&=0\quad\text{on } S_{r,\infty},\\
\partial_{k}h(\cdot,0,0)&=\partial_{k}\phi(\cdot,0)\quad\text{on }  B'_{r}.
\end{split}\end{align}
By Theorem \ref{3D} and Theorem \ref{4K}, we have, for any $r\in(0,1)$, $u\in C^{2+\alpha}(\bar{Q}_{r,\infty})$ and $u$ $C^{2+\alpha}_{\ast}$-converges to $0$ in $\bar{Q}_{r,\infty}$ at infinity. 
Hence, we obtain, for any $r\in(0,1)$, $f_{k}\in C^{\alpha}(\bar{Q}_{r,\infty})$ and $f_{k}(\cdot,t)$ 
converges to $0$ in $C^{\alpha}(\bar{G}_{r})$ as $t\rightarrow\infty$.  
We will prove, for any $r\in(0,1)$,
\begin{equation}\label{4bb}
\partial_{k}u(\cdot,t)\rightarrow 0\quad\text{in }L^{\infty}(G_{r})\quad\text{as }t\rightarrow\infty.
\end{equation}
Then, we can apply Theorem \ref{4K} to \eqref{4ax}-\eqref{4az} and conclude the desired decay of $\partial_{k}u$.

We now prove \eqref{4bb}. Take a cutoff function $\varphi=\varphi(x')\in C^{\infty}_{0}(B'_{3/4})$, with $\varphi=1$ in $B'_{1/2}$. Then,
\begin{equation}\label{4bc}
L(\varphi \partial_{k}u)=\varphi f_{k}+g_{k},
\end{equation}
where
$$
g_{k}=2x^{2}_{n}a_{ij}\partial_{i}\varphi\partial_{jk}u+(x_{n} a_{ij}\partial_{ij}\varphi+b_{i}\partial_{i}\varphi)x_{n}\partial_{k}u.
$$
Fix an $\varepsilon>0$. 
By $f_{k}(\cdot,t)$ converging to $0$ in $C^{\alpha}(\bar{G}_{3/4})$ as $t\rightarrow\infty$, there exists a constant $T>0$ such that
\begin{equation}\label{4bd}
\|f_{k}(\cdot,t)\|_{L^{\infty}(G_{3/4})}<\varepsilon\quad\text{for any }t>T.
\end{equation}
By assumptions and $u$ $C^{2+\alpha}_{\ast}$-converging to $0$ in $\bar{Q}_{3/4,\infty}$ at infinity, 
we can take $T$ larger if necessary such that
\begin{equation}\label{4be}
\|g_{k}(\cdot,t)\|_{L^{\infty}(G_{3/4})}<\varepsilon\quad\text{for any }t>T.
\end{equation}
We now examine $\varphi \partial_{k}u$ on $\partial_{p}Q^{T,\infty}_{3/4}$. 
First, $\varphi\partial_{k}u=0$ on $\partial B'_{3/4}\times (0,3/4)\times [T,\infty)$. 
Next, Theorem \ref{5A} (applied to the equation $Lu=f$ in $Q^{t-1,t}_{7/8}$) implies, for any $t>1$,
$$
\|\partial_{k}u\|_{L^{\infty}(B'_{3/4}\times \{3/4\}\times\{t\})}
\leq C\big\{\|u\|_{L^{\infty}(Q^{t-1,t}_{7/8})}+\|f\|_{C^{\alpha}_{\ast}(Q^{t-1,t}_{7/8})}\big\},
$$
where $C$ is a positive constant independent of $t$. 
By \eqref{4av} and $f(\cdot,t)$ converging to $0$ in $C^{\alpha}(\bar{G}_{7/8})$ as $t\rightarrow\infty$, 
we can take $T$ larger if necessary such that, for any $t>T$,
$$
\|\varphi\partial_{k}u\|_{L^{\infty}(B'_{3/4}\times \{3/4\}\times\{t\})}<\varepsilon.
$$
By \eqref{4ba}, $c(\cdot,t),f_{k}(\cdot,t)$ converging to $\bar{c},0$ in $C^{\alpha}(\bar{G}_{3/4})$ as $t\rightarrow\infty$, 
respectively, and Lemma \ref{4A}(i), we can take $T$ large further if necessary such that, for any $t>T$,
$$
\|(\varphi\partial_{k}h)(\cdot,0,t)\|_{L^{\infty}(B'_{3/4})}<\varepsilon.
$$
Fix such a $T>0$. By Theorem \ref{3F}, we have
$$
\|\varphi\partial_{k}u\|_{L^{\infty}(Q_{3/4,T})}\leq C_{T}.
$$
We can proceed similarly as in the proof of Theorem \ref{4C}. 
Then, there exists a constant $T'>T$ such that, for any $t>T'$,
$$
\|(\varphi\partial_{k}u)(\cdot,t)\|_{L^{\infty}(G_{3/4})}<C\varepsilon,
$$
where $C$ is a positive constant independent of $\varepsilon$. Hence, for any $t>T'$,
$$
\|\partial_{k}u(\cdot,t)\|_{L^{\infty}(G_{1/2})}<C\varepsilon,
$$
and we obtain \eqref{4bb} easily.
\end{proof}

By induction, We have the following more general result.

\begin{thm}\label{4M}
For some integer $\ell\geq 0$ and some constant $\alpha\in(0,1)$, assume that 
$D_{x'}^{\sigma}a_{ij},D_{x'}^{\sigma}b_{i},D_{x'}^{\sigma}c,D^{\sigma}_{x'}f\in C^{\alpha}(\bar{Q}_{1,\infty})$ 
and $D_{x'}^{\sigma}\bar{a}_{ij},D_{x'}^{\sigma}\bar{b}_{i},D_{x'}^{\sigma}\bar{c}\in C^{\alpha}(\bar{G}_{1})$, 
for any $\sigma\leq \ell$, with \eqref{4am}, \eqref{4au}, and $Q(\alpha)<0$ in $\bar{G}_{1}$,
and that $D_{x'}^{\sigma}a_{ij}(\cdot,t),D_{x'}^{\sigma}b_{i}(\cdot,t),D_{x'}^{\sigma}c(\cdot,t),D_{x'}^{\sigma}f(\cdot,t)$ converge to 
$D_{x'}^{\sigma}\bar{a}_{ij},D_{x'}^{\sigma}\bar{b}_{i},D_{x'}^{\sigma}\bar{c},0$
 in $C^{\alpha}(\bar{G}_{1})$ as $t\rightarrow\infty$, for any $\sigma\leq \ell$, respectively. 
 Suppose that, for some $\phi$ with $D^{\sigma}_{x'}\phi\in C^{2+\alpha}(\bar{G}_{1})$ for any $\sigma\leq \ell$, 
 the problem \eqref{4ao}-\eqref{4aq} admits a solution $u\in C^{2}(Q_{1,\infty})\cap C(\bar{Q}_{1,\infty})$ with \eqref{4av}. 
 Then, for any $\sigma\leq \ell$ and any $r\in(0,1)$, $D^{\sigma}_{x'}u$ $C^{2+\alpha}_{\ast}$-converges 
 to $0$ in $\bar{Q}_{r,\infty}$ at infinity.
\end{thm}

Now, we consider $C^{2+\alpha}_{\ast}$-decay of the $x_{n}$-derivative.

\begin{thm}\label{4N}
For some constant $\alpha\in(0,1)$, assume that $a_{ij},b_{i},c,f\in C^{1,\alpha}(\bar{Q}_{1,\infty})$ 
and $\bar{a}_{ij},\bar{b}_{i},\bar{c}\in C^{1,\alpha}(\bar{G}_{1})$, with \eqref{4am}, \eqref{4au}, 
and $Q(1+\alpha)<0$ in $\bar{G}_{1}$, and that $a_{ij}(\cdot,t),b_{i}(\cdot,t),c(\cdot,t),f(\cdot,t)$ 
converge to $\bar{a}_{ij},\bar{b}_{i},\bar{c},0$ in $C^{1,\alpha}(\bar{G}_{1})$ as $t\rightarrow\infty$, respectively. 
Suppose that, for some $\phi\in C^{1,2+\alpha}(\bar{G}_{1})$, the problem \eqref{4ao}-\eqref{4aq} 
admits a solution $u\in C^{2}(Q_{1,\infty})\cap C(\bar{Q}_{1,\infty})$ with \eqref{4av}. 
Then, for any $r\in(0,1)$, $u$ $C^{1,2+\alpha}_{\ast}$-converges to $0$ in $\bar{Q}_{r,\infty}$ at infinity.
\end{thm}

\begin{proof}
As in the proof of Theorem \ref{3I}, we have, 
for any $r\in(0,1)$, $\partial_{n}u\in C^{2}(Q_{r,\infty})\cap C(\bar{Q}_{r,\infty})$ satisfies
\begin{align}
L^{(1)}(\partial_{n}u)&=f_{n}\quad\text{in }Q_{r,\infty},\label{4bf}\\
\partial_{n}u(\cdot,0)&=\partial_{n}\phi\quad\text{on } G_{r},\label{4bg}\\
\partial_{n}u&=u_{1}\quad\text{on }  S_{r,\infty},\label{4bh}
\end{align}
where
\begin{align*}
L^{(1)}&=x^{2}_{n}a_{ij}\partial_{ij}+x_{n}(b_{i}+2a_{in})\partial_{i}+(c+b_{n})-\partial_{t},\\
f_{n}&=\partial_{n}f-x^{2}_{n}\partial_{n}a_{ij}\partial_{ij}u-x_{n}\partial_{n}b_{i}\partial_{i}u-\partial_{n}c u\\
&\qquad-2x_{n} a_{i\alpha}\partial_{i\alpha}u-b_{\alpha}\partial_{\alpha}u,
\end{align*}
with
$$
f_{n}=\partial_{n}f-\partial_{n}c h-b_{\alpha}\partial_{\alpha}h\quad \text{on } S_{r.\infty},
$$
and
\begin{align}\label{4bi}\begin{split}
\partial_{t}u_{1}-(c+b_{n})u_{1}+f_{n}&=0\quad\text{on } S_{r,\infty},\\
u_{1}(\cdot,0,0)&=\partial_{n}\phi(\cdot,0)\quad\text{on }  B'_{r}.
\end{split}
\end{align}
By Theorem \ref{3G} and Theorem \ref{4L}, we have, for any $r\in(0,1)$, $u,D_{x'}u\in C^{2+\alpha}(\bar{Q}_{r,\infty})$ 
and they both $C^{2+\alpha}_{\ast}$-converge to $0$ in $\bar{Q}_{r,\infty}$ at infinity. 
Hence, we obtain, for any $r\in(0,1)$, $f_{n}\in C^{\alpha}(\bar{Q}_{r,\infty})$ 
and $f_{n}(\cdot,t)$ converges to $0$ in $C^{\alpha}(\bar{G}_{r})$ as $t\rightarrow\infty$. 
Note that the limit elliptic operator of $L^{(1)}$ is given by
$$
L^{(1)}_{0}=x^{2}_{n}\bar{a}_{ij}\partial_{ij}+x_{n}(\bar{b}_{i}+2\bar{a}_{in})\partial_{i}+(\bar{c}+\bar{b}_{n}).
$$
We denote by $Q^{(1)}$ the characteristic polynomial of $L^{(1)}_0$. Then, $Q^{(1)}(\mu)=Q(\mu+1)$. Hence, 
$$
\bar{b}_{n}+\bar{c}=Q^{(1)}(0)=Q(1)<0,\quad Q^{(1)}(\alpha)=Q(1+\alpha)<0\quad\text{in }\bar{G}_{1}.
$$
By a similar argument as in the proof of Theorem \ref{4L}, we get, for any $r\in(0,1)$,
\begin{equation}\label{4bj}
\partial_{n}u(\cdot,t)\rightarrow 0\quad\text{in }L^{\infty}(G_{r})\quad\text{as }t\rightarrow\infty.
\end{equation}
Then, we can apply Theorem \ref{4K} to \eqref{4bf}-\eqref{4bh} and conclude the desired decay of $\partial_{n}u$.
\end{proof}

By induction and combining Theorem \ref{4M} and Theorem \ref{4N}, we have the following more general result.

\begin{thm}\label{4O}
For some integers $\ell\geq m\geq  0$ and some constant $\alpha\in(0,1)$, 
assume that $D_{x}^{\sigma}a_{ij},D_{x}^{\sigma}b_{i},D_{x}^{\sigma}c,,D_{x}^{\sigma}f\in C^{\alpha}(\bar{Q}_{1,\infty})$ 
for any $\sigma\leq\ell$ and $\bar{a}_{ij},\bar{b}_{i},\bar{c}\in C^{\ell,\alpha}(\bar{G}_{1})$, 
with \eqref{4am}, \eqref{4au}, and $Q(m+\alpha)<0$ in $\bar{G}_{1}$, 
and that $a_{ij}(\cdot,t),b_{i}(\cdot,t),c(\cdot,t),f(\cdot,t)$ converge to 
$\bar{a}_{ij},\bar{b}_{i},\bar{c},0$ in $C^{\ell,\alpha}(\bar{G}_{1})$ as $t\rightarrow\infty$, respectively. 
Suppose that, for some $\phi\in C^{\ell,2+\alpha}(\bar{G}_{1})$, the problem \eqref{4ao}-\eqref{4aq} 
admits a solution $u\in C^{2}(Q_{1,\infty})\cap C(\bar{Q}_{1,\infty})$ with \eqref{4av}. 
Then, for any nonnegative integers $\nu$ and $\tau$ with $\nu\leq m$ and $\tau+\nu\leq \ell$, 
and any $r\in(0,1)$, $\partial^{\nu}_{n}D^{\tau}_{x'}u$ $C^{2+\alpha}_{\ast}$-converges to $0$ in $\bar{Q}_{r,\infty}$ at infinity.
\end{thm}

Finally, we consider the $C^{\alpha}_{\ast}$-decay of the $t$-derivative. 
By proceeding similarly as in Subsection \ref{time} and replacing Theorem \ref{3J} with Theorem \ref{4O}, 
we have the following result.

\begin{thm}\label{4P}
For some integers $\ell\geq m\geq  0$ and some constant $\alpha\in(0,1)$, 
assume that $a_{ij},b_{i},c,f\in C^{\ell,\alpha}(\bar{Q}_{1,\infty})$ 
and $\bar{a}_{ij},\bar{b}_{i},\bar{c}\in C^{\ell,\alpha}(\bar{G}_{1})$, 
with \eqref{4am}, \eqref{4au}, and $Q(m+\alpha)<0$ in $\bar{G}_{1}$, 
and that $a_{ij}(\cdot,t),b_{i}(\cdot,t),c(\cdot,t),f(\cdot,t)$ converge to $\bar{a}_{ij},\bar{b}_{i},\bar{c},0$ 
in $C^{\ell,\alpha}(\bar{G}_{1})$ as $t\rightarrow\infty$, respectively. 
Suppose that, for some $\phi\in C^{\ell,2+\alpha}(\bar{G}_{1})$, 
the problem \eqref{4ao}-\eqref{4aq} admits a solution $u\in C^{2}(Q_{1,\infty})\cap C(\bar{Q}_{1,\infty})$ with \eqref{4av}. 
Then, for any for nonnegative integers $\nu$, $\tau$, and $\sigma$, 
with $\sigma\geq 2$, $\nu+2(\sigma-1)\leq m$, and $2\sigma+\tau+\nu\leq \ell+2$, 
and any $r\in(0,1)$, $\partial^{\sigma}_{t}\partial^{\nu}_{n}D^{\tau}_{x'}u$ $C^{\alpha}_{\ast}$-converge to $0$ in $\bar{Q}_{r,\infty}$ at infinity.
\end{thm}

By combining Theorem \ref{4O} and Theorem \ref{4P}, we can obtain the following $C^{k,2+\alpha}_{\ast}$-convergence result.

\begin{corollary}\label{4Q}
For some integer $k\geq  0$ and some constant $\alpha\in(0,1)$, 
assume that $a_{ij},b_{i},c,f\in C^{k,\alpha}(\bar{Q}_{1,\infty})$ 
and $\bar{a}_{ij},\bar{b}_{i},\bar{c},\bar{f}\in C^{k,\alpha}(\bar{G}_{1})$, 
with \eqref{4am}, \eqref{4au}, and $Q(k+\alpha)<0$ in $\bar{G}_{1}$, 
and that $a_{ij}(\cdot,t),b_{i}(\cdot,t),c(\cdot,t),f(\cdot,t)$ converge to 
$\bar{a}_{ij},\bar{b}_{i},\bar{c},\bar{f}$ in $C^{k,\alpha}(\bar{G}_{1})$ as $t\rightarrow\infty$, respectively. 
Suppose that, for some $\phi\in C^{k,2+\alpha}(\bar{G}_{1})$, 
the problem \eqref{4ao}-\eqref{4aq} admits a solution $u\in C^{2}(Q_{1,\infty})\cap C(\bar{Q}_{1,\infty})$, 
with $u(\cdot,t)$ converging to a solution $v\in C^{2}(G_{1})\cap C(\bar{G}_{1})$
of the problem \eqref{4as}-\eqref{4at} in $L^{\infty}(G_{1})$ as $t\rightarrow\infty$. 
Then, for any $r\in(0,1)$, $u$ $C^{k,2+\alpha}_{\ast}$-converges to $v$ in $\bar{Q}_{r,\infty}$ at infinity.
\end{corollary}

\begin{proof}
Let $w=u-v$ in $Q_{1,\infty}$. Then, $w$ satisfies 
\begin{align}
Lw&=f_{0}\quad\text{in }Q_{1,\infty},\label{4bk}\\
w(\cdot,0)&=\phi_{0}=\phi-v\quad\text{on }G_{1},\label{4bl}\\
w&=h_{0}=h-\bar{f}/\bar{c}\quad\text{on }S_{1,\infty},\label{4bm}
\end{align}
where
\begin{align*}
f_{0}&=Lu-L_{0}v+(L_{0}-L)v\\
&=f-\bar{f}-x^{2}_{n}(a_{ij}-\bar{a}_{ij})v_{ij}-x_{n}(b_{i}-\bar{b}_{i})v_{i}-(c-\bar{c})v.
\end{align*}
By applying the local version of Theorem \ref{1D'} (Refer to Theorem 6.2 in \cite{HanXie2024}.) to \eqref{4as}-\eqref{4at},
we have, for any $r\in(0,1)$,
\begin{equation}\label{4bn}
v\in C^{k,2+\alpha}(\bar{G}_{r}).
\end{equation}
Then, we get
$$
f_{0}= f-\bar{f}-(c-\bar{c})\frac{\bar{f}}{\bar{c}}=f-\frac{c\bar{f}}{\bar{c}}\quad\text{on }S_{1,\infty},
$$
and, by \eqref{4ar},
$$
\partial_{t}h_{0}-ch_{0}+f_{0}=\partial_{t}h-c\left(h-\frac{\bar{f}}{\bar{c}}\right)+f_{0}=-f+\frac{c\bar{f}}{\bar{c}}+f_{0}=0\quad\text{on }S_{1,\infty}.
$$
Hence, 
$$
h_{0}(x,t)=\phi_{0}(x)\exp\Big\{\int^{t}_{0}c(x,s)ds\Big\}-\int^{t}_{0}\exp\Big\{\int^{t}_{s}c(x,\tau)d\tau\Big\}f_{0}(x,s) d s.
$$
By \eqref{4bn} and the assumption that $a_{ij}(\cdot,t),b_{i}(\cdot,t),c(\cdot,t),f(\cdot,t)$ 
converge to $\bar{a}_{ij},\bar{b}_{i},\bar{c},\bar{f}$ in $C^{k,\alpha}(\bar{G}_{1})$ as $t\rightarrow\infty$, respectively, 
we have, for any $r\in(0,1)$, $f_{0}(\cdot,t)$ converges to $0$ in $C^{k,\alpha}(\bar{G}_{r})$ as $t\rightarrow\infty$. 
By the assumption, we have
$$
w(\cdot,t)\rightarrow 0\quad\text{in }L^{\infty}(G_{1})\quad\text{as }t\rightarrow\infty.
$$
Hence, we can apply Theorem \ref{4O} and Theorem \ref{4P} to \eqref{4bk}-\eqref{4bm} and conclude the desired result.
\end{proof}

We are ready to prove Theorem \ref{1E}.

\begin{proof}[Proof of Theorem \ref{1E}.]
The existence and uniqueness of $v$ can be obtained from Theorem \ref{1C'}. 
By Corollary \ref{4D}, we have $u(\cdot,t)$ converges to $v$ in $L^{\infty}(\Omega)$ as $t\rightarrow\infty$. 
Then, by Theorem \ref{5A} applied to the equations satisfied by $u-v$ and its derivatives, 
we can obtain that $u$ $C^{k+2,\alpha}_{\ast}$-converges to $v$ 
in $\bar{\Omega}'\times[0,\infty)$ at infinity, for any domain $\Omega'\subset\subset\Omega$. 
Moreover, we can apply Corollary \ref{4Q} to conclude a $C^{k,2+\alpha}_*$-convergence near the lateral boundary $S$. Recall the weighted H\"older function spaces defined in Subsection \ref{Notations}. 
Let $x_0$ be a point on $\partial D$ and $r>0$ be a constant. 
We can define local weighted H\"older function spaces in $[D\cap B_r(x_0)]\times (0,T]$ 
by replacing $Q_T$ and $S_T$ in the definition of Subsection \ref{Notations} 
with $[D\cap B_r(x_0)]\times (0,T]$ and $[\partial D\cap B_r(x_0)]\times [0,T]$. 
Then, we can apply Corollary \ref{4Q} to conclude that $u$ $C^{k,2+\alpha}_{\ast}$-converges 
to $v$ in $\overline{[\Omega \cap B_r(x_0)]}\times[0,\infty)$ at infinity for some $x_0\in\partial\Omega$ and small $r>0$.
Hence, a standard covering argument yields Theorem \ref{1E}.
\end{proof}


\section{Appendix: Intermediate Schauder Theory}\label{sec-Intermediate-Schauder}

In this section, we discuss some estimates in the intermediate Schauder theory for uniformly parabolic equations.
We set, for any $X=(x,t)\in\mathbb{R}^{n}\times\mathbb{R}$ and $r>0$,
$$
\begin{aligned}
Q(X,r)&=\{Y=(y,s)\in\mathbb{R}^{n+1}:|x-y|<r,t-r^{2}< s\leq  t\},\\
B_r(x)&=\{y\in\mathbb{R}^{n}:|x-y|<r\}.
\end{aligned}
$$
The following interior regularity is proved in \cite{B1969,K198081,L1992}.

\begin{thm}\label{5A}
For some point $X_{0}=(x_{0},t_{0})\in\mathbb{R}^{n+1}$ and some constants $r>0$ and $\alpha\in(0,1)$, 
let $Q=Q(X_{0},r)$ and assume $a_{i j},b_{i},c,f\in C^{\alpha}_{\ast}(\bar{Q})$, 
with $a_{i j}=a_{j i}$ and, for any $x\in\overline{B_{r}(x_{0})}$, almost all $t\in(t_{0}-r^{2},t_{0})$, and any $\xi\in\mathbb{R}^{n}$, 
$$
\lambda|\xi|^{2}\leq a_{ij}(x,t)\xi_{i}\xi_{j}\leq \Lambda|\xi|^{2},
$$
for some positive constants $\lambda$ and $\Lambda$. Suppose that $u:Q\rightarrow\mathbb{R}$ satisfies
$$
a_{ij}\partial_{ij}u+b_{i}\partial_{i}u+cu-\partial_{t}u=f\quad a.e.\text{ in }Q,
$$
with $D_xu,D^2_x u,\partial_t u\in L^{\infty}(Q)$. Then, for any $Q'=Q(X',r')\subset\subset Q$, $u\in C^{2,\alpha}_{\ast}(\bar{Q}')$, and
\begin{align*}
\|u\|_{C^{2,\alpha}_{\ast}(\bar{Q}')}\leq C\big\{\|u\|_{L^{\infty}(Q)}+\|f\|_{C^{\alpha}_{\ast}(\bar{Q})}\big\},
\end{align*}
where $C$ is a positive constant depending only on $n$, $\alpha$, $\lambda$, $\mathrm{dist}(Q',\partial_{p}Q)$, 
and the $C^{\alpha}_{\ast}$-norms of $a_{ij}$, $b_{i}$, and $c$ in $\bar{Q}$.
\end{thm}

For the purpose of applications, we extend the above result to the bottom of the parabolic boundary.
We set, for any $R,T>0$, 
$$B_R=B_{R}(0)\quad\text{and}\quad Q_{T}(R)=B_{R}\times (0,T].$$

\begin{corollary}\label{5G}
For some positive constants $R$, $T$, and $\alpha\in(0,1)$, assume $a_{i j},b_{i},c,f\in C^{\alpha}_{\ast}(\overline{Q_{T}(R)})$, 
with $a_{i j}=a_{j i}$ and,  for any $x\in\bar{B}_R$, almost all $t\in(0,T)$, and any $\xi\in\mathbb{R}^{n}$, 
$$
\lambda|\xi|^{2}\leq a_{ij}(x,t)\xi_{i}\xi_{j}\leq \Lambda|\xi|^{2},
$$
for some positive constants $\lambda$ and $\Lambda$. Suppose that $u:Q_{T}(R)\rightarrow\mathbb{R}$ satisfies
$$
a_{ij}\partial_{ij}u+b_{i}\partial_{i}u+cu-\partial_{t}u=f\quad a.e.\text{ in }Q_{T}(R),
$$
with $D_xu,D^2_x u,\partial_t u\in L^{\infty}(Q_{T}(R))$ and $u(\cdot,0)=\phi\in C^{2,\alpha}(\bar{B}_{R})$. 
Then, for any $r\in(0,R)$, $u\in C^{2,\alpha}_{\ast}(\overline{{Q}_{T}(r)})$, and
\begin{align*}
\|u\|_{C^{2,\alpha}_{\ast}(\overline{Q_{T}(r)})}
\leq C\big\{\|u\|_{L^{\infty}(Q_{T}(R))}+\|\phi\|_{C^{2.\alpha}(\bar{B}_{R})}+\|f\|_{C^{\alpha}_{\ast}(\overline{Q_{T}(R)})}\big\},
\end{align*}
where $C$ is a positive constant depending only on $R-r$, $T$, $n$, $\alpha$, $\lambda$, 
and the $C^{\alpha}_{\ast}$-norms of $a_{ij}$, $b_{i}$, and $c$ in $\overline{Q_{T}(R)}$.
\end{corollary}


\begin{proof}
Without loss of generality, we assume $\phi=0$, since we can consider the equation for $u-\phi$. 
We extend $u$ to a function in $Q^{\infty}_T(R)=B_R\times(-\infty,T]$ by defining
$$
\Tilde{u}=
\begin{cases}
u&\quad\text{in }Q_T(R),\\
0&\quad\text{in }B_R\times(-\infty,0].
\end{cases}
$$
Similarly, we define $\Tilde{b}_i$, $\Tilde{c}$, and $\Tilde{f}$. In addition, we define
$$
A_{ij}=
\begin{cases}
a_{ij}&\quad\text{in }Q_T(R),\\
\delta_{ij}&\quad\text{in }B_R\times(-\infty,0].
\end{cases}
$$
Then, we can verify that $\Tilde{u}:Q^{\infty}_{T}(R)\rightarrow\mathbb{R}$ satisfies
$$
A_{ij}\partial_{ij}\Tilde{u}+\Tilde{b}_{i}\partial_{i}\Tilde{u}+\Tilde{c}\Tilde{u}-\partial_{t}\Tilde{u}=\Tilde{f}\quad a.e.\text{ in }Q^{\infty}_{T}(R),
$$
with $D_x\Tilde{u},D^2_x \Tilde{u},\partial_t \Tilde{u}\in L^{\infty}(Q^{\infty}_{T}(R))$. 
In fact, we can prove that
\begin{equation}\label{weak1}
\partial_{ij}\Tilde{u}=
\begin{cases}
\partial_{ij}u&\quad\text{in }Q_T(R),\\
0&\quad\text{in }B_R\times(-\infty,0],
\end{cases}
\end{equation}
and
\begin{equation}\label{weak2}
\partial_{t}\Tilde{u}=
\begin{cases}
\partial_{t}u&\quad\text{in }Q_T(R),\\
0&\quad\text{in }B_R\times(-\infty,0].
\end{cases}
\end{equation}
To verify \eqref{weak1} and \eqref{weak2}, 
we take any $\zeta\in C^{\infty}(Q^{\infty}_{T}(R))$ with compact support in $B_R\times(-\infty,T)$. 
For any $\epsilon>0$ small, we take $\eta=\eta(t)\in C^{\infty}(\mathbb{R})$ such that
$$
\eta(t)=0\quad\text{for }|t|<\epsilon,\quad \eta(t)=1\quad\text{for }|t|>2\epsilon,\quad 0\leq\eta\leq 1,\quad|\eta'|\leq\frac{2}{\epsilon}.
$$
Note that $\zeta\eta|_{Q_{T}(R)}$ has compact support in $B_R\times(0,T)$. Then, we have
\begin{align*}
\int_{Q^{\infty}_{T}(R)}\Tilde{u}\partial_{ij}\zeta \eta dX=\int_{Q_{T}(R)}u\partial_{ij}(\zeta \eta) dX=\int_{Q_{T}(R)}\partial_{ij} u\zeta \eta dX.
\end{align*}
Here, the first equality holds since $\Tilde{u}=0$ in $B_R\times(-\infty,0]$ and 
$\eta$ is a function of $t$ only. 
By letting $\epsilon\rightarrow 0$, we obtain \eqref{weak1}. Next, we have
\begin{equation}\label{weak3}
\begin{aligned}
\int_{Q^{\infty}_{T}(R)}\Tilde{u}\partial_{t}\zeta \eta dX&=\int_{Q_{T}(R)}u\partial_{t}(\zeta \eta )dX-\int_{Q_{T}(R)}u\zeta\eta' dX\\
&=-\int_{Q_{T}(R)}\partial_{t} u\zeta \eta dX-\int_{Q_{T}(R)}u\zeta\eta' dX.
\end{aligned}
\end{equation}
Note that $u\in C(\overline{Q_T(R)})$ by the Sobolev embedding theorem and $u|_{t=0}=0$. 
We get
$$
\Big|\int_{Q_{T}(R)}u\zeta\eta' dX\Big|\leq \frac{2}{\epsilon}\int_{B_{R}\times(\epsilon,2\epsilon)}|u||\zeta|dX
\leq C\max_{B_{R}\times(\epsilon,2\epsilon)}|u|\rightarrow 0,
$$
as $\epsilon\rightarrow 0$. By letting $\epsilon\rightarrow 0$ in \eqref{weak3}, we obtain \eqref{weak2}.

Hence, we can apply Theorem \ref{5A} to $\Tilde{u}$ and obtain the desired results.
\end{proof}

\begin{remark}\label{5H}
In general, we cannot extend the above result to the global Schauder estimate for the initial-boundary value problem. 
Refer to Lemma 16.8 and the subsequent paragraph in \cite{L1992} for details.
\end{remark}


\begin{thebibliography}{DG}


\bibitem{ALZ2020} P. Allmann, L. Lin, J. Zhu,
\emph{Modified mean curvature flow of entire locally {L}ipschitz
radial graphs in hyperbolic space},
Math. Nachr., 293(2020), 861-878.


\bibitem{B1969} A. Brandt,
\emph{Interior {S}chauder estimates for parabolic differential- (or
difference-) equations via the maximum principle},
Israel J. Math., 7(1969), 254-262.

\bibitem{DR2024} H. Dong, J. Ryu,
\emph{Sobolev estimates for parabolic and elliptic equations in divergence form with degenerate coefficients}, arXiv:2412.00779.




\bibitem{EGG2012} H. Emamirad, G. R. Goldstein, J. A. Goldstein,
\emph{Chaotic solution for the {B}lack-{S}choles equation}, Proc. Amer. Math. Soc., 140(2012), 2043-2052.

\bibitem{EGG2014} H. Emamirad, G. R. Goldstein, J. A. Goldstein,
\emph{Corrigendum and improvement to ``{C}haotic solution for the
{B}lack-{S}choles equation'' [MR2888192]},
Proc. Amer. Math. Soc., 142(2014), 4385-4386.

\bibitem{F1964} A. Friedman,
\emph{Partial differential equations of parabolic type},
Prentice-Hall, Inc., Englewood Cliffs, NJ, 1964.


\bibitem{HanXie2024} Q. Han, J. Xie,
\emph{Optimal boundary regularity for uniformly degenerate elliptic equations}, arXiv:2411.16418.


\bibitem{JT2006} S. Janson, J. Tysk,
\emph{Feynman-{K}ac formulas for {B}lack-{S}choles-type operators},
Bull. London Math. Soc., 38(2006), 269-282.


\bibitem{K2006} K.-H. Kim,
\emph{Parabolic {SPDE}s degenerating on the boundary of non-smooth
domain}, Electron. J. Probab., 11(2006), no. 23, 563-584.


\bibitem{K2007} K.-H. Kim,
\emph{Sobolev space theory of parabolic equations degenerating on
the boundary of {$C^1$} domains}, Comm. Partial Differential Equations, 32(2007), 1261-1280.

\bibitem{K2008} K.-H. Kim,
\emph{{$L_p$}-theory of parabolic {SPDE}s degenerating on the
boundary of {$C^1$} domains}, J. Theoret. Probab., 21(2008),169-192.

\bibitem{K198081} B. F. Knerr,
\emph{Parabolic interior {S}chauder estimates by the maximum
principle},
Arch. Rational Mech. Anal., 75(1980/81), 51-58.


\bibitem{L1968}O. A. Lady\v{z}enskaja, V. A. Solonnikov, N. N. Ural'ceva,
\emph{Linear and quasilinear equations of parabolic type},
Translations of Mathematical Monographs, Vol. 23, American Mathematical Society, Providence, RI, 1968.


\bibitem{LiTam1991}
P. Li, L.-F. Tam, \emph{The heat equation and harmonic maps of complete manifolds},
Invent. Math., 105(1991), 1-46.

\bibitem{LiTam1993}
P. Li, L.-F. Tam, \emph{Uniqueness and regularity of proper harmonic maps},
Ann. of Math., 137(1993), 167-201.

\bibitem{LiTam1993Ind}
P. Li, L.-F. Tam, \emph{Uniqueness and regularity of proper harmonic maps. II},
Indiana Univ. Math. J., 42(1993), 591-635.


\bibitem{L1992} G. M. Lieberman,
\emph{Intermediate {S}chauder theory for second order parabolic
equations. {IV}. {T}ime irregularity and regularity},
Differential Integral Equations, 5(1992), 1219-1236.

\bibitem{L1996} G. M. Lieberman,
\emph{Second order parabolic differential equations},
World Scientific Publishing Co., Inc., River Edge, NJ, 1996.


\bibitem{LX2012} L. Lin, L. Xiao,
\emph{Modified mean curvature flow of star-shaped hypersurfaces in
hyperbolic space},
Comm. Anal. Geom., 20(2012), 1061-1096.

\bibitem{L2000} L. Lorenzi,
\emph{Optimal {S}chauder estimates for parabolic problems with data
measurable with respect to time},
SIAM J. Math. Anal., 32(2000), 588-615.

\bibitem{L2001} S. V. Lototsky,
\emph{Linear stochastic parabolic equations, degenerating on the boundary of a domain}, Electron. J. Probab., 6(2001), no. 24, 14.


\bibitem{U2003} P. Unterberger,
\emph{Evolution of radial graphs in hyperbolic space by their mean
curvature}, Comm. Anal. Geom., 11(2003), 675-695.

\end{thebibliography}
\end{document}